\documentclass[a4paper]{article}
\usepackage{mathrsfs,xcolor}
\usepackage{amsmath}
\usepackage{amssymb}
\usepackage{amsthm}
\usepackage{graphicx}
\usepackage{fancyhdr}
\usepackage{appendix}
\usepackage{enumerate}
\usepackage{soul}

\usepackage{cases}

\usepackage[margin=1in]{geometry}
\usepackage{comment}

\usepackage{cite}

\usepackage{amscd, amsthm}
\usepackage{mathtools}
\usepackage{mathrsfs}
\usepackage{bm}

\catcode`\@=11 \@addtoreset{equation}{section}

\catcode`\@=12
\usepackage{color}

\usepackage{framed}

\usepackage[colorlinks,
            linkcolor=blue,
            anchorcolor=blue,
            citecolor=red,
            ]{hyperref}
\newtheorem{Theorem}{Theorem}[section]
\newtheorem{TheoremPrime}{Theorem}

\newtheorem{Proposition}[Theorem]{Proposition}
\newtheorem{Lemma}[Theorem]{Lemma}
\newtheorem{Corollary}[Theorem]{Corollary}
\newtheorem{Definition}[Theorem]{Definition}

\newtheorem{Remark}[Theorem]{Remark}

\newcommand{\newcom}{\newcommand}
\newcommand{\bTheorem}[1]{
\begin{Theorem} \label{T#1} }
\newcommand{\eT}{\end{Theorem}}

\newcommand{\bProposition}[1]{
\begin{Proposition} \label{P#1}}
\newcommand{\eP}{\end{Proposition}}

\newcommand{\bLemma}[1]{
\begin{Lemma} \label{L#1} }
\newcommand{\eL}{\end{Lemma}}

\newcommand{\bCorollary}[1]{
\begin{Corollary} \label{C#1} }
\newcommand{\eC}{\end{Corollary}}

\newcommand{\beq}{\begin{equation}}
\newcommand{\eeq}{\end{equation}}
\newcom{\ben}{\begin{eqnarray}}
\newcom{\een}{\end{eqnarray}}
\newcom{\beno}{\begin{eqnarray*}}
\newcom{\eeno}{\end{eqnarray*}}
\newcom{\bali}{\begin{aligned}}
\newcom{\eali}{\end{aligned}}
\usepackage{mdframed}
\newcommand{\bFormula}[1]{
\begin{equation} \label{#1}}
\newcommand{\eF}{\end{equation}}

\newcommand{\f}{\frac}
\newcommand{\Om}{\Omega}

\newcommand{\p}{\partial}
\newcommand{\ess}{\rm ess}
\newcommand{\res}{\rm res}
\newcommand{\Ov}[1]{\overline{#1}}

\newcommand{\vr}{\varrho}
\newcommand{\vre}{\vr_\ep}

\newcommand{\vu}{\mathbf{u}}
\newcommand{\vU}{\mathbf{U}}
\newcommand{\vm}{\vc{v}_m}
\newcommand{\vv}{\mathbf{v} }
\newcommand{\vV}{  \mathbf{V}  }
\newcommand{\vW}{\mathbf{W}    }
\newcommand{\vw}{\mathbf{w} }

\newcommand{\ka}{\kappa}

\newcommand{\bd}{ \mathbb{D}  }
\newcommand{\ba}{ \mathbb{A}  }
\newcommand{\vc}[1]{{\boldsymbol #1}}
\newcommand{\Div}{{\rm div}}
\newcommand{\Grad}{\nabla}

\newcommand{\dx}{\,{\rm d} x}

\newcommand{\dt}{\,{\rm d} t }
\newcommand{\ds}{\,{\rm d} s}

\newcommand{\dxdt}{\dx\dt}

\newcommand{\ep}{\varepsilon}
\newcommand{\ed}{{\varepsilon,\delta}}
\newcommand{\de}{\delta}

\font\F=msbm10 scaled 1000
\newcommand{\R}{\mbox{\F R}}

\newcommand{\lr}[1]{\left( #1 \right)}
\newcommand{\eq}[1]{\begin{equation}
\begin{split}
#1
\end{split}
\end{equation}}
\newcommand{\eqh}[1]{\begin{equation*}
\begin{split}
#1
\end{split}
\end{equation*}}

\usepackage{color,colortbl,xcolor}
\definecolor{darkgreen}{rgb}{0,0.5,0}

\newcommand\Cbox[2]{%
    \newbox\contentbox%
    \newbox\bkgdbox%
    \setbox\contentbox\hbox to \hsize{%
        \vtop{
            \kern\columnsep
            \hbox to \hsize{%
                \kern\columnsep%
                \advance\hsize by -2\columnsep%
                \setlength{\textwidth}{\hsize}%
                \vbox{
                    \parskip=\baselineskip
                    \parindent=0bp
                    #2
                }%
                \kern\columnsep%
            }%
            \kern\columnsep%
        }%
    }%
    \setbox\bkgdbox\vbox{
        \color{#1}
        \hrule width  \wd\contentbox %
               height \ht\contentbox %
               depth  \dp\contentbox
        \color{black}
    }%
    \wd\bkgdbox=0bp%
    \vbox{\hbox to \hsize{\box\bkgdbox\box\contentbox}}%
    \vskip\baselineskip%
}

\usepackage{comment}

\newcommand{\tsl}{\textsl}

\newcommand{\mbb}{\mathbb}

\newcommand{\mc}{\mathcal}

\newcommand{\veps}{\varepsilon}

\newcommand{\wtilde}{\widetilde}

\newcommand{\oline}{\overline}

\newcommand{\g}{\gamma}
\renewcommand{\k}{\kappa}

\renewcommand{\t}{\tau}

\renewcommand{\vm}{\mathbf{m}}

\newcommand{\N}{\mathbb{N}}

\newcommand{\T}{\mathbb{T}}

\newcommand{\D}{\mathbb{D}}

\renewcommand{\div}{{\rm div}\,}

\allowdisplaybreaks

\begin{document}


\pagestyle{fancy} \lhead{\color{blue}{Anelastic approximation for degenerate Navier--Stokes system}}  \rhead{\emph{N.Chaudhuri et al.}}

\title{\bf Anelastic approximation for the degenerate compressible Navier--Stokes equations revisited }

\author{Nilasis Chaudhuri$^\dag$
\and
Francesco Fanelli$ ^* $
\and
Yang Li$ ^\ddag $
\and
Ewelina Zatorska$^\S$
}

\date{\today}

\maketitle

{
\footnotesize
\centerline{$^\dag\;$Institute of Applied Mathematics and Mechanics, University of Warsaw,}
\centerline{ul. Banacha 2 -- 02-097 Warszawa, Poland}

\smallbreak
\centerline{\small \texttt{nchaudhuri@mimuw.edu.pl}}

\bigbreak
\centerline{$^{*,a}\;$BCAM -- Basque Center for Applied Mathematics}
\centerline{Alameda de Mazarredo, 14 -- E-48009 Bilbao, Basque Country, Spain}

\smallbreak
\centerline{$^{*,b}\;$Ikerbasque -- Basque Foundation for Science}
\centerline{Plaza Euskadi 5, E-48009 Bilbao, Basque Country, SPAIN -- E-48009 Bilbao, Basque Country, Spain}

\smallbreak
\centerline{$^{*,c}\;$Univ. Lyon, Universit\'e Claude Bernard Lyon 1, CNRS UMR 5208, Institut Camille Jordan,}
\centerline{43, Boulevard du 11 novembre 1918 -- F-69622 Villeurbanne, France}

\smallbreak
\centerline{\small \texttt{ffanelli@bcamath.org}}

\bigbreak
\centerline{$^\ddag\;$School of Mathematical Sciences, Anhui University}
\centerline{Hefei -- 230601, People's Republic of China}

\smallbreak
\centerline{\small \texttt{lynjum@163.com}}

\bigbreak
\centerline{$^\S\;$Mathematics Institute, University of Warwick}
\centerline{Zeeman Building, Coventry CV4 7AL -- United Kingdom}

\smallbreak
\centerline{\small \texttt{ewelina.zatorska@warwick.ac.uk}}

}

\begin{abstract}
{
In this paper, we revisit the joint low-Mach and low-Froude number limit for the compressible Navier–Stokes equations with degenerate, density-dependent viscosity. Employing the relative entropy framework based on the concept of $\kappa$-entropy, we rigorously justify the convergence of weak solutions toward the generalized anelastic system in a three-dimensional periodic domain for well-prepared initial data. For general ill-prepared initial data, we establish high Reynolds convergence to the anelastic Euler system in the whole space, relying essentially on dispersive estimates for acoustic waves.     
Compared with the work of Fanelli and Zatorska [\emph{Commun. Math. Phys.}, 400 (2023), pp. 1463-1506], our analysis is conducted for the standard isentropic
pressure law, thereby eliminating the need for the cold pressure term that played a crucial role in the previous approach. To the best of our knowledge, this is the first rigorous singular limit result for the compressible Navier–Stokes equations with degenerate viscosity that requires no additional regularization of the system.
}

\end{abstract}

\vspace{2mm}


\vspace{2mm}


\tableofcontents

\section{Introduction}
In this work, we revisit the derivation of the generalized anelastic approximation from the compressible Navier–Stokes equations with degenerate, density-dependent viscosity, a problem first analyzed in \cite{FZ23}. We consider the regime where the Mach and Froude numbers share the same scaling with respect to a small parameter $\ep>0$, corresponding to the strong stratification regime.
The governing equations for the {\emph{primitive system}}, posed on a spatial domain $\Omega\subset\R^3$, take the following form: 

\begin{subnumcases}{\label{dege-ns}}
	\p_t  \vr+ \Div(  \vr \vu )=0,\label{dege-ns1}\\
	\p_t(\vr \vu) +\Div (\vr  \vu \otimes \vu ) + \f{1}{\ep^2}\Grad  p(\vr)- 2 \mu \Div ( \vr \bd \vu ) = \f{1}{\ep^2} \vr \Grad G,\label{dege-ns2}
\end{subnumcases} 
where the unknowns are the mass density $\vr=\vr(t,x)\geq 0$ and the velocity field $\vu=\vu(t,x)\in \R^3$. The barotropic pressure law is given by $p(\vr)=a\vr^{\gamma}$, with constants $a>0$ and $\gamma>1$; the  viscosity coefficient $\mu=\mu_\ep>0$ is constant; $\bd \vu=\frac12(\Grad \vu+\Grad^{t} \vu)$ denotes the symmetric part of the velocity gradient.
$G=G(x)$ represents a potential external force like gravity, for example. For later reference, we also denote by $\ba \vu=\frac12(\Grad \vu-\Grad^{t} \vu)$ the antisymmetric part of the velocity gradient.

We focus on the bulk dynamics and neglect boundary effects. Thus, the spatial domain considered here is
\begin{align*}
 \mbox{either the periodic box $\;\Om=\mathbb{T}^3$, \qquad 
 or the whole space $\;\Om=\mathbb{R}^3$.}
\end{align*}
In the latter case, the system is supplemented with the following far-field conditions:
\begin{align}\label{eq:far-field}
 \vr \rightarrow \overline{\vr}>0,
 \qquad 
 \vu \rightarrow \mathbf{0} \qquad
 \text{ as }\; |x|\rightarrow \infty\,,
\end{align}
where $\overline{\vr}>0$ is a constant reference density. 

From the physical point of view, viscosity typically depends on thermodynamic state variables, like pressure and temperature, hence on density through the equation of state. This
dependence is particularly relevant for dense gases, liquids, gas mixtures, polyatomic gases, and acoustically active or reacting flows. For example, in \cite{FN09}, only case of temperature-dependence was studied, which, for isothermal flows, translates into dependence on the density.
Density-dependent viscosities also arise naturally in viscous shallow-water models, where the role of the density is played by the fluid height and the effective viscosity is proportional to this height, see, for instance, \cite{BD03,BDLin03}. Other systems with density or pressure-dependent viscosities are used in the description of avalanches, lubrication, diffusive interface models or even some models of collective behaviour, see for example \cite{Abels,A-MF-W,Dunn-S,Kort,M45,Raja,Acaves} and references therein.

\subsection{The generalized anelastic system}
Assume that, as $\ep\to0$, we have the following limits:
\begin{align*}
   \vr=\vr_{\ep}\rightarrow b
   \qquad \mbox{ and }\qquad
   \vu=\vu_{\ep}
   \rightarrow 
   \vU.
\end{align*}
We seek for the equations governing the dynamics of the limiting profiles $b=b(t,x)$ and $\vU = \vU(t,x)$. The following derivation is formal; its rigorous justification
precisely constitutes the core of this paper.

A direct balance of the leading-order terms in \eqref{dege-ns2} yields the following constraint for $b$:
\begin{equation}\label{const-b-1_bis}
\Grad p(b)=b\Grad G.
\end{equation}
Equivalently,
\begin{equation} \label{eq:stat-b}
 a\,\frac{\g}{\g-1}\Grad b^{\g-1} = \Grad G,
\end{equation}
which implies that $\Grad b$ is time-independent. Hence $b(t,x) = c(t) + \wtilde b(x)$. In the whole-space setting $\Omega = \R^3$, the far-field condition $b \to \oline{\vr}$ as $|x|\to \infty$ forces $c(t)\equiv c_0$ to be constant in time. The same conclusion follows in the periodic case $\Omega= \T^3$: by taking the limit $\ep \rightarrow 0$ in \eqref{dege-ns1}
we find
\begin{align*}
 \p_tb + \Div (b \vU) = 0,
\end{align*}
and, by taking the integral over $\Omega$, we find again that $c(t)\equiv c_0$.

For smooth enough $G$, the monotonicity of $p(\vr)$ ensures the existence of a smooth, positive solution $b$ to \eqref{const-b-1_bis} satisfying the uniform bounds
\begin{equation}\label{const-b-2}
0\,<\,\underline{b}\,\leq\, b(x)\, \leq\, \overline{b}\,<\,\infty\qquad \text{  for all  }\quad x \in \Om.
\end{equation}
We refer to \cite{FZ23} for the details.
Let us point out here that, in the whole space case $\Omega=\R^3$, we will need to impose some additional assumption on the potential $G$, so that the above conditions \eqref{const-b-2}
for the static state $b$ will be reinforced to \eqref{eq:prop-b}.

Passing formally to the limit $\ep \rightarrow 0$ in  \eqref{dege-ns} then allows to derive a generalized anelastic system
with variable viscosity, writing $\mu=\lim_{\ep\to0}\mu_\ep$: 
\begin{subnumcases}{\label{ane-app}}
	 \Div(  b \vU )=0,\label{ane-app1}\\
	\p_t \vU + \vU \cdot \Grad \vU +\Grad \Pi -b^{-1}2\mu \Div (b \bd \vU)=\mathbf{0}. \label{ane-app2}
\end{subnumcases} 
We call it the \emph{generalized anelastic approximation}. In contrast, the \emph{classical} anelastic model corresponds to a constant-viscosity stress tensor  $ b^{-1} \mu \Delta \vU$ (notice that the factor $b^{-1}$ appears from dividing
the momentum equation by $b$ and should not be confused with a variable viscosity coefficient).

\subsection{Derivation of the anelastic approximation: previous results}

We briefly review the relevant literature on singular limits leading to the anelastic approximation, leaving aside the extensive body of work on general low-Mach and low-Froude number limits.

The model \eqref{ane-app} was first formally derived by Ogura and Phillips \cite{OP} for inviscid ideal gases without heat conduction.
When $\mu=0$, the limit system coincides with the lake equation, well-known in the analysis of shallow-water flows; see \cite{LOT96,Oli97}. Accordingly, \eqref{ane-app} may be viewed as a viscous extension of the lake equation with variable viscosity.

The rigorous justification of the anelastic limit for the compressible Navier–Stokes equations is a classical and challenging problem. For the barotropic Navier–Stokes system with constant viscosity, Masmoudi
\cite{Mas07} proved the limit in a bounded domain for finite-energy weak solutions and ill-prepared data, using an approach based on compensated compactness arguments
previously developed in \cite{LioMas98}. Feireisl et al. \cite{FMNS08} subsequently obtained a similar result in the periodic case for small perturbations of the ideal gas pressure law.
 In their monograph \cite{FN09}, Feireisl and Novotn\'{y}  extended this analysis to entropy weak solutions of the Navier–Stokes–Fourier system, justifying both strongly and weakly stratified regimes. More recently, 
Feireisl et al. \cite{FKNZ16} studied the Navier–Stokes–Fourier system without thermal diffusion and identified the limiting system as the anelastic approximation coupled with a temperature transport equation, as originally proposed in \cite{OP}. Working in an infinite slab $\mathbb{R}^2\times (0,1)$
with slip boundary conditions, they employed the RAGE theorem from scattering theory to establish local decay of acoustic modes, yielding strong convergence of the velocity field.
An analogous result in the whole space $\R^3$ was later obtained by Donatelli and Feireisl \cite{DF-17} using frequency-localized Strichartz estimates.

In contrast, singular-limit results for \emph{degenerate} compressible Navier–Stokes equations are considerably more scarce. The principal difficulty lies in the loss of coercivity near vacuum: as $\rho\approx0$, equation \eqref{dege-ns} provides no direct control on $\vu$ or its gradient.

Three main strategies have been explored in the literature to overcome this degeneracy:

\begin{enumerate}[(1)]
\item{
{\bf{Artificial drag terms}}. 
Bresch et al. \cite{BGL05} considered the low-Mach and low-Froude numbers limit of the degenerate system augmented by artificial friction terms $r_0 \vu+r_1\vr |\vu| \vu$.
For the shallow-water pressure law $p(\vr)=\vr^2/2$ and well-prepared initial data, they justified convergence to the anelastic system via a relative-energy method.
}
\item{
{\bf{Capillarity term}}. Taking the Coriolis force $\vr \vu^{\perp}$ and a capillarity term $\kappa \vr \nabla \Delta \vr$ into account, Bresch and Desjardins \cite{BD03} (see
also J\"{u}ngel et al. \cite{JLW14} for an analogous result for a slightly different choice of the capillarity term)
considered the low Mach number and low Rossby number limits for the degenerate compressible  Navier--Stokes-Korteweg system
(or Quantum Navier--Stokes as in \cite{JLW14}).
For well-prepared data, and two-dimensional periodic domain, they obtained convergence to the viscous quasi-geostrophic equation using techniques analogous to \cite{BGL05}. The analysis was extended to three dimensions and ill-prepared data by Fanelli \cite{F-2016} who employed symmetrization and generalized RAGE estimates; see also \cite{F-2016-JMFM} for a related study based on a compensated-compactness approach.

Further developments by Caggio et al. \cite{CDH-25} include the inviscid incompressible limit in $\R^3$ for degenerate capillary systems, and the high-Mach number regime leading to pressureless capillary flows \cite{CD24}. 

In all these works, the target density profile was constant, unlike in the present study where $b$ varies in space, cf. \eqref{const-b-1_bis}. Moreover, the relative entropy
used in these results does not include the BD entropy part, which is the most challenging one to estimate.
}
\item{
{\bf{Cold pressure component}}. In the three-dimensional torus $\T^3$, Fanelli and Zatorska \cite{FZ23} investigated the strongly stratified, low-Mach number limit for the degenerate compressible Navier--Stokes system with an additional \emph{cold pressure} term,
\begin{align*}
    p(\vr)= 
    \vr^{\gamma}-\vr^{-\kappa},\qquad 
    \gamma>1,\qquad \kappa \geq \gamma-2,\ \kappa>3. 
\end{align*}
which restores parabolicity of \eqref{dege-ns2} near vacuum. 
This modification enabled them to generalize \cite{BGL05} to more general pressure laws and to the case of ill-prepared initial data, while removing the artificial drag terms.  
}
\end{enumerate}

\subsection{Main contributions of this work} \label{ss:intro_overview}

Building on the theory of finite-energy weak solutions for degenerate compressible Navier–Stokes equations \cite{VaYu16, li-xin, {BVY21}}, 
our goal is to rigorously justify the generalized anelastic approximation \eqref{ane-app} for standard barotropic pressure laws, without introducing any additional regularization effect in the system -- no drag, capillarity, or cold-pressure terms.

More precisely, we establish the joint low-Mach/low-Froude number limit for the degenerate compressible Navier–Stokes system \eqref{dege-ns} in the class of global $\kappa$-entropy weak solutions introduced in \cite{BDZ15}, both for well-prepared and ill-prepared initial data. The novelty of this work lies in proving convergence without any auxiliary regularizing strategies.

Our results provide ``weak-to-strong" convergence:
weak $\kappa$-entropy solutions of the primitive system converge to strong solutions of the target anelastic equations
 \eqref{ane-app}, which are locally well-posed in three dimensions.
The proof relies on a relative-entropy framework tailored to degenerate viscosity, implemented via the two-velocity formulation and the associated relative $\kappa$-entropy functional developed in \cite{BDZ15, BNV15, BNV17}. For ill-prepared data, the argument is further combined with a careful analysis of acoustic wave dispersion, inspired by \cite{FJN14, F-N_AMPA, F-N_CPDE}  and refined by the Strichartz-type estimates used in \cite{DF-17}.  
The treatment of ill-prepared data introduces essential differences from the well-prepared case. First, the analysis is carried out in the whole space $\Omega = \R^3$ to exploit dispersive properties of acoustic modes,
whereas well-prepared data are treated on the periodic box  $\Omega = \T^3$. Second, the choice of test functions in the relative-entropy inequality differs. For well-prepared data, the comparison is made directly with the strong anelastic solution, emanating from the limiting initial data, for which the initial relative entropy converges to zero. For ill-prepared data, this is no longer the case, and an additional oscillatory correction—representing the acoustic component—must be included. Since the limiting anelastic system is independent of acoustic waves, one expects that the initial perturbation will disperse over the time, a fact rigorously justified through Strichartz-type estimates developed previously in \cite{DF-17}.

\subsection*{Structure of the paper}

The rest of this paper is arranged as follows. Section \ref{s:results} introduces the precise formulation of the main results.
Subsection \ref{ss:weak} recalls the definition of weak $\kappa$-entropy solutions to the primitive system \eqref{dege-ns} and its two-velocity formulation. More explanations on the theory of that class of solutions are postponed to Appendix.
Subsection \ref{ss:strong-target} is devoted to the local well-posedness of the target anelastic system \eqref{ane-app}.
Subsection \ref{ss:results} states the two main theorems:
Theorem \ref{TH1} for well-prepared data and Theorem \ref{TH2} for ill-prepared data.

Section \ref{sec:REI} derives the relative-entropy inequality for weak $\kappa$-entropy solutions and smooth test functions. The basic form is first derived  in Subsection  \ref{Sect:3.1}, and Subsection \ref{ss:improved} gives an improved version, obtained when assuming further conditions on the smooth
test functions.
Section \ref{Sec:unif} derives uniform in $\ep$ estimates for sequence of weak $\kappa$-entropy solutions following from the relative $\kappa$-entropy estimates.

Section \ref{sec:well} establishes the singular-limit result for well-prepared data, 
while Section \ref{sec:ill} deals with the more delicate case of ill-prepared data, combining relative-entropy estimates with dispersive analysis of acoustic waves.

The paper ends with Section \ref{sec:6}, containing some final remarks, and the above mentioned Appendix.

\section{Formulation of the main results} \label{s:results}

In this section we present the rigorous formulation of our convergence results from the degenerate compressible Navier–Stokes system to the generalized anelastic approximation.

Before stating the main theorems, we recall the notion of $\kappa$-entropy weak solutions for the primitive system, review the two-velocity formulation, and discuss the local well-posedness of the target anelastic system.

\subsection{Weak solutions to the primitive system} \label{ss:weak}

The global existence theory for degenerate compressible Navier–Stokes equations without any additional regularizing mechanism—such as artificial drag, capillarity, or cold pressure—was first established by Vasseur and Yu \cite{VaYu16} in the three-dimensional torus $\T^3$. Their definition of weak solution involves the standard distributional formulation of the mass and momentum equations together with the basic energy inequality (corresponding to
$\kappa=0$ in  \eqref{energy-inequa} below). Later, Bresch, Vasseur and Yu  \cite{BVY21} extended the result to nonlinear viscosity coefficients satisfying the Bresch–Desjardins relation, proving the existence of velocity-renormalized weak solutions. These solutions are automatically weak solutions in the usual distributional sense; they satisfy the $\kappa$-entropy inequality introduced in \cite{BDZ15} while bypassing the need for Mellet–Vasseur higher-integrability estimates on  $\vr|\vu|^2$.

Our definition of weak solutions for the primitive system \eqref{dege-ns} follows this $\k$-entropy framework.
This class of solutions admits a generalized relative-entropy inequality, which is the cornerstone of our asymptotic analysis.

For clarity, we first recall the standard definition on a torus $\Omega=\T^3$ (see Definition \ref{def-1} below) and a few technical points that are well-understood in the literature \cite{VaYu16,BVY21} but essential for the reader’s orientation. These will be formulated in the series of remarks that follow.

\begin{Definition}[\bf $\kappa$-entropy solution to one velocity system]\label{def-1}
Let $\Omega=\T^3$, $T>0$, $\kappa\in (0,1)$. A pair $(\vr,\vu)$ is said to be a \emph{weak $\kappa$-entropy solution} to \eqref{dege-ns}
with initial data $(\vr_0,\vu_0)$ if  the following conditions are satisfied:
\begin{itemize}
\item{ Regularity of solution:
\begin{align}
&  \vr \geq 0,\quad 
\overline{\sqrt{\vr} \, \Grad \vu} \in L^2\big(0,T;L^2(\Om;\R^{3\times 3})\big),\quad 
\vr \in L^{\infty}\big(0,T;L^{\gamma}(\Om)\big) ,      \nonumber  \\
&  \sqrt{\vr} \, \Grad\log \vr,    
\sqrt{\vr} \, \vu\in  L^{\infty}\big(0,T;L^2(\Om;\R^3)\big) ,
\quad 
   \Grad \vr^{\f{\gamma}{2}} \in L^2\big(0,T;L^2(\Om;\R^3)\big);          \label{regu}
\end{align}
}
\item { The continuity equation
\begin{align}
    &  \int_0^{T} \int_{\Om} (\vr \p_t \phi+ \vr \vu \cdot \Grad \phi) \dxdt= - \int_{\Om} \vr_0 \phi(0) \dx     \label{equ-conti}
\end{align}
      holds for any $\phi\in C_c^{\infty}( [0,T) \times \Om)$;
}
\item{ The momentum equation
\begin{align}
    &  \int_0^{T} \int_{\Om} \left(   \vr \vu \cdot \p_t \vc{\phi} +\vr \vu \otimes \vu :\Grad \vc{\phi}+ \f{1}{\ep^2} p(\vr) \Div \, \vc{\phi} -2\mu \sqrt{\vr} \, \overline{ \sqrt{\vr} \, \bd \vu}:\bd\vc{\phi}                     \right) \dxdt    \nonumber \\
    & \quad  =  -   \f{1}{\ep^2}  \int_0^{T} \int_{\Om}  \vr \Grad G \cdot \vc{\phi} \dxdt    -   \int_{\Om} \vr_0 \vu_0 \cdot \vc{\phi}(0) \dx      \label{balan-momentu}
\end{align}
      holds for any $\vc{\phi} \in C_c^{\infty}( [0,T) \times \Om;\R^3)$;
}
\item{$\kappa$-entropy inequality: define
\begin{align}\label{def:vw}
\vv:=\vu+2\kappa \mu \Grad \log \vr\quad  \text{and} \quad \vw:=2 \sqrt{\kappa (1-\kappa)} \, \mu \Grad \log \vr,
\end{align}
then for almost all $t\in (0,T)$,
\begin{align}
    &    \sup_{\tau \in [0,t]} \left[ \int_{\Om}\vr \left( \f{|\vv|^2}{2}+\f{|\vw|^2}{2}  \right) (\tau) \dx        + \f{1}{\ep^2} \int_{\Om}
          \f{a}{\gamma-1}\vr^{\gamma} (\tau) \dx  \right]    \nonumber     \\
    &   \quad   \quad        +  2\mu \int_0^t \int_{\Om}   \Big(   \left|\Ov{\sqrt{\vr\kappa} \, \ba \vv} \right|^2 + \left|\Ov{\sqrt{\vr}\lr{\sqrt{1-\kappa} \,\bd \vv -\sqrt{\kappa} \, \Grad\vw } } \right|^2  \Big)\dx\ds                  \nonumber     \\
    &    \quad   \quad        +  2\f{\kappa \mu}{\ep^2} \int_0^t \int_{\Om} \f{p'(\vr)}{\vr} |\Grad \vr|^2 \dx\ds                    \nonumber     \\
    &   \quad     \leq   \int_{\Om}\vr_0 \left( \f{|\vv_0|^2}{2}+\f{|\vw_0|^2}{2}  \right)  \dx        +  \f{1}{\ep^2}  \int_{\Om}
          \f{a}{\gamma-1}\vr_0^{\gamma} \dx   + \f{1}{\ep^2}  \int_0^t \int_{\Om} \vr \Grad G \cdot \vv  \dx\ds,             \label{energy-inequa}
\end{align}
where $\vv_0$ and $\vw_0$ are defined by formulas \eqref{def:vw} applied to the initial datum $(\vr_0,\vu_0)$.
}
\end{itemize}
\end{Definition}

\begin{Remark}
   The regularity class \eqref{regu} implies that
    \begin{align*}
    \sqrt{\vr} \vw, \sqrt{\vr}\vv\in  L^{\infty}\big(0,T;L^2(\Om;\R^3)\big).
    \end{align*}
    Hence, every term appearing in \eqref{energy-inequa} is well-defined.
\end{Remark}
\begin{Remark} \label{r:def-weak}
Because of the degeneracy of equation \eqref{dege-ns2} near vacuum, no information on $\vu$ or $\bd\vu$ is available when $\vr\approx0$. In particular,
even though we denote the solution by $(\vr,\vu)$, the quantity $\vu$ alone is not well-defined. The meaningful quantity that we work with is
$\sqrt{\vr} \,\vu$, 
which is indeed the quantity appearing in all the relations defining the weak solution, see \eqref{regu}--\eqref{balan-momentu}.
Likewise, we interpret \eqref{def:vw} defining $\vv$ and $\vw$ in terms of $\vu$ and $\Grad\log\vr$: they implicitly contain a factor $\sqrt{\vr}$ multiplying both
their left and right sides.
\end{Remark}

\begin{Remark}\label{Rem:2.1}
Any quantity involving gradients of $\vu$, e.g. $\sqrt{\vr} \, \Grad\vu$ 
can only be identified at the level of construction of approximate solutions. In the limit, these terms must be understood in the distributional  sense. To mark it,
we use the bar symbol, e.g. $\oline{\sqrt{\vr} \, \Grad\vu}$ instead of $\sqrt{\vr} \, \Grad\vu$, and denote
     \begin{align*}
    \sqrt{\vr} \, \overline{\sqrt{\vr} \, \Grad \vu}=\Grad\lr{\vr\vu}-2\sqrt{\vr}\vu\otimes\Grad\sqrt{\vr}\quad in \quad {\cal D}'\big((0,T)\times \Omega\big).
    \end{align*}
Similarly, we have
    \eq{
    \label{weak-id1}
    \int_0^{T} \!\!\!\int_{\Om} \sqrt{\vr}\, \Ov{\sqrt{\vr} \, \bd\vu}:\bd\vc{\phi} \dx \ds &=\langle\sqrt{\vr} \, \oline{\sqrt{\vr} \, \bd\vu},\bd\vc{\phi}\rangle\\
    &=-\int_0^{T} \!\!\!\int_{\Om} \frac{1}{2} \vr \vu \cdot \lr{\Delta \vc{\phi}+ \nabla \Div \vc{\phi}} \dx \ds -2 \int_0^{T} \!\!\!\int_{\Om}\lr{\nabla \sqrt{\vr } \otimes \sqrt{\vr}\vu} : \bd \vc{\phi} \dx\ds,
    }
    \eq{\label{weak-id2}
    \int_0^{T} \!\!\!\int_{\Om} \sqrt{\vr} \, \Ov{\sqrt{\vr} \, \ba\vu}:\ba\vc{\phi} \dx \ds &=\langle\sqrt{\vr} \, \oline{\sqrt{\vr} \, \ba\vu},\ba\vc{\phi}\rangle\\
    &=-\int_0^{T} \!\!\!\int_{\Om} \frac{1}{2} \vr \vu \cdot \lr{\Delta \vc{\phi}- \nabla \Div \vc{\phi}} \dx \ds 
    +2
    \int_0^{T} \!\!\!\int_{\Om}\lr{\nabla \sqrt{\vr } \otimes \sqrt{\vr}\vu} : \ba \vc{\phi} \dx\ds,
    }
    for any $\vc{\phi} \in C_c^{\infty}( (0,T) \times \Om;\R^3)$.
    \end{Remark}

More detailed discussion of the validity of the $\k$-entropy inequality \eqref{energy-inequa} is moved to  Appendix.

\medbreak
As pointed out by Remark 1.7 of \cite{BVY21}, regularity \eqref{regu} allows to consider an equation for $\Grad\log\vr$ that is satisfied in the sense of distributions. Subtracting this equation from the one for $\vu$ yields the \emph{two-velocity formulation} for the velocity fields $\vv, \vw$ introduced in \eqref{def:vw}:
\begin{subnumcases}{\label{dege-ns-aug}}
     \p_t  \vr+ \Div\lr{  \vr \lr{\vv- \sqrt{\frac{\kappa}{1-\kappa} }\vw} }=0,\label{dege-ns-aug1}\\
 \p_t(\vr \vv) +\Div \lr{\vr  \vv\otimes \lr{\vv- \sqrt{\frac{\kappa}{1-\kappa} } \vw}  }  + \f{1}{\ep^2}\Grad  p(\vr) =
\mu \Div (2\vr (1-\kappa)\bd \vv) + \mu \Div (2\kappa \vr \ba \vv)  \nonumber\\
     \quad  \quad   \quad  \quad  \quad  \quad   \quad  \quad        -\mu \Div \Big(   2\sqrt{\kappa (1-\kappa)} \vr \Grad \vw        \Big)+\f{1}{\ep^2} \vr \Grad G,\label{dege-ns-aug2}\\
\p_t (\vr\vw)+ \Div \lr{\vr  \vw\otimes \lr{\vv- \sqrt{ \frac{\kappa}{1-\kappa} } \vw}  } =\mu \Div (2\kappa \vr \Grad\vw)
-\mu \Div \Big( 2\sqrt{\kappa (1-\kappa)}   \vr (\bd-\ba)\vv\Big).\label{dege-ns-aug3}
\end{subnumcases}
Equations \eqref{dege-ns-aug} are rigorously obtained in \cite{BVY21} as the limit of smooth approximation $\big\{(\vr_n,\vu_n)\big\}_{n\in\N}$,
using the strong convergence of $\sqrt{\vr_n}$ and of $\sqrt{\vr_n}\varphi(\vu_n)$ for some renormalization function $\varphi$. 

In the formulas above, we have adopted the convention that
\begin{align*}
 \Div(a\otimes b) = \sum_{j=1}^3 \p_j(a b_j) = (b\cdot\Grad)a + \Div(b) a\,.
\end{align*}
We also notice that the writing $\oline{\sqrt{\vr}\,\ba\vv}$, which is simply a notation for an element (say) $\mbb V\in L^2_T(L^2)$ (keep \eqref{energy-inequa}
in mind), verifies $\sqrt{\vr}\,\mbb V = \oline{\vr\,\ba\vv}$. As a matter of fact, this equality can be obtained by passing to the limit in sequences of smooth approximate solutions.
Thus, this notation is consistent with the symbol $\vr\ba\vv$ appearing in system \eqref{dege-ns-aug},
and the same remains true also for the other quantities involving gradients.

System \eqref{dege-ns-aug} is the departure point for our definition of the two-velocity $\kappa$-entropy weak solution. Notice that discussions similar to the ones of
Remarks \ref{r:def-weak} and \ref{Rem:2.1}, concerning bared quantities instead of classical non-bared versions, apply also to the following definition.

\begin{Definition}[\bf $\kappa$-entropy weak solution of the two-velocity system]\label{def-1a}
Let $\Omega=\T^3$,  $0<T<\infty$, $\kappa\in(0,1)$. A triple $(\vr,\vv,\vw)$ is said to be a weak $\kappa$-entropy solution \eqref{dege-ns-aug} with initial data $(\vr_0,\vv_0,\vw_0)$ if the following conditions are satisfied:
    \begin{itemize}
    \item With $\vu := \vv - \sqrt{\frac{\k}{1-\k}}\, \vw$, the pair $(\vr,\vu)$ satisfies \eqref{regu};
\item The following properties hold true:
\eq{\label{def2:prop}
\oline{\sqrt{\vr} \, \ba\vw}=\mathbf{0}\quad and \quad \oline{\sqrt{\vr} \, \bd\vw} = \oline{\sqrt{\vr} \, \Grad \vw};}
    \item Equation \eqref{dege-ns-aug1} is satisfied in the weak sense: for any $\phi\in C_c^{\infty}( [0,T) \times \Om)$,
\begin{align}
    &  \int_0^{T} \int_{\Om} \lr{\vr \p_t \phi+ \vr \lr{\vv-\sqrt{ \frac{\kappa}{1-\kappa} } \vw}\cdot \Grad \phi} \dxdt= - \int_{\Om} \vr_0 \phi(0) \dx;     \label{equ-conti-aug}
\end{align}
\item{Equations \eqref{dege-ns-aug2} and \eqref{dege-ns-aug3} are satisfied in the sense of distributions:
for any $\vc{\phi} \in C_c^{\infty}( [0,T) \times \Om;\R^3)$, 
\begin{align}
    &  \int_0^{T} \int_{\Om} \left(   \vr \vv \cdot \p_t \vc{\phi} +\vr \vv \otimes \lr{\vv- \sqrt{ \frac{\kappa}{1-\kappa} } \vw} :\Grad \vc{\phi}+ \f{1}{\ep^2} p(\vr) \Div \vc{\phi}         \right) \dxdt    \nonumber \\
    &\quad \quad 
    -2\mu\int_0^{T} \int_{\Om}\lr{\sqrt{1-\kappa}\sqrt\vr  \, \overline{ \sqrt{1-\kappa} \sqrt{\vr} \, \bd \vv-\sqrt{\kappa}\sqrt{\vr}  \, \Grad\vw}:\bd \vc{\phi} + \kappa \sqrt{\vr} \, \Ov{\sqrt\vr  \,\ba\vv}:\ba\vc{\phi}     } \dxdt \nonumber\\
    & \quad 
    =  -   \f{1}{\ep^2}  \int_0^{T} \int_{\Om}  \vr \Grad G \cdot \vc{\phi} \dxdt    -   \int_{\Om} \vr_0 \vu_0 \cdot \vc{\phi}(0) \dx    \label{balan-momentu-aug}
\end{align}
and 
\begin{align}
&  \int_0^{T} \int_{\Om} \left(   \vr \vw \cdot \p_t \vc{\phi} +\vr \vw \otimes \lr{\vv- \sqrt{\frac{\kappa}{1-\kappa} }\vw} :\Grad \vc{\phi}     \right) \dxdt    \nonumber \\
& \quad \quad 
+2\mu 
\int_0^{T} \int_{\Om} 
\left( \sqrt{\kappa}\sqrt\vr  \, \overline{ \sqrt{1-\kappa} \sqrt{\vr} \,\bd \vv-\sqrt{\kappa}\sqrt{\vr} \, \Grad\vw}:\bd \vc{\phi}- \sqrt{ \kappa (1-\kappa) }\sqrt{\vr}
\,\, \Ov{ \sqrt{\vr}  \, \ba\vv }: \ba \vc{\phi}
\right) \dxdt \nonumber    \nonumber \\
& \quad 
=  -   \int_{\Om} \vr_0 \vw_0 \cdot \vc{\phi}(0) \dx  ;    \label{balan-momentu-aug-2}
\end{align}
}
\item{
        The $\kappa$-entropy inequality \eqref{energy-inequa} holds for almost all $t \in (0,T)$.
}
    \end{itemize}
\end{Definition}

\begin{Remark}
    The distributional formulations 
    \eqref{balan-momentu-aug}--\eqref{balan-momentu-aug-2} mirror \eqref{balan-momentu},
with terms involving $\bd\vv$, $\ba\vv$ and $\Grad\vw$ understood in the distributional sense, in full analogy with \eqref{weak-id1} and \eqref{weak-id2} above. 
\end{Remark}

\begin{Remark} \label{r:w}
 Properties $\oline{\sqrt{\vr} \, \ba\vw } = \mathbf{0}$ and $ \oline{\sqrt{\vr} \, \bd \vw} = \oline{ \sqrt{\vr} \, \Grad \vw}$,
appearing in \eqref{def2:prop},
are enough to give meaning to all the terms appearing in Definition \ref{def-1a}.
They are also the only properties needed in Section \ref{sec:REI} to rigorously derive the relative $\k$-entropy
inequality.

We prove in the Appendix that one can guarantee those properties in the process of construction of $\k$-entropy weak solutions,
in the sense of Definition \ref{def-1a}.
Furthermore, also in the Appendix, we show  that
\[
 \sqrt{\vr} \, \vw = 2 \mu \sqrt{\k (1-\k)} \, \sqrt{\vr} \,\Grad\log\vr\,,
\]
which plays a crucial role in our computations.
\end{Remark}

The $\kappa$-entropy two-velocity formulation originates from \cite{BDZ15} and was later employed in \cite{BNV15, BNV17} to define the relative $\kappa$-entropy functional: 
\begin{align}
 \mathscr{E}( \vr,\vv,\vw| r,\vV,\vW):= & \f{1}{2} \int_{\Om}  \vr \Big(|\vv-\vV|^2+|\vw-\vW|^2  \Big)     \dx  \nonumber \\
& \quad +\f{1}{\ep^2} \int_{\Om}   \Big( P(\vr)-P(r)-P'(r)(\vr-r)  \Big)                        \dx,  \label{eq:def-E}
\end{align}
where the so-called pressure potential $P$ reads as 
\begin{align}
 P(\vr)\, =\, \f{a}{\g-1}\,\vr^\g\,.\label{Ppot}
\end{align}

The functional \eqref{eq:def-E} measures the ``distance" between a $\kappa$-entropy solution and any smooth test state $ (r,\vV,\vW)$ and serves as the core tool for the singular-limit analysis.
It can be seen as a two-velocity analogue of the very well-studied relative entropy for the compressible Navier--Stokes(-Fourier) equations with constant (or temperature-dependent) viscosity coefficients, see \cite{FJA12,FJN14,FNS11} for related results.  
 In particular, in \cite{BGL19} the authors generalized this relative entropy to the compressible Navier--Stokes-Korteweg system and used it to justify the limit when the viscosity coefficient tends to zero, leading to the quantum Euler equations. In the context of capillary fluids \cite{CD24}, a variant of \eqref{eq:def-E} incorporating capillary effects has been used to study the weak-strong uniqueness property and the high Mach number limit (which is exactly the opposite regime considered in the present work).

\medbreak
\begin{Remark} \label{r:space}
Definitions \ref{def-1} and \ref{def-1a}, as well as the following remarks, are tailored to the problem on bounded periodic space-domains.
However, performing an expansion argument one can obtain solutions on $\Omega=\R^3$, see \cite{li-xin, CCH22}. In both of these results the far-field condition for the density at spatial infinity is zero. For non-trivial far-field conditions \eqref{eq:far-field}, the definitions above adapt by:
\begin{itemize}
 \item replacing global $L^p$ bounds in \eqref{regu} with their local-in-space counterparts;
 \item replacing condition $\vr \in L^{\infty}(0,T;L^{\gamma}(\Om))$ in \eqref{regu} with
$P(\vr)-P(\overline{\vr})-P'(\overline{\vr})(\vr-\overline{\vr}) \in L^{\infty}(0,T;L^{1}(\Om))$;
\item replacing the terms $\f{a}{\gamma-1}\vr^{\gamma}$, $\f{a}{\gamma-1}\vr_0^{\gamma}$ on both sides of \eqref{energy-inequa} by  $P(\vr)-P(\overline{\vr})-P'(\overline{\vr})(\vr-\overline{\vr})$ and $P(\vr_0)-P(\overline{\vr})-P'(\overline{\vr})(\vr_0-\overline{\vr})$, respectively.
\end{itemize}
\end{Remark}

\subsection{Strong solutions to the target system} \label{ss:strong-target}
We next recall the local-in-time well-posedness theory for the limiting system \eqref{ane-app} in $\mathbb{T}^3$, whose regular solutions serve as test functions in the relative-entropy method.
\begin{Theorem}\label{thm:LWP}
Let $\Omega=\T^3$ and $m\geq 3$ an integer.
Let $b\in C^m(\Omega)$ satisfy \eqref{const-b-2} and $\mathbf{U}_0\in H^{m}(\Omega;\mathbb{R}^3)$ with $\Div (b \mathbf{U}_0)=0$.
Then, there exists $T_{\ast}>0$ such that \eqref{ane-app} admits a unique local strong solution $(\mathbf{U},\nabla\Pi)$ on $[0,T_{\ast})\times\Omega$
satisfying
\begin{align*}
& \sup_{0\leq t\leq T}
\Big( 
\| \mathbf{U} (t)  \|_{ H^{m}(\Omega)  }
+
\| \nabla \Pi (t)  \|_{ H^{m-2}(\Omega)  }
\Big) + \| \Grad\mathbf{U} \|_{L^2(0,T;H^m(\Omega))}  \leq C,
\end{align*}   
for any $0<T<T_{\ast}$, and for a positive constant $C = C(T,b,\mathbf{U}_0)$ depending only on the quantities inside the brackets.
\end{Theorem}

The proof follows from the classical theory of incompressible Navier–Stokes equations \cite{Ka72,Mc68},
since $b$ is smooth and bounded away from zero, or equivalently from the framework of semi-stationary inhomogeneous incompressible Navier–Stokes equations, see Chapter 3 of \cite{AKM90} or Section 2.2 of \cite{DF-17}.

For the sake of completeness, let us mention that in the periodic case the pressure $\Pi$ is fixed by the normalization $\int_{\mathbb{T}^3}\Pi(t,x)\dx=0$.
On $\R^3$, instead, one can only determine $\nabla \Pi$; for this, one solves the elliptic equation
\[
 -\,\div\big(b\,\nabla\Pi\big)\,=\,\div\Big(b\,\vu\cdot\nabla\vu\,-\,2\,\mu\,b\,\D \vu\Big)\,,
\]
which indeed admits a unique (up to additive constants of the space variable) solution, provided the right-hand side
decays sufficiently fast to zero as $|x|\rightarrow \infty$ (a condition which is guaranteed by the assumptions of Theorem \ref{thm:LWP}).

To treat the singular limit in $\mathbb{R}^3$, we additionally consider $\mu_\ep\to0$. 
The anelastic Euler system is \eqref{ane-app} with $\mu=0$, i.e., 
\begin{subnumcases}{\label{ane-Euler}}
 \Div(b\vU)=0,\label{ane-app1-Euler}\\
 \p_t\vU+(\vU\cdot\Grad)\vU+\Grad\Pi=0.\label{ane-app2-Euler}
\end{subnumcases}
Applying the classical result of Oliver \cite{Oli97} on the anelastic Euler system, see also Section 2.2 of \cite{DF-17},
we can derive the following statement.

\begin{TheoremPrime}\label{thm:LWP-R3}
Let $\Omega=\mathbb{R}^3$. Under the data hypotheses of Theorem~\ref{thm:LWP}. Then, there exists $T_{\ast}>0$ such that \eqref{ane-Euler} admits a unique local strong solution $(\mathbf{U},\nabla \Pi)$ on $[0,T_{\ast}) \times \Omega$ satisfying 
\begin{align*}
& \sup_{0\leq t\leq T}
\Big( 
\| \mathbf{U} (t)  \|_{ H^{m}(\Omega)  }
+
\| \nabla \Pi (t)  \|_{ H^{m}(\Omega)  }
\Big) 
\leq C,
\end{align*}   
for any $0<T<T_{\ast}$, and for a positive constant $C = C(T,b,\mathbf{U}_0)$ depending only on the quantities inside the brackets. 
\end{TheoremPrime}

\subsection{Statement of the main results} \label{ss:results}

We can now state the main theorems describing the singular limit under well-prepared and ill-prepared initial data.

\begin{Theorem}[Well-prepared data]\label{TH1}
Let $\Omega=\T^3$. 
Let $G\in C^\infty(\Omega)$ and $b\in C^\infty(\Omega)$ satisfy \eqref{const-b-1_bis} and \eqref{const-b-2}.
Let the initial data be subject to
\begin{align*}
& \vr_{0,\ep}=b+\ep \vr_{0,\ep}^{(1)},\quad   \vr_{0,\ep}^{(1)}  \rightarrow 0 \text{   strongly in  }  L^{\infty}( \mathbb{T}^3 ) ,   \\
&  \{ \Grad \vr_{0,\ep}^{(1)} \}_{\ep>0} \text{   is uniformly bounded in  } L^{\infty}(\mathbb{T}^3 ;\mathbb{R}^3),  \\ 
&   \vu_{0,\ep} \rightarrow \mathbf{U}_0 \text{   strongly in  }  L^2(\mathbb{T}^3 ;\mathbb{R}^3),\quad  \Div (b \mathbf{U}_0)=0, \\
&  \mathbf{U}_0 \in H^{m}(\mathbb{T}^3;\mathbb{R}^3), \quad
\text{  with  } m\in\N\,, \ m \geq 3. 
\end{align*}
Let $(\vr_{\ep},\vv_{\ep},\vw_{\ep})$ be a $\kappa$-entropy solution to  \eqref{dege-ns-aug} emanating from $(\vr_{0,\ep},\vv_{0,\ep},\vw_{0,\ep})$ as in Definition \ref{def-1a}, and let $\vU$ be the local strong solution of \eqref{ane-app} from Theorem \ref{thm:LWP}, with initial data $\vU_0$ and defined on the time interval $[0,T_\ast)$.

Define
\begin{align*}
\vV=\vU + 2\kappa \mu \Grad \log b,\quad 
\vW=2\sqrt{\kappa (1-\kappa)} \, \mu \Grad \log b.
\end{align*}
Then, for the relative-entropy functional $\mathscr{E}$ defined in \eqref{eq:def-E}, for any $0<T<T_\ast$ we have
\begin{align*}
\sup_{t\in[0,T]}\mathscr{E}(\vr_{\ep},\vv_{\ep},\vw_{\ep} | b,\vV ,\vW )(t) \rightarrow 0.
\end{align*}
Furthermore, in the limit $\ep\to 0$, for any $0<T<T_{\ast}$ we have:
\eq{\label{conv:well}
& \vr_{\ep} \rightarrow b \quad \text{ strongly in } 
L^{\infty}\big(0,T; L^{ \min\{2,\gamma\} }(\mathbb{T}^3)  \big) ,
\\
& 
\sqrt{\vr_{\ep}} \vv_{\ep}
\rightarrow 
\sqrt{b} \vV 
\quad 
\text{ strongly in } 
L^{\infty}\big( 0,T; L^2( \mathbb{T}^3 ; \mathbb{R}^3)  \big),
\\
& 
\sqrt{\vr_{\ep}} \vw_{\ep}
\rightarrow 
\sqrt{b} \vW 
\quad 
\text{strongly in } 
L^{\infty}\big( 0,T; L^2( \mathbb{T}^3 ; \mathbb{R}^3)  \big),
\\
& 
\sqrt{\vr_{\ep}} \vu_{\ep}
\rightarrow 
\sqrt{b} \vU 
\quad 
\text{ strongly in } 
L^{\infty}\big( 0,T; L^2( \mathbb{T}^3 ; \mathbb{R}^3)  \big).
}
\end{Theorem}

\bigskip

As mentioned in the introduction, the ill-prepared data case requires $\Omega = \R^3$. 
Moreover, as in \cite{FJN14}, we assume that $G\in C_c^{\infty}(\mathbb{R}^3)$ with $G \geq 0$, although these requirements could be somewhat relaxed in the spirit
of \cite{DF-17}, see Section \ref{sec:6} for more discussion. In addition,  we will assume that $b$ not only solves
\eqref{const-b-1_bis}, \eqref{const-b-2}, but also satisfies the far-field condition $b(x) \to \oline{\vr}$ for $|x|\to \infty$, for some constant value $\oline\vr>0$.
Then, $b$ must obey $P'(b)=G+P'(\overline{\vr})$, which, together with the assumption that $p'(\vr)>0$ for all $\vr>0$, gives
\begin{align}\label{eq:prop-b}
    b\in C^{\infty}(\mathbb{R}^3),
    \quad 
    b(x)\geq \overline{\vr}>0 \quad \text{ for  } x\in \mathbb{R}^3,
    \quad 
    b(x)= \overline{\vr} \quad \text{ for  } x\in \mathbb{R}^3 \backslash \text{supp}\, G. 
\end{align}
It goes without saying that \eqref{eq:prop-b} in particular implies \eqref{const-b-2}.

Before stating our second result, let us also introduce the generalized Helmholtz decomposition associated with
the reference density profile $b$. For any vector field $\mathbf{z}\in L^2(\R^3;\R^3)$, we write
\begin{align*}
\mathbf{z}= 
\mathbb{Q}_b[\mathbf{z}]+\mathbb{P}_b[\mathbf{z}],\quad \text{ where } \quad\mathbb{Q}_b[\mathbf{z}]:=b\Grad \Phi, \,\,\,\,\,\, \Div \,\mathbb{P}_b[\mathbf{z}]=0, 
\end{align*}
where $\Phi\in \dot{H}^{1}(\mathbb{R}^3)$ is the unique solution to the elliptic equation
\begin{align*}
\Div(b \Grad \Phi )=\Div \, \mathbf{z}  
\quad \text{ with } \mathbf{z}\in L^2(\mathbb{R}^3;\mathbb{R}^3).
\end{align*}
It is well-known that $\mathbb{P}_b[\mathbf{z}]$ and $\mathbb{Q}_b[\mathbf{z}]$ are bounded linear operators in $L^p(\mathbb{R}^3;\mathbb{R}^3)$ for any $1<p<\infty$, and orthogonal in the weighted Hilbert space $L^2_b(\mathbb{R}^3;\mathbb{R}^3)$ with weight $1/b$.

\begin{Theorem}[Ill-prepared data]\label{TH2}
Let $\Omega=\R^3$. Let $G\in C^\infty_c(\Omega)$ and $b\in C^\infty(\Omega)$ satisfy \eqref{const-b-1_bis} and \eqref{eq:prop-b}. Suppose that 
\begin{equation}\label{eq:van-vis}
 0<\mu=\mu_\ep\longrightarrow0\qquad\text{when}\ \ \ \ep\to0.
\end{equation}
Assume that the initial data satisfy
\begin{align*}
&  \vr_{0,\ep}=b+\ep \phi_{0,\ep},       \\ 
&  \{  \phi_{0,\ep}\}_{\ep>0} \text{  is uniformly bounded in  }L^{2}(\mathbb{R}^3) \cap L^{\infty}(\mathbb{R}^3) , \\
&  \left\{  \Grad \log \f{ \vr_{0,\ep} }{b}\right \}_{\ep>0} \text{  is uniformly bounded in  } L^{2}(\mathbb{R}^3;\mathbb{R}^3 ) \cap L^{\infty}(\mathbb{R}^3;\R^3) ,  \\
&  \{  \vu_{0,\ep}\}_{\ep>0} \text{  is uniformly bounded in  } L^{2}(\mathbb{R}^3;\mathbb{R}^3 ) 
\cap 
L^{\infty}(\mathbb{R}^3;\R^3) ,        \\
& 
(\phi_{0,\ep},  \vu_{0,\ep})
\to
(\phi_0, \vu_0)
\text{ strongly in }L^2(\R^3;\R^4),
\\
& \f{1}{b} \mathbb{P}_b[b \vu_0] \in H^{m}(\mathbb{R}^3;\R^3),  \quad 
\text{  with  } m\in\N\,, \ m \geq 3.
\end{align*} 
Let  $(\vr_{\ep},\vv_{\ep},\vw_{\ep})$ be a $\kappa$-entropy solution to \eqref{dege-ns-aug} emanating from $(\vr_{0,\ep},\vv_{0,\ep},\vw_{0,\ep})$ and let  $\vU$ be the local strong solution of \eqref{ane-Euler} from Theorem \ref{thm:LWP-R3} with initial data $\vU_0 := \f{1}{b} \mathbb{P}_b[b \vu_0]$ on $[0,T_{\ast})$.

Then,
in the limit $\ep\to 0$, for any $0<T<T_{\ast}$ we have the following convergence properties:
\begin{align*}
& 
\vr_{\ep} \rightarrow b \quad \text{ strongly in } 
L^{\infty}(0,T; L^2(\mathbb{R}^3)+L^{\gamma}(\mathbb{R}^3)  ) ,
\\
& 
\sqrt{\vr_{\ep}} \vv_{\ep}
\rightarrow 
\sqrt{b} \vU 
\quad 
\begin{cases}
\text{ strongly in } L^2( 0,T; L^2_{\rm loc}( \mathbb{R}^3 ; \mathbb{R}^3)  ),  \\
\text{ weakly-$\ast$ in } L^{\infty}( 0,T; L^2( \mathbb{R}^3 ; \mathbb{R}^3)  ),
\end{cases}
\\
& 
\sqrt{\vr_{\ep}} \vw_{\ep}
\rightarrow 
\mathbf{0}
\quad 
\begin{cases}
\text{ strongly in } L^2( 0,T; L^2_{\rm loc}( \mathbb{R}^3 ; \mathbb{R}^3)  ),  \\
\text{ weakly-$\ast$ in } L^{\infty}( 0,T; L^2( \mathbb{R}^3 ; \mathbb{R}^3)  ).
\end{cases}
\end{align*}
\end{Theorem}

We conclude this part with a comment on the choice of the viscosity coefficient in system \eqref{dege-ns}. 
A more general version of this system, allowing for nonlinear density-dependent viscosity coefficients,
\begin{align} \label{eq:ANS}
    \begin{cases}
        \p_t  \vr+ \Div(  \vr \vu )=0, \\ 
        \p_t(\vr \vu) +\Div (\vr  \vu \otimes \vu ) + \dfrac{1}{\ep^2}\Grad  p(\vr)=
        2  \Div \big( \mu( \vr) \bd \vu \big)
        +\Grad\big(\lambda(\vr)\Div \vu\big)
        + \dfrac{1}{\ep^2} \vr \Grad G, 
    \end{cases}
\end{align}
where the viscosity coefficients satisfy the BD-entropy relation
\[
    \lambda(\varrho)=2\big(\mu'(\varrho)\varrho-\mu(\varrho)\big),
\]
together with suitable growth assumptions, was studied in \cite{BVY21}. 
In view of this existence theory, we expect that the method developed here can be adapted to that more general framework. 
Indeed, a key ingredient in our analysis is the two-velocity formulation, expressed through the notion of $\kappa$-entropy solutions; this framework was developed for the general system \eqref{eq:ANS} by Bresch et al. \cite{BDZ15}. 
For clarity and to keep the presentation focused, we restrict ourselves in the present work to the case
\[
    \mu(\varrho)=\varrho, 
    \qquad 
    \lambda(\varrho)=0,
\]
which has been more extensively treated in the literature; see, for example, \cite{FZ23,VaYu16}.

\section{Relative \texorpdfstring{$\kappa$}{}-entropy inequality}\label{sec:REI}
In this section we derive and rigorously justify the relative $\kappa$-entropy inequality satisfied by weak $\kappa$-entropy solutions of the degenerate compressible Navier–Stokes system \eqref{dege-ns-aug}. Formally, this inequality has been derived in Bresch et al. \cite{BNV15} without the external force, and we focus here on the rigorous derivation relying on the definition of the weak solution.

\subsection{Derivation for general test functions}\label{Sect:3.1}
The purpose of this section is to prove the following result.

\begin{Lemma}\label{rela-entr-ine}
Let $\Omega=\T^3$ or  $\R^3$ and let $T>0$ and $\kappa\in(0,1)$.
Let $(\vr,\vv,\vw)$ be a $\kappa$-entropy solution to system \eqref{dege-ns-aug}
in the sense of Definition \ref{def-1a} and let $(r,\vV,\vW)$ be a triple of smooth functions such that 
\begin{align*}
& r\in C^1([0,T]\times \Om), \quad r>0, \quad  \vV,\vW \in C^1([0,T]\times \Om;\mathbb{R}^3),\\ 
& 
r \rightarrow \overline{\vr} \quad \text{as} \quad|x| \rightarrow \infty \quad \text{in case of} \quad\Omega= \mathbb{R}^3.  
\end{align*}
Let $\mathscr{E}$ be given by \eqref{eq:def-E}.

Then, for a.e.\ $\t\in(0,T)$, one has the following inequality:
\eq{\label{E:lem3.1}
& \mathscr{E}( \vr,\vv,\vw| r,\vV,\vW)(\tau)  -\mathscr{E}( \vr,\vv,\vw| r,\vV,\vW)(0)  \\
& \quad \quad +   2 \kappa \mu \int_0^{\tau}  \!\!\!\int_{\Om} \left| \oline{\sqrt{\vr} \, \ba(\vv-\vV) } \right|^2 \dx\ds + 
2\mu \int_0^{\tau}  \!\!\!\int_{\Om} \left|\oline{ \sqrt{\vr} \, \bd \left( \sqrt{1-\kappa}(\vv-\vV)-\sqrt{\kappa}(\vw-\vW)   \right)    }     \right|^2       \dx\ds                    \\
& \quad \quad + 2\f{\kappa \mu}{\ep^2} \int_0^{\tau}  \!\!\!\int_{\Om} \vr \Big( p'(\vr)\Grad \log \vr-p'(r)\Grad \log r  \Big) \cdot (\Grad \log \vr-\Grad \log r    )           \dx\ds \\ 
& \quad  \leq  \int_0^{\tau}\!\!\! \int_{\Om}\vr \left[  \left( \vv-\sqrt{ \f{\ka}{1-\ka} } \vw  \right)\cdot \Grad \vW \cdot (\vW-\vw)
+\left( \vv-\sqrt{ \f{\ka}{1-\ka} } \vw  \right)\cdot \Grad \vV \cdot (\vV-\vv)             \right] \dx\ds\\
& \quad \quad + \int_0^{\tau}  \!\!\!\int_{\Om}\vr \Big(  \p_t \vW\cdot (\vW-\vw)+\p_t \vV \cdot (\vV-\vv)     \Big) \dx\ds \\ 
& \quad \quad +\f{1}{\ep^2} \int_0^{\tau}  \!\!\!\int_{\Om}\p_t P'(r)(r-\vr) \dx\ds \\
& \quad \quad 
- 
\f{1}{\ep^2} \int_0^{\tau}  \!\!\!\int_{\Om}\Grad P'(r) \cdot \left[ \vr \left( \vv-\sqrt{ \f{\ka}{1-\ka} } \vw  \right)-r\left( \vV-\sqrt{ \f{\ka}{1-\ka} } \vW   \right)         \right] \dx\ds  \\
& \quad \quad +\f{1}{\ep^2} \int_0^{\tau}  \!\!\!\int_{\Om}( p(r)-p(\vr)  ) \Div \left( \vV-\sqrt{ \f{\ka}{1-\ka} } \vW   \right)  \dx\ds \\
&\quad \quad -\f{\ka}{\ep^2} \int_0^{\tau}  \!\!\!\int_{\Om}p'(\vr)\Grad \vr \cdot \left( 2\mu \f{\Grad r}{r} - \f{1}{ \sqrt{\ka(1-\ka)}  }\vW       \right) \dx\ds \\ 
& \quad \quad 
+ 2\mu \int_0^{\tau}  \!\!\!\int_{\Om} \sqrt{\vr} 
\Big(  
\bd(\sqrt{1-\kappa} \vV)
-   \bd( \sqrt{\ka} \vW)    
\Big):
\oline{\sqrt{\vr}\Big(  
\bd(\sqrt{1-\kappa} (\vV-\vv))
-    \bd( \sqrt{\ka} (\vW-\vw))     
\Big)}       \dx\ds \\
&  \quad \quad + 2\ka \mu \int_0^{\tau}  \!\!\!\int_{\Om} \sqrt{\vr} \, \ba  \vV: \oline{\sqrt{\vr} \, \ba(\vV-\vv)} \dx\ds
+ 
2\f{\kappa \mu}{\ep^2} \int_0^{\tau}  \!\!\!\int_{\Om}\f{\vr}{r} p'(r) \Grad r \cdot \left( \f{\Grad r}{r}-\f{\Grad \vr}{\vr}  \right)           \dx\ds   \\
&
\quad\quad +2 \sqrt{ \ka (1-\ka) } \mu \int_0^{\tau}  \!\!\!\int_{\Om}   \sqrt{\vr} \, \ba \vW: \oline{\sqrt{\vr} \,\ba \vv}    \dx\ds
+ \f{1}{\ep^2} \int_0^{\tau} \!\!\!\int_{\Om}\vr \Grad G \cdot (\vv-\vV) \dx\ds,
}
where we denoted for brevity $\Ov{\sqrt{\vr} \,\bd \lr{\vV-\vv}}=\sqrt\vr \, \bd\vV-\Ov{\sqrt{\vr} \, \bd \vv}$, and similarly for the other terms.
\end{Lemma}

\begin{proof}
We first expand the  relative $\kappa$-entropy functional introduced in \eqref{eq:def-E} as follows:
\begin{align}
	& \mathscr{E}( \vr,\vv,\vw| r,\vV,\vW)(\tau)-\mathscr{E}( \vr,\vv,\vw| r,\vV,\vW)(0) \nonumber \\
	& \quad =  \int_{\Om} \left(  \f{1}{2}\vr (|\vv|^2+|\vw|^2)  +\f{1}{\ep^2}P(\vr)         \right) (\tau) \dx -\int_{\Om} \left(  \f{1}{2}\vr (|\vv|^2+|\vw|^2)  +\f{1}{\ep^2}P(\vr)         \right) (0) \dx\nonumber \\
    & \quad \quad+ \int_{\Om} \f{1}{2}\vr (|\vV|^2+|\vW|^2)         (\tau)\dx-\int_{\Om} \f{1}{2}\vr (|\vV|^2+|\vW|^2)         (0)\dx\nonumber \\
	& \quad \quad-\int_{\Om}\vr(\vv\cdot \vV+\vw\cdot \vW)               (\tau)\dx+\int_{\Om}\vr(\vv\cdot \vV+\vw\cdot \vW)               (0)\dx\nonumber \\
	& \quad \quad- \int_{\Om}\f{1}{\ep^2} (P(r)+P'(r) (\vr-r)  )               (\tau)\dx+\int_{\Om}\f{1}{\ep^2} (P(r)+P'(r) (\vr-r)  )               (0)\dx  .\label{refE}
\end{align}

We now estimate the right-hand side of \eqref{refE} line by line. The first line is estimated using the $\kappa$-entropy inequality \eqref{energy-inequa}
More precisely, we have:
\begin{align}
    &     \int_{\Om} \left(  \f{1}{2}\vr (|\vv|^2+|\vw|^2)  +\f{1}{\ep^2}P(\vr)         \right) (\tau) \dx -\int_{\Om} \left(  \f{1}{2}\vr (|\vv|^2+|\vw|^2)  +\f{1}{\ep^2}P(\vr)         \right) (0) \dx    \nonumber     \\
    &   \quad  \leq - 2\mu \int_0^\tau \int_{\Om}   \left(   \left| \Ov{\sqrt{\vr\kappa} \, \ba \vv} \right|^2 + \left| \Ov{\sqrt{\vr}\lr{\sqrt{1-\kappa} \, \bd \vv -\sqrt{\kappa} \,\Grad\vw } } \right|^2  \right)\dx\ds  
                      \nonumber     \\
    &  \quad \quad -  2\f{\kappa \mu}{\ep^2} \int_0^\tau \int_{\Om} \f{p'(\vr)}{\vr} |\Grad \vr|^2 \dx\ds   + \f{1}{\ep^2}  \int_0^\tau \int_{\Om} \vr \Grad G \cdot \vv  \dx\ds.             \label{energy-inequa2}
\end{align}

For the second line on the right-hand side of \eqref{refE},
we  test the continuity equation by $\f{1}{2}|\vV|^2$ and $\f{1}{2}|\vW|^2$, to obtain
\eq{\label{line2}
	& \int_{\Om}  \f{1}{2}\vr (|\vV|^2+|\vW|^2)(\tau) \dx-\int_{\Om}  \f{1}{2}\vr (|\vV|^2+|\vW|^2)(0) \dx\\
    	&  =\int_0^{\tau}  \!\!\!\int_{\Om}\vr( \vV \cdot \p_t \vV+\vW \cdot \p_t \vW ) \dx\ds  \\
	& \quad +    \int_0^{\tau}  \!\!\!\int_{\Om} \vr \left[  \left( \vv-\sqrt{ \f{\ka}{1-\ka} } \vw   \right) \cdot \Grad \vV \cdot \vV     +\left( \vv-\sqrt{ \f{\ka}{1-\ka} } \vw   \right) \cdot \Grad \vW \cdot \vW                            \right]               \dx\ds. 
}

Next, to treat the third line, we test the equations for $\vv$ and $\vw$,
respectively \eqref{balan-momentu-aug} and \eqref{balan-momentu-aug-2}, by $\vV$ and $\vW$. We then get
\eq{\label{line3}
	&  - \int_{\Om} \vr (\vv \cdot \vV+ \vw \cdot \vW)(\tau) \dx+\int_{\Om} \vr (\vv \cdot \vV+ \vw \cdot \vW)(0) \dx \\
    &=-
	 \int_0^{\tau}  \!\!\!\int_{\Om} \vr(\vv\cdot \p_t \vV+ \vw \cdot \p_t \vW)\dx\ds\\
	&  \quad  - \int_0^{\tau}  \!\!\!\int_{\Om}\vr \left[  \left(  \vv-\sqrt{\f{\ka}{1-\ka}} \vw  \right) \cdot \Grad \vV \cdot \vv  + \left(  \vv-\sqrt{\f{\ka}{1-\ka}} \vw  \right) \cdot \Grad \vW \cdot \vw                          \right] \dx\ds                                      \\
	& \quad + 2\ka \mu   \int_0^{\tau}  \!\!\!\int_{\Om}   \Ov{\sqrt{\vr} \, \ba \vv} : \sqrt{\vr} \,  \ba \vV  \dx \dt        \\
	& \quad +   2 \mu   \int_0^{\tau}  \!\!\!\int_{\Om}\Ov{\sqrt{\vr} \Big( \bd(\sqrt{1-\ka}\vv)- \Grad (\sqrt{\ka} \vw)       \Big)} : \sqrt{\vr} \Big( \bd(\sqrt{1-\ka}\vV)- 
     \bd (\sqrt{\ka} \vW)    
    \Big) \dx \dt \\
	& \quad -\f{1}{\ep^2} \int_0^{\tau}  \!\!\!\int_{\Om}p(\vr)\Div \left(  \vV-\sqrt{\f{\ka}{1-\ka}}\vW  \right) \dx\ds  +\f{1}{\ep^2}  \sqrt{\f{\ka}{1-\ka}} \int_0^{\tau}  \!\!\!\int_{\Om}p'(\vr) \Grad \vr \cdot \vW \dx\ds     \\
	& \quad +   \int_0^{\tau}  \!\!\!\int_{\Om} 2 \sqrt{ \ka (1-\ka) } \mu \sqrt{\vr}  \Ov{\sqrt{\vr}  \, \ba \vv}: \ba \vW    \dx\ds                               \\
	& \quad 
    - \f{1}{\ep^2}\int_0^{\tau}  \!\!\!\int_{\Om}\vr \Grad G \cdot \vV \dx\ds.
}

Combining the bared terms in the right-hand sides of \eqref{energy-inequa2} and \eqref{line3}, we obtain
\eq{\label{almost:fin}
	  & -2 \kappa \mu \int_0^{\tau}  \!\!\!\int_{\Om} \left|\Ov{\sqrt{\vr} \, \ba \vv} \right|^2 \dx\ds - 2\mu \int_0^{\tau}  \!\!\!\int_{\Om} \left| \Ov{\sqrt{\vr} \, \bd (\sqrt{1-\kappa}\vv ) -\Grad(\sqrt{\kappa}\vw)  }
     \right|^2 \dx\ds \\ 
	 &+ 2\kappa \mu \int_0^{\tau}  \!\!\!\int_{\Om}\Ov{\sqrt{\vr} \,\ba \vv} : \sqrt{\vr} \, \ba \vV \dx \dt \\
	 &+2 \mu \int_0^{\tau}\int_{\Om}\Ov{\sqrt{\vr} \Big( \bd(\sqrt{1-\ka}\vv)- \Grad (\sqrt{\ka} \vw)       \Big)} : \sqrt{\vr} 
     \Big(
     \bd(\sqrt{1-\ka}\vV)
     - \bd (\sqrt{\ka} \vW)    
     \Big) \dx \dt \\
     &+   \int_0^{\tau}  \!\!\!\int_{\Om} 2 \sqrt{ \ka (1-\ka) } \mu \sqrt{\vr} \, \Ov{\sqrt{\vr}  \, \ba \vv}: \ba \vW    \dx\ds  \\
	 =&  - 2 \kappa \mu \int_0^{\tau}  \!\!\!\int_{\Om}\Ov{\sqrt{\vr} \,\ba \vv} :  \Ov{\sqrt{\vr} \, \ba \lr{\vv-\vV}} \dx\ds \\
	 & - 2\mu \int_0^{\tau}  \!\!\!\int_{\Om} \Ov{\sqrt{\vr} \, \bd (\sqrt{1-\kappa}\vv ) -\Grad(\sqrt{\kappa}\vw)  } :   
     \Ov{
     \sqrt{\vr} \, \bd (\sqrt{1-\kappa} \lr{\vv-\vV} ) 
     - \bd(\sqrt{\kappa}\lr{\vw-\vW} )  
     } \dx\ds \\
    &+   \int_0^{\tau}  \!\!\!\int_{\Om} 2 \sqrt{ \ka (1-\ka) } \mu \sqrt{\vr}  \Ov{\sqrt{\vr} \, \ba \vv}: \ba \vW    \dx\ds  \\
	 =&-2 \kappa \mu \int_0^{\tau}  \!\!\!\int_{\Om}  \left| \Ov{\sqrt{\vr} \, \ba(\vv-\vV) } \right|^2 \dx\ds - 2\mu \int_0^{\tau}  \!\!\!\int_{\Om} \left| \Ov{\sqrt{\vr} \, \bd \left( \sqrt{1-\kappa}(\vv-\vV)-\sqrt{\kappa}(\vw-\vW)   \right)         } \right|^2       \dx\ds \\
	 &  +
     2\ka \mu \int_0^{\tau}  \!\!\!\int_{\Om} \sqrt{\vr} \ba  \vV: \Ov{ \sqrt{\vr} \, \ba(\vV-\vv) } \dx\ds \\
	 &  +
     2 \mu \int_0^{\tau}  \!\!\!\int_{\Om}\sqrt{\vr} 
     \Big(  
     \bd(\sqrt{1-\kappa} \vV)
     -  \bd( \sqrt{\ka} \vW)     
     \Big): \Ov{ \sqrt{\vr} \Big(  
     \bd(\sqrt{1-\kappa} (\vV-\vv))
     - \bd( \sqrt{\ka} (\vW-\vw))   
     \Big)   }    \dx\ds\\
     &+   \int_0^{\tau}  \!\!\!\int_{\Om} 2 \sqrt{ \ka (1-\ka) } \mu \sqrt{\vr} \, \Ov{\sqrt{\vr} \,\ba \vv}: \ba \vW    \dx\ds ,
}
where we have used the property \eqref{def2:prop} multiple times, in particular that $\overline{\sqrt{\vr} \,\bd\vw}=\overline{\sqrt{\vr} \,\Grad\vw}$.

At this stage, we observe that the first two terms on the right-hand side of \eqref{almost:fin} are precisely the bared terms appearing on the left-hand side of  relative $\kappa$-entropy inequality \eqref{E:lem3.1}, while the three last ones are all the bared terms from the right-hand side of  \eqref{E:lem3.1}.

To treat the fourth line on the right-hand side of equality \eqref{refE}, we first test the continuity equation by $\f{1}{\ep^2}P'(r)$, and get
\begin{align*}
	&  -\f{1}{\ep^2} \int_0^{\tau}  \!\!\!\int_{\Om}\vr P'(r) (\tau) \dx+ \f{1}{\ep^2} \int_0^{\tau}  \!\!\!\int_{\Om}\vr P'(r) (0) \dx\\ 
    & \quad =-\f{1}{\ep^2} \int_0^{\tau}  \!\!\!\int_{\Om}\left[  \vr \p_t P'(r) + \vr  \left( \vv-\sqrt{ \f{\ka}{1-\ka} } \vw   \right) \cdot \Grad P'(r)         \right] \dx\ds.
\end{align*}
For the rest of the terms from the fourth line, we note that
the basic relation $\p_t (P'(r)r-P(r))=r \p_t P'(r) $ implies
\begin{align*}
	&  -\f{1}{\ep^2} \int_{\Om} \Big(P(r)-P'(r)r \Big)(\tau)\dx + \f{1}{\ep^2} \int_{\Om} \Big(P(r)-P'(r)r \Big)(0)\dx
	=\f{1}{\ep^2} \int_0^{\tau}  \!\!\!\int_{\Om} r \p_t P'(r) \dx\ds.
\end{align*}
We conclude the derivation of \eqref{E:lem3.1} by adding side by side the following two identities  to equation \eqref{refE}. The first one follows directly from definition \eqref{Ppot}: we have $P'(r)r - P(r) = p(r)$ for any $r>0$ and therefore, integrating by parts, we get
\begin{align*}
	0& = \f{1}{\ep^2} \int_{\Om} \Div \left( (P'(r)r-P(r) )\left( \vV-\sqrt{\f{\ka}{1-\ka}}\vW  \right)       \right) \dx \\
	&= \f{1}{\ep^2} \int_{\Om} p(r) \Div  \left( \vV-\sqrt{\f{\ka}{1-\ka}}\vW  \right)  \dx+\f{1}{\ep^2}  \int_{\Om} r \left( \vV-\sqrt{\f{\ka}{1-\ka}}\vW  \right) \cdot \Grad P'(r)  \dx.
\end{align*} 
The second one is the correcting term of the form 
\begin{align*}
	&   2\f{\kappa \mu}{\ep^2} \int_0^{\tau}  \!\!\!\int_{\Om} \vr \Big( p'(\vr)\Grad \log \vr-p'(r)\Grad \log r  \Big) \cdot (\Grad \log \vr-\Grad \log r    )           \dx\ds \\
	& \quad  =  2\f{\kappa \mu}{\ep^2} \int_0^{\tau}  \!\!\!\int_{\Om}\f{p'(\vr)}{\vr} |\Grad \vr|^2 \dx\ds \\
	& \quad   \quad 
    +   2\f{\kappa \mu}{\ep^2} \int_0^{\tau}  \!\!\!\int_{\Om}\f{\vr}{r} p'(r) \Grad r \cdot \left( \f{\Grad r}{r}-\f{\Grad \vr}{\vr}  \right)           \dx\ds -   2\f{\kappa \mu}{\ep^2} \int_0^{\tau}  \!\!\!\int_{\Om} p'(\vr) \Grad \vr \cdot \f{\Grad r}{r} \dx\ds. 
\end{align*}
Observe that the first term after the equality cancels out the penultimate term of \eqref{energy-inequa2}.

Putting all these expressions together, we finally complete the proof of Lemma \ref{rela-entr-ine}.
\end{proof}

\subsection{Reduction of the relative \texorpdfstring{$\kappa$}{}-entropy inequality} \label{ss:improved}
In this subsection we further assume that the test functions $ (r, \vV, \vW) $  used in \eqref{E:lem3.1} satisfy in addition:
\begin{align}\label{h1}
	\begin{split}
	&\partial_t r + \Div(r \vU)=0, \qquad \inf_{\mathbb{R}_+\times\Omega}r>0, \\  
	&\vU:=\vV - \beta \vW, \qquad \text{ with } \quad\beta=\sqrt{\frac{\kappa}{1-\kappa}},\\
	&\vW= 2\sqrt{\kappa(1-\kappa)}  \,  \mu \nabla \log r.   
\end{split}
\end{align}
Then, using again the fact that $P'(r)r-P(r)=p(r)$, we can write
\begin{align*}
    &\f{1}{\ep^2} \int_0^{\tau}  \!\!\!\int_{\Om}\p_t P'(r)(r-\vr) \dx\ds \\
    &
    \quad
    =
    \f{1}{\ep^2} \int_0^{\tau}  \!\!\!\int_{\Om} P^{\prime \prime }(r) \partial_t r (r-\vr) \dx\ds \\
    &
    \quad
    =- \f{1}{\ep^2} \int_0^{\tau}  \!\!\!\int_{\Om} P^{\prime \prime }(r)  (r-\vr) \Div(r\vU) \dx\ds \\
    &
    \quad
    =- \f{1}{\ep^2} \int_0^{\tau}  \!\!\!\int_{\Om} P^{\prime \prime }(r)  (r-\vr) r \Div\vU \dx\ds - \f{1}{\ep^2} \int_0^{\tau}  \!\!\!\int_{\Om} P^{\prime \prime }(r)  (r-\vr) \nabla r  \cdot \vU \dx\ds \\
    &
    \quad
    =- \f{1}{\ep^2} \int_0^{\tau}  \!\!\!\int_{\Om} p^{\prime }(r)  (r-\vr)  \Div\vU \dx\ds- \f{1}{\ep^2} \int_0^{\tau}  \!\!\!\int_{\Om} \nabla P^{\prime  }(r)  (r-\vr) \cdot \vU \dx\ds.
\end{align*}
Hence, using $\displaystyle \vr \vu-r\vU= \vr \lr{\vu-\vU} + (\vr-r)\vU $, we can transform the $\ep^{-2}$-order terms on the right-hand side of \eqref{E:lem3.1} as follows:
{\small{
\begin{align*}
    & \f{1}{\ep^2} \int_0^{\tau}  \!\!\!\int_{\Om}\p_t P'(r)(r-\vr) \dx\ds \\
& 
- 
\f{1}{\ep^2} \int_0^{\tau}  \!\!\!\int_{\Om}\Grad P'(r) \cdot \left[ \vr \left( \vv-\sqrt{ \f{\ka}{1-\ka} } \vw  \right)-r\left( \vV-\sqrt{ \f{\ka}{1-\ka} } \vW   \right)         \right] \dx\ds  \\
& +\f{1}{\ep^2} \int_0^{\tau}  \!\!\!\int_{\Om}( p(r)-p(\vr)  ) \Div \left( \vV-\sqrt{ \f{\ka}{1-\ka} } \vW   \right)  \dx\ds + 2\f{\kappa \mu}{\ep^2} \int_0^{\tau}  \!\!\!\int_{\Om}\f{\vr}{r} p'(r) \Grad r \cdot \left( \f{\Grad r}{r}-\f{\Grad \vr}{\vr}  \right)           \dx\ds\\
& -\f{\ka}{\ep^2} \int_0^{\tau}  \!\!\!\int_{\Om}p'(\vr)\Grad \vr \cdot \left( 2\mu \f{\Grad r}{r} - \f{1}{ \sqrt{\ka(1-\ka)}  }\vW       \right) \dx\ds   + \f{1}{\ep^2} \int_0^{\tau}  \!\!\!\int_{\Om}\vr \Grad G \cdot (\vv-\vV) \dx\ds\\
=&\f{1}{\ep^2} \int_0^{\tau}  \!\!\!\int_{\Om}\p_t P'(r)(r-\vr) \dx\ds-\f{1}{\ep^2} \int_0^{\tau}  \!\!\!\int_{\Om}\nabla P^\prime(r) (\vr-r) \cdot \vU  \dx\ds +\f{1}{\ep^2} \int_0^{\tau}  \!\!\!\int_{\Om}( p(r)-p(\vr)  ) \Div \vU  \dx\ds \\
&-\f{1}{\ep^2} \int_0^{\tau}  \!\!\!\int_{\Om}\Grad P'(r) \cdot \vr \lr{\vu-\vU}\dx\ds  
 -\f{1}{\ep^2} \int_0^{\tau}  \!\!\!\int_{\Om}\Grad P'(r) \cdot \vr \beta \lr{\vw-\vW}\dx\ds+ \f{1}{\ep^2} \int_0^{\tau}  \!\!\!\int_{\Om}\vr \Grad G \cdot (\vv-\vV) \dx\ds\\
=&  -\f{1}{\ep^2} \int_0^{\tau}  \!\!\!\int_{\Om}\lr{ p(\vr)-p(r) -  p'(r)(\vr-r) } \Div \vU\dx\ds 
	+ \f{1}{\ep^2} \int_0^{\tau}  \!\!\!\int_{\Om}\vr \lr{\Grad G- \Grad P^\prime(r)} \cdot (\vv-\vV) \dx\ds. 
\end{align*}     
}}
Consequently, under conditions \eqref{h1}, the relative energy inequality of Lemma \ref{rela-entr-ine} reduces to the following simpler expression:
\begin{align}
\nonumber	& \mathscr{E}( \vr,\vv,\vw| r,\vV,\vW)(\tau) -\mathscr{E}( \vr,\vv,\vw| r,\vV,\vW)(0) \\
\nonumber	& \quad \quad +   2 \kappa \mu \int_0^{\tau}  \!\!\!\int_{\Om}  \left|\oline{ \sqrt{\vr}  \, \ba(\vv-\vV) } \right|^2 \dx\ds + 
2\mu \int_0^{\tau}  \!\!\!\int_{\Om} \left| \oline{\sqrt{\vr} \, \bd \left( \sqrt{1-\kappa}(\vv-\vV)-\sqrt{\kappa}(\vw-\vW)   \right)} \right|^2   \dx\ds             \\
\nonumber	& \quad \quad + 2\f{\kappa \mu}{\ep^2} \int_0^{\tau}  \!\!\!\int_{\Om} \vr \Big( p'(\vr)\Grad \log \vr-p'(r)\Grad \log r  \Big) \cdot (\Grad \log \vr-\Grad \log r    )           \dx\ds \\ 
\nonumber	& \quad  \leq   \int_0^{\tau}\!\!\! \int_{\Om}\vr \left(  \vu -\vU  \right)\cdot \Grad \vW \cdot (\vW-\vw)\dx\ds 
	+\int_0^{\tau}  \!\!\!\int_{\Om}\vr  \left(  \vu -\vU  \right)  \cdot \Grad \vV \cdot (\vV-\vv)            \dx\ds\\
\nonumber	& \quad \quad + \int_0^{\tau}  \!\!\!\int_{\Om}\vr (\vW-\vw) \cdot\lr{ \p_t \vW+ \vU \cdot \nabla \vW} \dx \ds + \int_0^{\tau}  \!\!\!\int_{\Om} \vr  
 (\vV-\vv) \cdot
 \lr{ \p_t \vV+\vU \cdot \nabla \vV}   \dx\ds \\
\nonumber	& \quad \quad -\f{1}{\ep^2} \int_0^{\tau}  \!\!\!\int_{\Om}\lr{ p(\vr)-p(r) -  p'(r)(\vr-r) } \Div \vU\dx\ds \\
\nonumber	&  \quad \quad + \f{1}{\ep^2} \int_0^{\tau} \!\!\!\int_{\Om}\vr \lr{\Grad G- \Grad P^\prime(r)} \cdot (\vv-\vV) \dx\ds\\
\nonumber	& \quad \quad + 2\mu \int_0^{\tau}  \!\!\!\int_{\Om} \sqrt{\vr} \Big(  \bd(\sqrt{1-\kappa} \vV)- \bd( \sqrt{\ka} \vW)     \Big):
\oline{ \sqrt{\vr}\Big(  \bd(\sqrt{1-\kappa} (\vV-\vv))-\bd( \sqrt{\ka} (\vW-\vw))     \Big) }      \dx\ds \\
	&  \quad \quad + 2\ka \mu \int_0^{\tau}  \!\!\!\int_{\Om} \sqrt{\vr}  \, \ba  \vV: \oline{\sqrt{\vr}  \, \ba(\vV-\vv)} \dx\ds.
	\label{rel-ent-2}
\end{align}
To further reduce the left-hand side,  we use the following identity:
\begin{align}
\nonumber	&  \vr \Big( p'(\vr)\Grad \log \vr-p'(r)\Grad \log r \Big) \cdot (\Grad \log \vr-\Grad \log r    )       \\
	&\qquad = \vr  p'(\vr) \vert \Grad \log \vr-\Grad \log r \vert^2          +   \vr \Big( p'(\vr)-p'(r)\Big) \Grad \log r \cdot (\Grad \log \vr-\Grad \log r    ),
\label{eq:log-rho}
\end{align}
where the second term equals
\begin{align}
\nonumber &\vr \Big( p'(\vr)-p'(r)\Big) \Grad \log r \cdot (\Grad \log \vr-\Grad \log r    )      \\
	&\qquad = \nabla\big( p(\vr)-p(r) -  p'(r)(\vr-r) \big) \cdot \nabla \log r - \Big( \vr \big(p^\prime(\vr) - p^\prime(r)\big)- 
	p^{\prime \prime}(r) \ r \ (\vr-r) \Big) \vert \nabla \log r \vert^2. \label{eq:log-rho2}
\end{align}
Therefore, we deduce that
\begin{align}
\nonumber	& \mathscr{E}( \vr,\vv,\vw| r,\vV,\vW)(\tau) -\mathscr{E}( \vr,\vv,\vw| r,\vV,\vW)(0) \\
\nonumber	& \quad \quad +   2 \kappa \mu \int_0^{\tau}  \!\!\!\int_{\Om}  \left|\oline{ \sqrt{\vr}  \, \ba(\vv-\vV) } \right|^2 \dx\ds + 
2\mu \int_0^{\tau}  \!\!\!\int_{\Om} \left| \oline{\sqrt{\vr}   \, \bd \left( \sqrt{1-\kappa}(\vv-\vV)-\sqrt{\kappa}(\vw-\vW)   \right)} \right|^2   \dx\ds             \\
\nonumber	& \quad \quad + 2\f{\kappa \mu}{\ep^2} \int_0^{\tau}  \!\!\!\int_{\Om}\vr  p'(\vr) \vert \Grad \log \vr-\Grad \log r \vert^2             \dx\ds \\ 
\nonumber	& \quad  \leq   \int_0^{\tau}  \!\!\!\int_{\Om}\vr \left(  \vu -\vU  \right)\cdot \Grad \vW \cdot (\vW-\vw)\dx\ds 
	+\int_0^{\tau}  \!\!\!\int_{\Om}\vr  \left(  \vu -\vU  \right)  \cdot \Grad \vV \cdot (\vV-\vv)            \dx\ds\\
\nonumber	& \quad \quad + \int_0^{\tau}  \!\!\!\int_{\Om}\vr (\vW-\vw) \cdot\lr{ \p_t \vW+ \vU \cdot \nabla \vW} \dx \ds + \int_0^{\tau}  \!\!\!\int_{\Om} \vr  
 (\vV-\vv) \cdot
 \lr{ \p_t \vV+\vU \cdot \nabla \vV}   \dx\ds \\
\nonumber	& \quad \quad -\f{1}{\ep^2} \int_0^{\tau}  \!\!\!\int_{\Om}\lr{ p(\vr)-p(r) -  p'(r)(\vr-r) } \Div \vU\dx\ds \\
\nonumber	&  \quad \quad + \f{1}{\ep^2} \int_0^{\tau}  \!\!\!\int_{\Om}\vr \lr{\Grad G- \Grad P^\prime(r)} \cdot (\vv-\vV) \dx\ds\\
\nonumber	& \quad \quad + 2\mu \int_0^{\tau}  \!\!\!\int_{\Om} \sqrt{\vr} \Big(  \bd(\sqrt{1-\kappa} \vV)-\bd( \sqrt{\ka} \vW)     \Big):
\oline{ \sqrt{\vr}\Big(  \bd(\sqrt{1-\kappa} (\vV-\vv))-\bd( \sqrt{\ka} (\vW-\vw))     \Big) }      \dx\ds \\
	&  \quad \quad + 2\ka \mu \int_0^{\tau}  \!\!\!\int_{\Om} \sqrt{\vr}  \,\ba  \vV: \oline{\sqrt{\vr}  \, \ba(\vV-\vv)} \dx\ds \nonumber\\
\nonumber	& \quad \quad + 2\f{\kappa \mu}{\ep^2} \int_0^{\tau}  \!\!\!\int_{\Om}\big( p(\vr)-p(r) -  p'(r)(\vr-r) \big) \Delta \log r            \dx\ds \\ 
	& \quad \quad + 2\f{\kappa \mu}{\ep^2} \int_0^{\tau}  \!\!\!\int_{\Om}\Big( \vr \big(p^\prime(\vr) - p^\prime(r)\big)- 
	p^{\prime \prime}(r) \ r \ (\vr-r) \Big) \vert \nabla \log r \vert^2     \dx\ds.
	\label{rel-ent-2b}
\end{align}

\section{Uniform bounds based on the relative \texorpdfstring{$\k$}{}-entropy inequality}\label{Sec:unif}
In this section, we derive a series of uniform bounds for the family $\big\{(\vr_\veps, \vv_\veps, \vw_\veps)\big\}_{\veps>0}$
of $\k$-entropy solutions to the two-velocity system \eqref{dege-ns-aug}. They will be all deduced from the relative $\kappa$-entropy inequality
\eqref{rel-ent-2b} with a certain choice of the test functions $(r, \vV, \vW)$.

\subsection{Coercivity of the relative \texorpdfstring{$\k$}{}-entropy functional} \label{ss:coercivity}

We first explain the coercivity properties of the relative $\kappa$-entropy functional. These properties will be repeatedly used in the computations of Sections \ref{sec:well} and \ref{sec:ill}, when proving our main results.

Let $\big\{(\vr_\veps, \vv_\veps, \vw_\veps)\big\}_{\veps>0}$ be a family of weak $\k$-entropy solutions to the two-velocity system \eqref{dege-ns-aug} in the sense of
Definition \ref{def-1a}. Assume also that $(r, \vV, \vW)$ is a triplet of smooth test functions defined on some time interval $[0,T]$, with $T>0$.
We moreover require for $r$ to satisfy  $\underline{b}/2 \leq r(t,x) \leq 2\oline{b}$ for any $(t,x)\in[0,T]\times\Omega$, where the constants
$\underline{b}$ and $\oline b$ have been introduced in \eqref{const-b-2}; in the case $\Omega=\R^3$, further assume that
$r(t,x)\to\oline\vr$ for $|x|\to\infty$.

Recall that, for any $\tau \in [0,T]$, $\ep\in(0,1)$ fixed, the $\k$-relative energy functional is defined by the formula
\begin{align*}
 \mathscr{E}_\ep(\tau)=\mathscr{E}( \vr_\ep,\vv_\ep,\vw_\ep| r,\vV,\vW)(\tau):= 
 & \f{1}{2} \int_{\Om}  \vr_\ep \Big(|\vv_\ep-\vV|^2+|\vw_\ep-\vW|^2  \Big) (\tau)    \dx  \nonumber \\
& \quad +\f{1}{\ep^2} \int_{\Om}   \Big( P(\vr_\ep)-P(r)-P'(r)(\vr_\ep-r)  \Big)  (\tau)                      \dx,
\end{align*}
with the convention that, when $\t=0$, the triplet $(\vr_\ep,\vv_\ep,\vw_\ep)$ is replaced with the initial datum $(\vr_{0,\ep},\vv_{0,\ep},\vw_{0,\ep})$.

Then, we have
\begin{align*} 
    \mathscr{E}_\ep(\tau) \geq C \begin{cases}
        &\displaystyle \int_{\Om}\vr_\ep   \Big(|\vv_\ep-\vV|^2+|\vw_{\ep}-\vW|^2  \Big)     \dx  + \int_{\Om}   \lr{\f{\vr_\ep-r}{\ep}}^2         \dx, \quad \vr_\ep \in \left[ \underline{b}/2, 2\overline{b}\right],  \\[3ex] 
         &\displaystyle\int_{\Om}  \vr_\ep \Big(|\vv_\ep-\vV|^2+|\vw_\ep-\vW|^2  \Big)     \dx  +\f{1}{\ep^2} \int_{\Om}   \lr{1+\vr_\ep^\gamma} \dx ,\quad  \text{ otherwise},
    \end{cases}
\end{align*}
where $C$ depends only on $r$, more precisely on the value of the constants $\underline{b}$ and $\oline b$.

Moreover, since for $\beta = \beta(\k) = \sqrt{\k/1-\k}$ we have 
    \begin{align} \label{eq:u-v-w}
    \vu_\ep= \vv_\ep-\beta \vw_\ep \qquad \text{ and } \qquad \vU= \vV- \beta \vW,
    \end{align} 
there exists a constant $C(\beta)>0$ such that 
    \begin{align}\label{est-k-1}
         \vr_\ep |\vu_\ep -\vU |^2 \leq  C(\beta)\; \vr_\ep \Big(|\vv_\ep-\vV|^2+|\vw_\ep-\vW|^2  \Big) .
    \end{align}

Following a nowadays classical approach (see \tsl{e.g.} book \cite{FN09}), we introduce the essential and residual parts of a given function. 
Let us fix a smooth $\psi \in C_c^\infty(0,\infty)$, with $0\leq \psi\leq 1$ and $\psi(r)=1$ for all $r\in \left[ \underline{b}/2, 2\overline{b}\right]$.
Then, for any measurable function $h$, we define 
\begin{align*}  
	h=[h]_{\text{ess}} + [h]_{\text{res}},\qquad \mbox{ with } \qquad  [h]_{\text{ess}} =\psi(\vr_\ep) h,\quad [h]_{\text{res}}= (1- \psi(\vr_\ep)) h.
\end{align*}
Therefore,  for any $\t\in[0,T]$ and any $\ep\in (0,1)$ we have:
\begin{align}
\nonumber & 
 \left\| \left[ \f{\vr_{\ep} -r }{\ep} \right]_{\rm ess} (\t,\cdot)  \right\|^2_{L^2(\Om)} \leq \mathscr{E}_\ep(\tau),  \\
 & \int_{\Om} 
 \left[  1+\vr_{\ep}^{\gamma} \right]_{\rm  res} (\t,\cdot) \dx \leq C \ep^2 \ \mathscr{E}_\ep(\tau), \label{est:coerc-E} \\
\nonumber & 
 \left\| \sqrt{ \vr_{\ep}} \Big(\vu_{\ep} - \vU\Big) (\t,\cdot)  \right\|^2_{L^2(\Om)}+ \left\| \sqrt{ \vr_{\ep}} \Big(\vv_{\ep} - \vV\Big) (\t,\cdot)  \right\|^2_{L^2(\Om)} + 
 \left\| \sqrt{ \vr_{\ep}} \Big(\vw_{\ep} - \vW\Big) (\t,\cdot)  \right\|^2_{L^2(\Om)} \leq C \ \mathscr{E}_\ep(\tau).  
\end{align}

\subsection{Uniform estimates}\label{SSec:uni}

In this subsection we derive more refined estimates that  will be  needed in the course of our study, in Sections \ref{sec:well} and \ref{sec:ill} below, all uniform with respect to $\mu$. To this end, we consider
\begin{equation}\label{eq:test-ill}
 (r,\vV,\vW)=
 (b,\vV_b,\vW_b),\qquad 
 \vW_b=2\sqrt{\k(1-\k)}\mu\Grad\log b,
 \qquad \vV_b=\beta\vW_b,
\end{equation}
that satisfy the relations in \eqref{h1} and hence can be used as a test function in \eqref{rel-ent-2b}. Notice that the cases $\Omega=\T^3$ and $\Omega=\R^3$ can be considered simultaneously, and the difference happens only at the end, when convergences are deduced.
However, keep in mind that, although we use the same notation $b$ for both domains, when  $\Omega=\R^3$ the profile $b$ is assumed not only to solve \eqref{const-b-1_bis} and \eqref{const-b-2}, but also to satisfy the far-field condition $b(x) \to \oline{\vr}$ as $|x| \to \infty$, for some constant value $\oline\vr > 0$. Therefore, $\displaystyle \Ov{b} = \sup_{x \in \Om} b(x)$ and $\displaystyle \underline{b} = \inf_{x \in \Om} b(x)$ may differ depending on the domain.

To begin, let us notice that, from equation \eqref{eq:stat-b} it follows that
\begin{equation} \label{eq:b-second}
 \nabla P'(b) = \nabla G\, \quad \mbox{for}\ \  P(b) = \frac{a}{\g-1} b^{\g}.
\end{equation}
Thus, we can immediately reduce \eqref{rel-ent-2b} with $ (r, \vV, \vW) = (b, \vV_b,\vW_b) $ into the following form:
\begin{align}\label{kapp-eng-2c}
	\begin{split}
		& \mathscr{E}( \vr_\ep,\vv_\ep,\vw_\ep| b,\vV_b,\vW_b)(\tau)  -\mathscr{E}( \vr_{0,\ep},\vv_{0,\ep},\vw_{0,\ep}| b,\vV_b,\vW_b)  \\
		& \quad \quad + 
        2 \kappa \mu \int_0^{\tau}  \!\!\!\int_{\Om}  \left|\oline{\sqrt{\vr_\ep}  \,\ba \vv_\ep }  \right|^2 \dx\ds
        + 
		2\mu \int_0^{\tau}  \!\!\!\int_{\Om} \left|\oline{ \sqrt{\vr_\ep} \, \bd \left( \sqrt{1-\kappa}\vv_\ep-\sqrt{\kappa}\vw_{\ep}  \right) }  \right|^2 \dx\ds   \\
		& \quad \quad + 2\f{\kappa \mu}{\ep^2} \int_0^{\tau}  \!\!\!\int_{\Om}  \vr_\ep  p'(\vr_\ep) \vert \Grad \log \vr_\ep-\Grad \log b \vert^2      \dx\ds \\ 
		& \quad 
        \leq 
C \mu 
\int_0^\tau  \mathscr{E}( \vr_{\ep},\vv_{\ep},\vw_{\ep}| b,\vV_b,\vW_b)(s) \ds
        \\
        & \quad \quad 
        +
        2\f{\kappa \mu}{\ep^2} \int_0^{\tau}  \!\!\!\int_{\Om}  \Big( p(\vr_\ep)-p(b) -  p'(b)(\vr_\ep-b) \Big) \cdot \Delta \log b    \dx\ds \\ 
		&\quad \quad + 2\f{\kappa \mu}{\ep^2} \int_0^{\tau}  \!\!\!\int_{\Om}   \Big( \vr_\ep \big(p^\prime(\vr_\ep) - p^\prime(b)\big)- p^{\prime \prime}(b) \ b \ (\vr_\ep-b) \Big)
		\vert \nabla \log b \vert^2     \dx\ds. \\ 
	\end{split}
\end{align}
Since $\vU_b=\mathbf{0}$ and $\|\Grad(\vV_b,\vW_b)\|_{L^{\infty}}\le C\mu$, \eqref{est-k-1} enables us to bound the two transport terms by 
\begin{align*}
C \mu 
\int_0^\tau  \mathscr{E}( \vr_{\ep},\vv_{\ep},\vw_{\ep}| b,\vV_b,\vW_b)(s) \ds.
\end{align*}
Direct calculation gives
\begin{equation} \label{est:press-remainder}
 \Big( \vr_{\ep} \big(p^\prime(\vr_{\ep}) - p^\prime(b)\big)-
		p^{\prime \prime}(b) \ b \ (\vr_{\ep}-b) \Big) 
        = \gamma (\gamma-1)
 \Big( P(\vr_{\ep})-P(b) -  P'(b)(\vr_{\ep}-b) \Big).
\end{equation}
Therefore, we can estimate
\begin{align}\label{est-k-2a}
	\begin{split}
		&2\f{\kappa \mu}{\ep^2} \int_0^{\tau}  \!\!\!\int_{\Om}  \Big( p(\vr_{\ep})-p(b) -  p'(b)(\vr_{\ep}-b) \Big) \cdot \Delta \log b    \dx\ds \\ 
		&\qquad 
        + 2\f{\kappa \mu}{\ep^2} \int_0^{\tau}  \!\!\!\int_{\Om}   \Big( \vr_{\ep} \big(p^\prime(\vr_{\ep}) - p^\prime(b)\big)-
		p^{\prime \prime}(b) \ b \ (\vr_{\ep}-b) \Big) \vert \nabla \log b \vert^2     \dx\ds \\
		& \quad 
        \leq  C
		 \frac{\mu}{\ep^2}  \int_0^{\tau}  \!\!\!\int_{\Om}  \Big( P(\vr_{\ep})-P(b) -  P'(b)(\vr_{\ep}-b) \Big) \dx \ds\\
         & \quad 
         \leq
         C \mu
         \int_0^\tau  \mathscr{E}( \vr_{\ep},\vv_{\ep},\vw_{\ep}| b,\vV_b,\vW_b)(s) \ds,
		\end{split}
	\end{align}
thus \eqref{kapp-eng-2c} implies
\begin{align*}
	\begin{split}
		& \mathscr{E}( \vr_{\ep},\vv_{\ep},\vw_{\ep}| b, \vV_b,\vW_b)(\tau)   \\
		& \quad \quad +   2 \kappa \mu \int_0^{\tau}  \!\!\!\int_{\Om}  \left| \oline{\sqrt{\vr_\veps}  \, \ba(\vv_{\ep})}  \right|^2 \dx\ds + 
		2\mu \int_0^{\tau}  \!\!\!\int_{\Om} \left| \oline{ \sqrt{\vr_{\ep}} \, \bd \left( \sqrt{1-\kappa}\vv_{\ep}-\sqrt{\kappa}\vw_{\ep}   \right)} \right|^2  \dx\ds       \\
		& \quad \quad + 2\f{\kappa \mu}{\ep^2} \int_0^{\tau}  \!\!\!\int_{\Om}  \vr_{\ep}  p'(\vr_{\ep}) \vert \Grad \log \vr_{\ep}-\Grad \log b \vert^2      \dx\ds \\ 
		& \quad  \leq \mathscr{E}( \vr_{0,\ep},\vv_{0,\ep},\vw_{0,\ep}| b,\vV_b,\vW_b)
        + C \mu 
        \int_0^\tau  \mathscr{E}( \vr_{\ep},\vv_{\ep},\vw_{\ep}| b,\vV_b,\vW_b)(s) \ds . 
	\end{split}
\end{align*}
From the assumptions on initial data, we have 
\begin{align*}
	 \mathscr{E}( \vr_{0, \ep},\vv_{0,\ep},\vw_{0,\ep}| b,\vV_b,\vW_b)
     \leq C,
\end{align*} 
where $ C>0 $ is independent of $ \ep $. Therefore, for any $T>0$ fixed, the bound
\begin{align} \label{id-hyp2}
\mathscr{E}( \vr_{\ep},\vv_{\ep},\vw_{\ep}| b,\vV_b,\vW_b)(\tau) \leq C\exp(C T) 
\quad \text{ for a.e. } \tau \in (0,T)
\end{align}
is uniform with respect to $\ep.$ In particular, we deduce the following uniform bounds 
\eq{\label{bound:vw}
 \sup_{t\in (0,T)} \| \sqrt{ \vr_{\ep}} \vv_{\ep} (t,\cdot)  \|_{L^2(\Om)} +  \sup_{t\in (0,T)} \| \sqrt{ \vr_{\ep}} \vw_{\ep} (t,\cdot)  \|_{L^2(\Om)} \leq C.
}

Moreover,  as in \eqref{est:coerc-E}, we can now deduce from \eqref{id-hyp2} the following  uniform bounds on $\big\{(\vr_\veps, \vu_\veps)\big\}_{\veps>0}$:
\eq{\label{bound:rhou}
&  \sup_{t\in (0,T)}
 \left\| \left[ \f{\vr_{\ep} -b }{\ep} \right]_{\rm ess} (t,\cdot)  \right\|_{L^2(\Om)} \leq C, 
 \quad 
 \sup_{t\in (0,T)}
 \int_{\Om} 
 \left[  1+\vr_{\ep}^{\gamma} \right]_{\rm  res} (t,\cdot) \dx \leq C \ep^2, \\
 &  \sup_{t\in (0,T)}
 \| \sqrt{ \vr_{\ep}} \vu_{\ep} (t,\cdot)  \|_{L^2(\Om)} \leq C, 
}
for some $C>0$ independent of $\ep$. On the torus $\Omega=\mathbb{T}^3$, by the decomposition  $\vr_\ep\vu_\ep = \sqrt{\vr_\ep} \sqrt{\vr_\ep}\vu_\ep$, we also see that
\[
\sup_{t\in (0,T)} \| \vr_{\ep} \vu_{\ep} (t,\cdot)  \|_{L^q(\Om)} \leq C
 \qquad \text{ with }\quad  q =  \f{ 2 \min\{2,\gamma  \}    }{  \min\{2,\gamma  \} +1   } \in (1,2). 
\]

As a direct consequence of the uniform estimates above, up to an omitted extraction of subsequence
we deduce (see \tsl{e.g.} \cite{FZ23} for details) that, for $\Omega=\T^3$, we have
\begin{align} \label{eq:conv_prop}
\begin{split}
& \vr_{\ep} \rightarrow b \qquad \text{ strongly in } \quad
L^{\infty}\big(0,T; L^{ \min\{2,\gamma\} }(\Om)  \big), \\
& \sqrt{\vr_\ep}\vu_\ep \stackrel{*}{\rightharpoonup} \wtilde{\vu} \qquad \text{ weakly-$*$ in } \quad L^\infty\big(0,T; L^2(\Omega;\R^3)\big), \\
& \vr_\ep \vu_\ep \stackrel{*}{\rightharpoonup} \vm = \sqrt{b}\wtilde{\vu} \qquad \text{ weakly-$*$ in } \quad L^\infty\big(0,T; L^q(\Omega;\R^3)\big),
\end{split}
\end{align}
while for $\Omega=\R^3$ we have instead
\begin{align} \label{eq:conv_prop:ill-2}
\begin{split}
& \vr_{\ep} -b \rightarrow 0 \qquad \text{strongly in } \quad
L^{\infty}\big(0,T; L^2 + L^{\gamma} (\Om)  \big), \\
& \sqrt{\vr_\ep}\vu_\ep \stackrel{*}{\rightharpoonup} \wtilde{\vu} \qquad \text{weakly-$*$ in } \quad L^\infty\big(0,T; L^2(\Omega;\R^3)\big), \\
& \vr_\ep \vu_\ep \stackrel{*}{\rightharpoonup} \vm = \sqrt{b}\wtilde{\vu} \qquad \text{weakly-$*$ in } \quad L^\infty\big(0,T; L^2 + L^{\frac{2\gamma}{\gamma+1}}(\Omega;\R^3)\big).
\end{split}
\end{align}
From now on, we omit to specify that weak convergence properties are verified up to an extraction of a suitable subsequence.
By using the convergence properties from \eqref{eq:conv_prop} and \eqref{eq:conv_prop:ill-2}, it is a routine matter to pass to the limit in the weak form of the conservation of mass, that is in relation \eqref{equ-conti}, and find that
\begin{equation} \label{eq:constr-well}
\Div \, \vm = \Div\big( \sqrt{b} \wtilde\vu\big) = 0\qquad\qquad \mbox{ in the sense of distributions on }\quad (0,T)\times\Omega.
\end{equation}

From \eqref{bound:vw} we may also deduce that there exist two vector fields $\wtilde \vv$ and $\wtilde \vw$ belonging to $L^\infty\big(0,T;L^2(\Omega)\big)$ such that
\eq{\label{vw:conv}
 \sqrt{\vr_\ep}\vv_\ep \stackrel{*}{\rightharpoonup} \wtilde\vv \qquad \mbox{ and }\qquad \sqrt{\vr_\ep}\vw_\ep \stackrel{*}{\rightharpoonup} \wtilde\vw
}
in the weak-$*$ topology of that space. 
Now, due to the equality $\sqrt{\vr_\ep}\vw_\ep = 4\mu\sqrt{\k(1-\k)} \Grad\sqrt{\vr_\ep}$,
and to the first convergence property in \eqref{eq:conv_prop} or \eqref{eq:conv_prop:ill-2}, we find that
\eq{\label{vw:conv2}
 \wtilde\vw = 4 \big(\lim_{\ep\to0}\mu_\ep\big)
 \sqrt{\k (1-\k)}   \, \nabla \sqrt{b},
}
see also Remark \ref{r:w}.

\section{Singular limit for well-prepared initial data}\label{sec:well}

In this section, we give the proof of Theorem \ref{TH1}, devoted to the convergence result in the case of well-prepared initial data.
In particular, throughout this section we will work in the domain
\begin{align*}
\Omega = \T^3.
\end{align*}
Having derived a series of uniform bounds for the family $\big\{(\vr_\veps, \vv_\veps, \vw_\veps)\big\}_{\veps>0}$ in the previous section, we see that we are still lacking information about strong convergence of any of the velocities postulated in \eqref{conv:well}. 
This will be done by employing the relative entropy inequality of Section \ref{sec:REI} in a suitable way.

\subsection{Relative entropy inequality for well-prepared data}\label{subs:well}
To proceed, we take the test function $ (r, \vV, \vW) $ in the relative entropy inequality \eqref{rel-ent-2b} to be equal to $ (b, \vV,\vW) $,
where $ (b, \vU=\vV-\beta \vW) $ given by Theorem \ref{thm:LWP} solves the target system \eqref{ane-app}. Therefore, the following identities hold true:
\eq{\label{target_new}
	&\partial_t b+ \Div(b \vU)=0,\qquad  \vW= 2\mu \sqrt{\kappa(1-\kappa)} \nabla \log b,\\
	&  \p_t \vV + \vU \cdot \Grad \vV +\Grad \Pi = b^{-1}\left[
		\mu \Div (2 b (1-\kappa)\bd \vV ) + \mu \Div (2\kappa b \ba \vV)      -\mu \Div \Big(   2\sqrt{\kappa (1-\kappa)} b \Grad \vW  \Big)      \right],\\
	&  \vU \cdot \Grad \vW =b^{-1} \left[ \mu \Div (2\kappa b \Grad \vW)
	-\mu \Div \Big( 2\sqrt{\kappa (1-\kappa)}   b \Grad^{t}\vV \Big) \right].
}

We see that the relative  entropy inequality  \eqref{rel-ent-2b} applied to the new test functions yields
\begin{align}
\nonumber	& \mathscr{E}( \vr_\ep,\vv_\ep,\vw_\ep| b,\vV,\vW)(\tau)-\mathscr{E}( \vr_{0,\ep},\vv_{0,\ep},\vw_{0,\ep}| b,\vV_0,\vW_0)    \\
\nonumber	& \quad \quad +   2 \kappa \mu \int_0^{\tau}  \!\!\!\int_{\Om}  \left|\oline{ \sqrt{\vr_\ep}  \, \ba(\vv_\ep-\vV) } \right|^2 \dx\ds + 
2\mu \int_0^{\tau}  \!\!\!\int_{\Om} \left| \oline{\sqrt{\vr_\ep}  \, \bd \left( \sqrt{1-\kappa}(\vv_\ep-\vV)-\sqrt{\kappa}(\vw_\ep-\vW)   \right)} \right|^2   \dx\ds             \\
\nonumber	& \quad \quad 
+ 2\f{\kappa \mu}{\ep^2} \int_0^{\tau}  \!\!\!\int_{\Om}  \vr_{\ep}  p'(\vr_{\ep}) \vert \Grad \log \vr_{\ep}-\Grad \log b \vert^2       \dx\ds \\ 
\nonumber	& \quad  \leq \int_0^{\tau}  \!\!\!\int_{\Om}\vr_\ep \left(  \vu_\ep -\vU  \right)\cdot \Grad \vW \cdot (\vW-\vw_\ep)\dx\ds 
	+\int_0^{\tau}  \!\!\!\int_{\Om}\vr_\ep  \left(  \vu_\ep -\vU  \right)  \cdot \Grad \vV \cdot (\vV-\vv_\ep)            \dx\ds\\
\nonumber	& \quad \quad + \int_0^{\tau}  \!\!\!\int_{\Om}\vr_\ep (\vW-\vw_\ep) \cdot\lr{ \p_t \vW+ \vU \cdot \nabla \vW} \dx \ds + \int_0^{\tau}  \!\!\!\int_{\Om} \vr_\ep  
 (\vV-\vv_\ep) \cdot
 \lr{ \p_t \vV+\vU \cdot \nabla \vV}   \dx\ds \\
\nonumber	& \quad \quad -\f{1}{\ep^2} \int_0^{\tau}  \!\!\!\int_{\Om}\Big( p(\vr_\ep)-p(b) -  p'(b)(\vr_\ep-b) \Big) \Div \vU\dx\ds \\
\nonumber	&  \quad \quad + \f{1}{\ep^2} \int_0^{\tau}  \!\!\!\int_{\Om}\vr_\ep \Big(\Grad G- \Grad P^\prime(b)\Big) \cdot (\vv_\ep-\vV) \dx\ds\\
\nonumber	& \quad \quad + 2\mu \int_0^{\tau}  \!\!\!\int_{\Om} \sqrt{\vr_\ep} \Big(  \bd(\sqrt{1-\kappa} \vV)-\bd( \sqrt{\ka} \vW)     \Big):
\oline{ \sqrt{\vr_\ep}\Big(  \bd(\sqrt{1-\kappa} (\vV-\vv_\ep))-\bd( \sqrt{\ka} (\vW-\vw_\ep))     \Big) }      \dx\ds \\
\nonumber	&  \quad \quad + 2\ka \mu \int_0^{\tau}  \!\!\!\int_{\Om} \sqrt{\vr_\ep}  \, \ba  \vV: \oline{\sqrt{\vr_\ep}  \, \ba(\vV-\vv_{\ep})} \dx\ds\\
	\nonumber	& \quad \quad + 2\f{\kappa \mu}{\ep^2} \int_0^{\tau}  \!\!\!\int_{\Om} 
    \Big( p(\vre)-p(b) -  p'(b)(\vre-b) \Big) \Delta \log b            \dx\ds \\ 
\nonumber	& \quad \quad + 2\f{\kappa \mu}{\ep^2} \int_0^{\tau}  \!\!\!\int_{\Om}\Big( \vre \big(p^\prime(\vre) - p^\prime(b)\big)- 
	p^{\prime \prime}(b) \ b \ (\vre-b) \Big) \vert \nabla \log b \vert^2     \dx\ds\\
	&\quad =\sum_{j=1}^{10}  R_j.
	\label{rel-ent-2bis}
\end{align}

Our next goal is to bound all the terms $R_j$. In the estimates below, we will make some multiplicative constants appear, which depend on suitable
functional norms of the functions $\vV$ and $\vW$. Of course, we tacitly understand that those quantities can be expressed
in terms of the target profiles $\vU$ and $b$, solving the target system \eqref{ane-app}.

First of all, notice that some terms, like the two last ones for instance,
can be treated exactly as in \eqref{est-k-2a}. More precisely, it is easy to check that 
\[
\left|R_1\right| + \left| R_2 \right| + \left| R_5 \right| + \left| R_9\right|+\left|R_{10} \right| \leq C \int^\t_0\mathscr{E}( \vr_\ep,\vv_\ep,\vw_\ep| b,\vV,\vW)(s)\,\ds\,,
\]
for a constant $C>0$ depending on $\g$, $b,\mu,\kappa$ and on the $L^\infty\big( (0,T)\times\Omega\big)$ norms of $\Grad \vV$ and $\Grad\vW$.
Moreover, due to \eqref{eq:b-second}, we have $R_6\equiv0$.

The terms $R_7$ and $R_8$ will cancel out other terms later on, so we only need to deal with $R_3$ and $R_4$.
To begin with, we deduce from \eqref{target_new} that
\eq{\label{prob_Vt}
& \int_0^{\tau}  \!\!\!\int_{\Om} \vr_{\ep} (\vV-\vv_{\ep}) \cdot\lr{ \p_t \vV+\vU \cdot \nabla \vV}   \dx\ds \\
&= - \int_0^{\tau}  \!\!\!\int_{\Om}	\vr_{\ep} (\vV-\vv_{\ep})    \cdot  \Grad \Pi     \dx\ds \\
&\quad + \int_0^{\tau}  \!\!\!\int_{\Om} \frac{\vr_{\ep}}{b} (\vV-\vv_{\ep})\cdot \Big[ 
 		\mu \Div (2 b (1-\kappa)\bd \vV ) + \mu \Div (2\kappa b \ba \vV)
 -\mu \Div \Big(   2\sqrt{\kappa (1-\kappa)} b \Grad \vW        \Big) \Big]    \dx\ds 
}
and, as $\partial_t\vW=0$, that
\eq{\label{prob_Wt}
	&	\int_0^{\tau}  \!\!\!\int_{\Om}\vr_{\ep} (\vW-\vw_{\ep}) \cdot\lr{ \p_t \vW+ \vU \cdot \nabla \vW} \dx \ds \\
	&= \int_0^{\tau}  \!\!\!\int_{\Om}\frac{\vr_{\ep}}{b} (\vW-\vw_{\ep}) \cdot \left[ \mu \Div (2\kappa b \Grad \vW)
	-\mu \Div \Big( 2\sqrt{\kappa (1-\kappa)}   b \Grad^{t}\vV \Big) \right]  \dx \ds \\
}
At this point, one would like to integrate by parts the terms appearing in the last lines of \eqref{prob_Vt} and \eqref{prob_Wt}. However, this is not possible,
as we lack of information on $\vv_\ep-\vV$ to give sense to the term $\nabla(\frac{\vr_\ep}{b}) \ (\vv_\ep-\vV)$ (and analogous problem for the $\vw_\ep-\vW$ term).
Therefore, we have to argue in a different way. Namely, we sum up equalities \eqref{prob_Vt} and \eqref{prob_Wt} to get
\begin{align*}
	& \int_0^{\tau}\!\!\!\int_{\Om}  \vr_{\ep} (\vV -\vv_{\ep}) \cdot \lr{ \p_t \vV+\vU \cdot \nabla \vV}   \dx\ds + 
	\int_0^{\tau} \!\!\!\int_{\Om} \vr_{\ep} (\vW -\vw_{\ep}) \cdot \lr{ \p_t \vW+ \vU \cdot \nabla \vW} \dx \ds \\
	& = \int_0^{\tau}\!\!\!\int_{\Om}  \vr_{\ep} (\vv_{\ep}-\vV) \cdot \Grad \Pi \dx\ds  \\
	&\quad +\int_0^{\tau}\!\!\!\int_{\Om}   \vr_{\ep} (\vv_{\ep}-\vV) \cdot b^{-1} \Big[ 
	\mu \Div (2 b (1-\kappa)\bd \vV ) + \mu \Div (2\kappa b \ba \vV) 	     -\mu \Div \Big(   2\sqrt{\kappa (1-\kappa)} b \Grad \vW        \Big) \Big]    \dx\ds \\
	&\quad +\int_0^{\tau}\!\!\! \int_{\Om} \vr_{\ep} (\vw_{\ep}-\vW)\cdot b^{-1} \left[ \mu \Div (2\kappa b \Grad \vW)
	-\mu \Div \Big( 2\sqrt{\kappa (1-\kappa)}   b \Grad^{t}\vV \Big) \right]  \dx \ds .
\end{align*}
Let us focus on the last two terms. Taking advantage of the usual relations between $\vu_\ep, \vv_{\ep}, \vw_\ep $ and $\vU, \vV, \vW$, we can write
\begin{align}\label{v-w-eqn}
	\nonumber & \int_0^{\tau}\!\!\!\int_{\Om}   \vr_{\ep} (\vv_{\ep}-\vV) \cdot b^{-1} \Big[ 
	\mu \Div (2 b (1-\kappa)\bd \vV ) + \mu \Div (2\kappa b \ba \vV) 	     -\mu \Div \Big(   2\sqrt{\kappa (1-\kappa)} b \Grad \vW        \Big) \Big]    \dx\ds \\
\nonumber	&\qquad+ \int_0^{\tau} \!\!\!\int_{\Om} \vr_{\ep} (\vw_{\ep}-\vW) \cdot b^{-1} \left[ \mu \Div (2\kappa b \Grad \vW)
	-\mu \Div \Big( 2\sqrt{\kappa (1-\kappa)}   b \Grad^{t}\vV \Big) \right]  \dx \ds \\
\nonumber	& \quad 
    =
    \int_0^{\tau}\!\!\!\int_{\Om}   \vr_{\ep} (\vu_{\ep}-\vU) \cdot b^{-1} \Big[ 
	\mu \Div (2 b (1-\kappa)\bd \vV ) 	     -\mu \Div \Big(   2\sqrt{\kappa (1-\kappa)} b \Grad \vW        \Big) \Big]    \dx\ds  \\ 
\nonumber	&\qquad+\int_0^{\tau}\!\!\! \int_{\Om} \vr_{\ep} (\vv_{\ep}-\vV) \cdot b^{-1}  \mu \Div (2\kappa b \ba \vV) \dx\dt   \\
\nonumber	&\qquad + \int_0^{\tau} \!\!\!\int_{\Om} \vr_{\ep} (\vw_{\ep}-\vW) \cdot b^{-1} \left[ \mu \Div \Big( 2\sqrt{\kappa (1-\kappa)}   b \ba \vV \Big) \right]  \dx \ds \\
\nonumber	&\quad 
    =
    \int_0^{\tau}\!\!\!\int_{\Om}   \vr_{\ep} (\vu_{\ep}-\vU)  \cdot \Big[ 
	\mu \Div (2  (1-\kappa)\bd \vV ) 	     -\mu \Div \Big(   2\sqrt{\kappa (1-\kappa)}  \Grad \vW        \Big) \Big]    \dx\ds  \\ 
 \nonumber   &\qquad+\int_0^{\tau} \!\!\!\int_{\Om} \vr_{\ep} (\vv_{\ep}-\vV) \cdot   \mu \Div (2\kappa  \ba \vV) \dx\dt  \\
  \nonumber  &\qquad + \int_0^{\tau}\!\!\! \int_{\Om} \vr_{\ep} (\vw_{\ep}-\vW) \cdot  \left[ \mu \Div \Big( 2\sqrt{\kappa (1-\kappa)}    \ba \vV \Big) \right]  \dx \ds \\
	\nonumber&\qquad+ \int_0^{\tau}\!\!\!\int_{\Om}  \nabla \log b \otimes[ \vr_{\ep} (\vu_{\ep}-\vU)] :\Big[ 
	2\mu   (1-\kappa)\bd \vV  	     -2\mu \Big(   \sqrt{\kappa (1-\kappa)}  \Grad \vW        \Big) \Big]    \dx\ds  \\ 
	\nonumber &\qquad {-}\int_0^{\tau} \!\!\!\int_{\Om}\nabla \log b\otimes [\vr_{\ep} (\vv_{\ep}-\vV)] : {2\mu\kappa}\,\ba\vV \dx\dt   \\
\nonumber	&\qquad {-}\int_0^{\tau} \!\!\!\int_{\Om} \nabla \log b \otimes[\vr_{\ep} (\vw_{\ep}-\vW)] :\left[ \mu  \Big( 2\sqrt{\kappa (1-\kappa)}    \ba \vV \Big) \right]  \dx \ds \\
	& \quad
    =
    \sum_{i=1}^{6} Z_i .
\end{align}
Now, we will consider the first three terms only, namely $Z_1$, $Z_2$ and $Z_3$.
To treat $Z_1$, we notice that
\begin{align}\label{z-rel}
	&\mu \Div (2  (1-\kappa)\bd \vV ) 	     -\mu \Div \Big(   2\sqrt{\kappa (1-\kappa)}  \Grad \vW        \Big) 
	=\mu (1-\kappa) \lr{\Delta \vV+ \nabla\Div \vV} -  2\mu \sqrt{\kappa (1-\kappa)}\Delta \vW. 
\end{align}
Therefore, applying formula \eqref{weak-id1} yields that
\begin{align}\label{z1}
\begin{split}
Z_1& =
-4(1-\kappa)\mu \int_0^{\tau} \!\!\!\int_{\Om}\lr{\nabla \sqrt{\vr_\ep } \otimes \sqrt{\vr_\ep}(\vu_\ep-\vU)} : \bd\vV \dx\ds-2\mu(1-\kappa)\langle\sqrt{\vr_\ep}\oline{\sqrt{\vr_\ep}\bd(\vu_\ep-\vU)},\bd\vV\rangle\\
& \quad 
+4\sqrt{\kappa(1-\kappa)}\mu \int_0^{\tau} \!\!\!\int_{\Om}\lr{\nabla \sqrt{\vr_\ep } \otimes \sqrt{\vr_\ep}(\vu_\ep-\vU)} : \bd\vW \dx\ds \\
& \quad 
+ 2\mu\sqrt{\kappa(1-\kappa)}\langle\sqrt{\vr_\ep}\oline{\sqrt{\vr_\ep}\bd(\vu_\ep-\vU)},\bd\vW\rangle\\
& =
-2\mu\int_0^{\tau}\!\!\!\int_{\Om}   \Grad\log\vr_\ep\otimes\vr_{\ep} (\vu_{\ep}-\vU):  \Big[ 
	 (1-\kappa)\bd \vV 	     -   \sqrt{\kappa (1-\kappa)}  \bd \vW         \Big]    \dx\ds \\
     & \quad 
     {+} 2 \mu \int_0^{\tau} \!\!\!\int_{\Om}  \Ov{ \sqrt{\vr_{\ep}} \Big(  \bd(\sqrt{1-\kappa} (\vV-\vv_{\ep}))-\sqrt{\ka}\bd  (\vW-\vw_{\ep})     \Big)   } :
     \sqrt{\vr_{\ep}} \Big(  \bd(\sqrt{1-\kappa} \vV)-\sqrt{\ka}\bd \vW     \Big)   \dx\ds\\
     &
     =: 
Z_{1,a}+Z_{1,b}  ,    
\end{split}
\end{align}
where we have made use of the identity $\oline{\sqrt{\vr_\ep} \,\bd \vw_\ep}=\oline{\sqrt{\vr_\ep} \, \Grad\vw_\ep}$ (keep in mind the constraint \eqref{def2:prop}).

For $Z_2$, after noticing that $\Div(\ba\vV)=\frac12(\Delta\vV-\Grad\Div\vV)$, we can use formula \eqref{weak-id2} for $(\vv_\ep-\vV)$ in place of $\vu_\ep$ to obtain
\eq{\label{z2}
Z_2& =
{+}4\mu \kappa\int_0^{\tau} \!\!\!\int_{\Om}\lr{\nabla \sqrt{\vr_\ep } \otimes \sqrt{\vr_\ep}(\vv_\ep-\vV)} : \ba\vV \dx\ds-2\mu\kappa\langle\sqrt{\vr_\ep}
\oline{\sqrt{\vr_\ep}\ba(\vv_\ep-\vV)},\ba\vV\rangle\\
&=
{+}2\mu\kappa\int_0^{\tau} \!\!\!\int_{\Om}\lr{\nabla \log\vr_\ep  \otimes {\vr_\ep}(\vv_\ep-\vV)} : \ba\vV \dx\ds-2\mu\kappa\int_0^{\tau} \!\!\!\int_{\Om}\sqrt{\vr}\Ov{\sqrt{\vr}\ba(\vv_\ep-\vV)}:\ba\vV \dx\ds\\
& =: 
Z_{2,a}+ Z_{2,b} . 
}
The last term $Z_3$ is treated in the same way as ${Z_2}$: 
\begin{align}\label{z3}
\begin{split}
Z_3  & =
{+}4\mu\sqrt{(1-\kappa)\kappa} \int_0^{\tau} \!\!\!\int_{\Om}\lr{\nabla \sqrt{\vr_\ep } \otimes \sqrt{\vr_\ep}(\vw_\ep-\vW)} : \ba\vV \dx\ds \\
&\qquad  \qquad 
-2\mu\sqrt{(1-\kappa)\kappa} \langle\sqrt{\vr_\ep}\oline{\sqrt{\vr_\ep}\ba(\vw_\ep-\vW)},\ba\vV\rangle\\
&=
{+}2\mu\sqrt{(1-\kappa)\kappa} \int_0^{\tau} \!\!\!\int_{\Om}\lr{\nabla \log{\vr_\ep } \otimes {\vr_\ep}(\vw_\ep-\vW)} : \ba\vV \dx\ds,
\end{split}
\end{align}
where the term $-2\mu\sqrt{(1-\kappa)\kappa} \langle\sqrt{\vr_\ep}\oline{\sqrt{\vr_\ep}\ba(\vw_\ep-\vW)},\ba\vV\rangle$ above disappears thanks to the relation $\oline{\sqrt\vr_\ep \ba\vw_\ep}=0$ (which holds true due to the constraint \eqref{def2:prop} again).

At this point, we observe that the terms $Z_{1,b}$ and $Z_{2,b}$ exactly cancel out the terms $R_7$ and $R_8$ appearing in \eqref{rel-ent-2bis} above.
On the other hand, if we define $Z := Z_{1,a} + Z_{2,a} + Z_{3} + Z_4 + Z_5 + Z_6$, we can write
\begin{align}\label{z}
\nonumber Z
& = {-}2\mu\int_0^{\tau}\!\!\!\int_{\Om}  \nabla \log b \otimes\vr_{\ep} (\vu_{\ep}-\vU): \Big[ 
	  (1-\kappa)\bd \vV  	     -    \sqrt{\kappa (1-\kappa)}  \Grad \vW         \Big]    \dx\ds  \\ 
\nonumber	 &\quad   +
     2\mu\kappa\int_0^{\tau} \!\!\!\int_{\Om} \nabla \log b \otimes\vr_{\ep} (\vv_{\ep}-\vV) :     \ba \vV \dx\dt   \\
\nonumber	& \quad +  
    \int_0^{\tau} \!\!\!\int_{\Om} \Grad\log b\otimes \vr_{\ep} (\vw_{\ep}-\vW) :\left[ \mu  \Big( 2\sqrt{\kappa (1-\kappa)}    \ba \vV \Big) \right]  \dx \ds \\
\nonumber    & \quad 
    {+}2\mu\int_0^{\tau}\!\!\!\int_{\Om}   \Grad\log\vr_\ep\otimes\vr_{\ep} (\vu_{\ep}-\vU):  \Big[ 
	 (1-\kappa)\bd \vV 	     -   \sqrt{\kappa (1-\kappa)}  \Grad \vW         \Big]    \dx\ds \\
\nonumber     & \quad 
     -2\mu\kappa\int_0^{\tau} \!\!\!\int_{\Om}{\nabla \log\vr_\ep  \otimes {\vr_\ep}(\vv_\ep-\vV)} : \ba\vV \dx\ds\\ 
 \nonumber   &  \quad 
    -2\mu\sqrt{(1-\kappa)\kappa} \int_0^{\tau} \!\!\!\int_{\Om}{\nabla \log{\vr_\ep } \otimes {\vr_\ep}(\vw_\ep-\vW)} : \ba\vV \dx\ds \\ 
 \nonumber   & =
    {-}2\mu\int_0^{\tau}\!\!\!\int_{\Om}  \lr{\nabla \log b-\nabla\log\vr_\ep} \otimes\vr_{\ep} (\vu_{\ep}-\vU): \Big[ 
	  (1-\kappa)\bd \vV  	     -    \sqrt{\kappa (1-\kappa)}  \Grad \vW         \Big]    \dx\ds  \\
 \nonumber      & \quad 
       +2\mu\kappa\int_0^{\tau} \!\!\!\int_{\Om} \lr{\nabla \log b-\nabla\log\vr_\ep} \otimes\vr_{\ep} (\vv_{\ep}-\vV) :     \ba \vV \dx\dt   \\
	& \quad 
    +  2\mu\sqrt{\kappa (1-\kappa)}  \int_0^{\tau}  \!\!\!\int_{\Om}\lr{\nabla \log b-\nabla\log\vr_\ep}\otimes \vr_{\ep} (\vw_{\ep}-\vW) :   \ba \vV   \dx \ds .
\end{align}

Therefore, \eqref{rel-ent-2bis} reduces to the following inequality:
\begin{align}
\nonumber	& \mathscr{E}( \vr_\ep,\vv_\ep,\vw_\ep| b,\vV,\vW)(\tau)-\mathscr{E}( \vr_{0,\ep},\vv_{0,\ep},\vw_{0,\ep}| b,\vV_0,\vW_0)    \\
\nonumber	& \qquad +   2 \kappa \mu \int_0^{\tau}  \!\!\!\int_{\Om}  \left|\oline{ \sqrt{\vr_\ep}  \, \ba(\vv_\ep-\vV) } \right|^2 \dx\ds  \\
\nonumber & \qquad +2\mu \int_0^{\tau}  \!\!\!\int_{\Om} \left| \oline{\sqrt{\vr_\ep}  \, \bd \left( \sqrt{1-\kappa}(\vv_\ep-\vV)-\sqrt{\kappa}(\vw_\ep-\vW)   \right)} \right|^2   \dx\ds             \\
\nonumber	& \qquad 
+ 2\f{\kappa \mu}{\ep^2} \int_0^{\tau}  \!\!\!\int_{\Om}  \vr_{\ep}  p'(\vr_{\ep}) \vert \Grad \log \vr_{\ep}-\Grad \log b \vert^2       \dx\ds \\ 	& \quad   \leq C\int^\t_0\mathscr{E}( \vr_\ep,\vv_\ep,\vw_\ep| b,\vV,\vW)(s)\ds +\sum_{j=1}^4\mathcal{A}_j ,
 	\label{rel-ent-3}
\end{align}
where we have defined
\eqh{
 \sum_{j=1}^4\mathcal{A}_j =     &  \int_0^{\tau}\!\!\!\int_{\Om}  \vr_{\ep} (\vv_{\ep}-\vV) \cdot  \Grad \Pi \dx\ds \\
& + 2\mu\int_0^{\tau}\!\!\!\int_{\Om}  \lr{\nabla \log b-\nabla\log\vr_\ep} \otimes\vr_{\ep} (\vu_{\ep}-\vU): 
\Big[ (1-\kappa)\bd \vV  	     -    \sqrt{\kappa (1-\kappa)}  \Grad \vW         \Big]    \dx\ds  \\
   &    +2\mu\kappa\int_0^{\tau} \!\!\!\int_{\Om} \lr{\nabla \log b-\nabla\log\vr_\ep} \otimes\vr_{\ep} (\vv_{\ep}-\vV) :     \ba \vV \dx\dt   \\
	& +  2\mu\sqrt{\kappa (1-\kappa)}  \int_0^{\tau}  \!\!\!\int_{\Om}\lr{\nabla \log b-\nabla\log\vr_\ep}\otimes \vr_{\ep} (\vw_{\ep}-\vW) :   \ba \vV   \dx \ds .
 }
\subsection{Estimate of the right-hand side of the relative \texorpdfstring{$\kappa$}{}-entropy inequality}\label{Sec:5.2}
We have to bound all the $\mathcal{A}_j$'s in terms of the relative $\k$-entropy inequality \eqref{rel-ent-3}.

We start by considering $\mc A_2 + \mc A_3 + \mc A_4$.
The crucial observation here is that, due to Remark \ref{r:w}, we have the equality
\begin{align*}
\sqrt{\vr_\ep}  \, \nabla \log \vr_{\ep} = 2\nabla\sqrt{\vr_\ep} = \frac{1}{2\mu\sqrt{\k (1-\k)}} \sqrt{\vr_\ep}\vw_\ep.
\end{align*}
Therefore, we deduce that
\[
 \sqrt{\vr_\ep} \Big(\nabla\log\vr_\ep - \nabla\log b\Big) = \frac{1}{2\mu\sqrt{\k (1-\k)}} \sqrt{\vr_\ep}\Big(\vw_\ep- \vW\Big),
\]
whence we can easily estimate, for a suitable constant $C>0$, those terms as
\begin{align*}
\left|\mc A_2 \right| + \left|\mc A_3 \right| + \left|\mc A_4 \right|\leq	  C  \int_0^{\tau}  \mathscr{E}( \vr_{\ep},\vv_{\ep},\vw_{\ep}| b,\vV,\vW)(s)  \ds.
\end{align*}
Thus, in order to conclude the convergence towards the target profiles, it remains to find a suitable bound for $ \mc A_1$.
The argument for doing that is based on the weak convergence properties of the family $\big\{(\vr_\ep, \vv_\ep, \vw_\ep)\big\}_{\ep>0}$.
To begin with, let us split $\mc A_1$ as follows:
\begin{align*}
 \mc A_1 &= \int_0^{\tau}\!\!\!\int_{\Om}  \vr_{\ep} (\vv_{\ep}-\vV) \cdot \Grad \Pi \dx\ds \\
 &=\int_0^{\tau}  \!\!\!\int_{\Om}\vr_{\ep} (\vu_{\ep}-\vU) \cdot\nabla \Pi  \dx\ds + 
 \beta \int_0^{\tau}  \!\!\!\int_{\Om}\vr_{\ep} (\vw_{\ep}- \vW) \cdot \nabla \Pi  \dx\ds  \\
		&=\int_0^{\tau}  \!\!\!\int_{\Om} (\vr_{\ep} \vu_{\ep}-\vr_{\ep} \vU )\cdot  \nabla \Pi  \dx\ds + 
		\beta \int_0^{\tau}  \!\!\!\int_{\Om} \Big((\vr_{\ep} \vw_{\ep} - b\vW) - (\vr_{\ep}-b)\vW \Big) \cdot \nabla \Pi  \dx\ds .
\end{align*}
Let us introduce the following useful notation: we denote by $R_\ep(t)$ any ``remainder term'', namely any term having the property that
\begin{align}\label{eq:remainder}
\text{for all} \quad 
0<T<T_\ast,\qquad \lim_{\ep\to0} {\operatorname*{ess\,sup}_{t\in(0,T)} |R_\ep(t)|} = 0,
\end{align}
where $T_\ast>0$ is the time of existence of the smooth solution to the target problem \eqref{ane-app}.

Owing to the strong convergence of $\vr_\ep$ towards $b$ claimed in \eqref{eq:conv_prop}, which could actually be quantified in terms
of powers of $\ep$ (see \tsl{e.g.} relation (24) in \cite{FZ23}), we easily see that the integral
\[
 \int_0^{\tau}  \!\!\!\int_{\Om}  (\vr_{\ep}-b)\vW \cdot \nabla \Pi  \dx\ds  = R_\ep(\t),
\]
in the sense that relation \eqref{eq:remainder} is satisfied.

Next, we focus on the difference $\vr_{\ep} \vw_{\ep} - b\vW$. Recall that from \eqref{vw:conv} and \eqref{vw:conv2} it follows that
\[
 \sqrt{\vr_\ep}\vw_\ep = 4\mu\sqrt{\k(1-\k)} \nabla\sqrt{\vr_\ep} \stackrel{*}{\rightharpoonup} 4\mu\sqrt{\k(1-\k)} \nabla\sqrt{b}
 \qquad \mbox{ in }\quad L^\infty\big(0,T;L^2(\Omega)\big).
\]
Then, using the strong convergence of $\vre$ from \eqref{eq:conv_prop}, we gather that
\[
 \vr_\ep \vw_\ep \stackrel{*}{\rightharpoonup} 4\mu\sqrt{\k(1-\k)} \sqrt{b} \nabla\sqrt{b} \qquad \mbox{ in }\quad L^\infty\big(0,T;L^q(\Omega)\big),
\]
where $q = 2 \min\{2,\g\} / (1+\min\{2,\g\})$ as defined in Subsection \ref{SSec:uni}, see after \eqref{bound:rhou}. Since we have
\[
 4\mu\sqrt{\k(1-\k)} \sqrt{b} \nabla\sqrt{b} = 2\mu\sqrt{\k(1-\k)} \nabla b = b \vW,
\]
we see that the difference  $\vr_{\ep} \vw_{\ep} - b\vW$ verifies
\[
  \vr_{\ep} \vw_{\ep} - b\vW \stackrel{*}{\rightharpoonup} 0 \qquad \mbox{ in }\quad L^\infty\big(0,T;L^q(\Omega)\big).
\]
We can thus apply Lemma 2.1 of \cite{C-F-G} to infer that also
\[
 \int_0^{\tau}  \!\!\!\int_{\Om} (\vr_{\ep} \vw_{\ep} - b\vW) \cdot \nabla \Pi  \dx\ds = R_\ep(\t)
\]
is a remainder in the sense of relation \eqref{eq:remainder}.

An absolutely analogous argument applies also to the two remaining integrals, namely
\[
 \int_0^{\tau}  \!\!\!\int_{\Om} \vr_{\ep} \vu_{\ep}\cdot  \nabla \Pi  \dx\ds \qquad \mbox{ and }\qquad
 \int_0^{\tau}  \!\!\!\int_{\Om} \vr_{\ep} \vU \cdot  \nabla \Pi  \dx\ds,
\]
and show that they are both remainders in the sense of \eqref{eq:remainder}. Indeed, for the latter term one uses the strong convergence
of $\vr_\ep$ to $b$ obtained in \eqref{eq:conv_prop}, together with the fact that $\Div(b \vU) = 0$. As for the former integral,
we use the last weak convergence property in \eqref{eq:conv_prop}, together with \eqref{eq:constr-well}.

All in all, we have shown that
\[
 \mc A_1 = R_\ep(\t)\qquad \qquad \mbox{ in the sense of relation \eqref{eq:remainder}}.
\]

Putting this bound into \eqref{rel-ent-3}, and forgetting about the viscous terms (which are of no use in our argument, as they do not allow to infer additional
regularity properties for the gradient of the velocity field) we find the estimate
\begin{align}
\nonumber	&\mathscr{E}( \vr_\ep,\vv_\ep,\vw_\ep| b,\vV,\vW)(\tau) \\
& \quad \leq  
\mathscr{E}( \vr_{0,\ep},\vv_{0,\ep},\vw_{0,\ep}| b,\vV_0,\vW_0)+C\int^\t_0\mathscr{E}( \vr_\ep,\vv_\ep,\vw_\ep| b,\vV,\vW)(s)\ds + R_\ep(\t) . 
 \label{rel-ent-final}
\end{align}
Observe that, owing to the well-preparation assumption on the initial datum $\big(\vr_{0,\ep},\vv_{0,\ep},\vw_{0,\ep}\big)$, it is easy to see that
\begin{equation} \label{conv:in-dat}
\mathscr{E}( \vr_{0,\ep},\vv_{0,\ep},\vw_{0,\ep}| b,\vV_0,\vW_0) \longrightarrow 0 \qquad \mbox{ for }\quad \ep\to0.
\end{equation}
After applying the Gr\"onwall lemma, we therefore get that, for any $0<T<T_\ast$, one has
\eqh{
\sup_{\t\in[0,T]}\mathscr{E}( \vr_\ep,\vv_\ep,\vw_\ep| b,\vV,\vW)(\tau) \longrightarrow 0 \qquad \mbox{ for }\quad \ep\to0.}
Now,  using the triangular inequality this implies
\[
 \left\|\sqrt{\vr_\ep} \vv_\ep - \sqrt{b}\vV\right\|_{L^2(\Omega)} \leq \left\|\sqrt{\vr_\ep} \vv_\ep - \sqrt{\vr_\ep}\vV\right\|_{L^2(\Omega)} +
 \left\|\big(\sqrt{\vr_\ep} - \sqrt{b}\big)\vV\right\|_{L^2(\Omega)}
\]
and similarly for $\vw_\ep$'s. Hence, we deduce the strong convergence properties claimed in \eqref{conv:well}:
\[
 \sqrt{\vr_\ep}\vv_\ep \longrightarrow \sqrt{b}\vV\quad \mbox{ and }\quad \sqrt{\vr_\ep}\vw_\ep \longrightarrow \sqrt{b}\vW
 \qquad \mbox{ in }\quad L^\infty\big(0,T;L^2(\Omega)\big).
\]
As a byproduct, using the well-known relations linking $\vu_\ep$, $\vv_\ep$ and $\vw_\ep$, we deduce that
\[
  \sqrt{\vr_\ep}\vu_\ep \longrightarrow \sqrt{b}\vU
 \quad\qquad \mbox{ in }\quad L^\infty\big(0,T;L^2(\Omega)\big),
\]
which in particular implies, together with \eqref{eq:conv_prop}, that $\vU = \wtilde \vu / \sqrt{b}$.

This concludes the proof of Theorem \ref{TH1}.

\section{Singular limit for  ill-prepared initial data}\label{sec:ill}

Here we consider the case of ill-prepared initial data; write $\mu=\mu_\ep$ throughout. Recall that $\mu_{\ep} \rightarrow 0$ as $\ep \rightarrow 0$. By relying again on the relative $\kappa$-entropy inequality, we will complete the proof
of Theorem \ref{TH2}.

Our starting point are the uniform estimates and convergences established in  Section \ref{SSec:uni}. In particular, we have the convergences \eqref{eq:conv_prop:ill-2} and identification of the limit in the continuity equation \eqref{eq:constr-well}. 
The difference now is that, unlike in the well-prepared data case, the application of the $\kappa$-entropy inequality with test functions $ (b, \vV,\vW) $,
where $ (b, \vU=\vV-\beta \vW) $  given by Theorem {\ref{thm:LWP-R3}} solves the target system \eqref{ane-app}, does not lead to desired convergence.
Indeed, the hypotheses on the initial data are not enough to guarantee the strong convergence towards the initial datum of the target system;
in other words, we have
\begin{equation*} 
\lim _{\ep\to0}\mathscr{E}( \vr_{0,\ep},\vv_{0,\ep},\vw_{0,\ep}| b,\vV_0,\vW_0)\neq 0,
\end{equation*}
which does not allow to apply the Gr\"onwall-type argument as before. At the same time, this lack of convergence generates fast oscillations in time of the solutions.
The remedy is to consider test functions augmented by solutions to a certain wave system.

\subsection{Acoustic waves  system}\label{subs:acou}
The key issue in analyzing the ill-prepared data case is to capture the dispersive behavior of solutions to the acoustic wave equations. To derive these equations formally, we denote
\begin{align*}
    s_{\ep}:= \f{\vr_{\ep}-b}{\ep} \qquad \mbox{ and }\qquad
    \mathbf{m}_{\ep}:=
    \vr_{\ep} \vu_{\ep}
\end{align*}
and reformulate \eqref{dege-ns} as 
\begin{align}\label{deg-ns-m}
\left\{\begin{aligned}
& \ep \p_t  s_{\ep}+ \Div\, \mathbf{m}_{\ep}=0,\\
&  \ep \p_t   \mathbf{m}_{\ep} 
+
\f{1}{\ep} \Big(
\Grad p(\vr_{\ep}) -\vr_{\ep}\Grad G
\Big)
=\ep 
\Big(
2 \mu_{\ep} \Div ( \vr_{\ep} \bd \vu_{\ep} )
-\Div (\vr_{\ep}  \vu_{\ep} \otimes \vu_{\ep} )
\Big)
.
\end{aligned}\right.
\end{align}
Following Lemma 4.6 in \cite{FZ23}, a direct algebraic calculation based on the relation $\Grad p(b)=b\Grad G$ gives 
\begin{align*}
    \f{1}{\ep^2} \Big(
\Grad p(\vr_{\ep}) -\vr_{\ep}\Grad G
\Big)
=
\f{1}{\ep} b \Grad ( P''(b) s_{\ep}    )
+
 \f{1}{\ep^2}  \Grad 
 \Big(
p(\vr_{\ep})-p(b)-p'(b)(\vr_{\ep}-b)
 \Big).
\end{align*}
Hence, we further rewrite \eqref{deg-ns-m} as 
\begin{align}\label{deg-ns-m-2}
\left\{\begin{aligned}
& \ep \p_t  s_{\ep}+ \Div \, \mathbf{m}_{\ep}=0,\\
&  \ep \p_t   \mathbf{m}_{\ep} 
+
b \Grad ( P''(b) s_{\ep}    )
=\ep 
\mathbf{f}_{\ep}
,
\end{aligned}\right.
\end{align}
where we have defined
\begin{align*}
  \mathbf{f}_{\ep}
  :=
2 \mu_{\ep} \Div ( \vr_{\ep} \bd \vu_{\ep} )
-
\Div (\vr_{\ep}  \vu_{\ep} \otimes \vu_{\ep} )
-
\f{1}{\ep^2}  \Grad 
 \Big(
p(\vr_{\ep})-p(b)-p'(b)(\vr_{\ep}-b)
 \Big).
\end{align*}
The generalized Helmholtz decomposition of the vector field $\mathbf{m}_{\ep}$ is given as:
\begin{align*}
\mathbf{m}_{\ep}
=
\mathbb{P}_{b}[\mathbf{m}_{\ep}]
+
\mathbb{Q}_{b}[\mathbf{m}_{\ep}]
=
\mathbb{P}_{b}[\mathbf{m}_{\ep}]
+
b \Grad \Phi_{\ep}, 
\quad 
\div \,  \mathbb{P}_{b}[\mathbf{m}_{\ep}]  =0.
\end{align*}
So, applying $\mathbb{Q}_{b}$ to \eqref{deg-ns-m-2}$_2$ and dropping the force term $ \mathbf{f}_{\ep}$, we get the following system:
\begin{align*}
\left\{\begin{aligned}
& {\ep} \partial_t s_{\ep} + \Div\lr{ b  \nabla \Phi_{\ep}} =0,\\
&  \ep \partial_t \nabla \Phi_{\ep} + \Grad \lr{  s_{\ep} P^{\prime \prime}(b)} = 0.
\end{aligned}\right.
\end{align*}
If $ s_{\ep} $ and $ \Phi_{\ep}$ are sufficiently smooth, one obtains from this system a linear wave equation with variable coefficient $b$:  
\begin{align*}
       &{\ep}^2 \partial_{tt} s_{\ep} - \Div \lr{ b \nabla \lr{  P^{\prime \prime}(b) \, s_{\ep}}} =0.
\end{align*}

To make the following arguments rigorous,  we denote by $ ( s_\ed, \Phi_\ed )$  a unique, and as we seee below regular, solution to the acoustic wave system:
\begin{equation}\label{acc-wave}
\left\{\begin{aligned}
& {\ep} \partial_t s_\ed + \Div\lr{ b  \nabla \Phi_\ed} =0,\\
&  \ep \partial_t \nabla \Phi_\ed =  -\Grad \lr{ s_\ed P^{\prime \prime}(b)}, 
\end{aligned}\right.
\end{equation}
emanating from the smooth initial data 
\begin{align}\label{acc-wave-data}
	s_\ed (0,\cdot) = s_{0,\ep,\delta} = \frac{b}{p^\prime(b)} \left[ \frac{p^\prime(b)}{b} \phi_{0,\ep} \right]_\delta , 
    \quad 
    \Phi_\ed(0,\cdot) = \Phi_{0,\ep,\delta}=\left[ \Phi_{0,\ep} \right]_{\delta}, 
    \quad
    \nabla \Phi_{0,\ep}=  
    \f{1}{b} \mathbb{Q}_{b}[  b \mathbf{u}_{0,\ep}]     ,
\end{align}
where $[\cdot]_{\delta}$ denotes a suitable regularization which we are now going to describe.
Following  a regularization  scheme from \cite{FJN14}, we first introduce the operator
\begin{align*}
	\mathcal{T}_{b} [v] :v \mapsto -\frac{p^\prime(b)}{b} \Div \lr{ b \nabla v}.
\end{align*}
$\mathcal{T}_{b}$ is a linear, self-adjoint, non-negative operator on the weighted Hilbert space $L^2_g$ with the weight $g:=\frac{b}{p^\prime(b)}$, studied for example in \cite{De-Pr-91,De-Pr-92}.
Moreover, let $\mathcal{M}_\delta\in C_c^\infty(\mathbb{R})$ and $\psi_\delta\in C_c^\infty(\mathbb{R}^3)$ be such that 
\begin{align*}
	\begin{split}
		&0\leq \mathcal{M}_\delta (\cdot )\leq 1,\quad  \mathcal{M}_\delta(z)=\mathcal{M}_\delta(-z) 
        \\
        &
         \mathcal{M}_\delta(z)=\begin{cases}
			1 \quad \text{ for } z\in (-\frac{1}{\delta},-\delta) \cup (\delta, \frac{1}{\delta}) , \\
			0  \quad \text{ for } z\in (-\infty,-\frac{2}{\delta})
            \cup (-\frac{\delta}{2},\frac{\delta}{2} )
            \cup (\frac{2}{\delta},\infty ) , 
		\end{cases} \\
& \psi_\delta(x)=\psi(\delta x),\quad 0\leq \psi_\delta\leq 1 \text{ such that } \psi_\delta(x)=\begin{cases}
			1  \text{ for } |x| < \frac{1}{\delta} , \\
			0   \text{ for } |x| > \frac{2}{\delta} .
		\end{cases}
	\end{split}
\end{align*} 
Then, we use $\mc M_\delta$ to perform a truncation in frequencies and $\psi_\delta$ to localize functions on the physical space.
More precisely, the $\delta$-regularization  of any  $f \in L^2(\mathbb{R}^3)$ is defined by 
\begin{align*}
	[f]_\delta := \mathcal{M}_\delta\lr{\sqrt{\mathcal{T}_b}}[ \psi_\delta f ].
\end{align*}

Now, note that total energy of the acoustic wave system \eqref{acc-wave} is conserved, namely
\begin{align}\label{acc-wave-energy}
	\frac{\text{d}}{\dt} \int_{\mathbb{R}^3}
    {\Big(b| \nabla \Phi_\ed |^2 + 
     P''(b) |s_\ed|^2\Big)} \dx =0. 
\end{align} 
As a consequence of \eqref{acc-wave-energy} and the fact that $b$ has positive lower and upper bounds, one can further deduce the following global regularity estimate for the solutions:    
	\begin{align}  \label{acc-wave-reg}      
		\sup_{t\in[0,T]} \lr{ \Vert \Phi_\ed(t)\Vert_{W^{m,2} \lr{\mathbb{R}^3}} + \Vert s_\ed\Vert_{W^{m,2} \lr{\mathbb{R}^3}} }  
        \leq 
        C(m,\delta) 
        \lr{
        \Vert  \nabla \Phi_{0,\ep,\delta} \Vert_{L^{2} \lr{ \mathbb{R} ^3}} 
        + \Vert s_{0,\ep,\delta}\Vert_{L^{2} \lr{ \mathbb{R}^3}} 
        },
	\end{align}
where $m\geq 0$ in an integer, and the positive constant $C$ depends only on $m$ and $\delta$; in particular, it is independent of $\ep$. 
In addition,  we have the following dispersive estimates:
\begin{align}\label{acc-wave-dis}
		\int_{0}^T \lr{ \Vert \Phi_\ed(t)\Vert_{W^{m,\infty} \lr{\mathbb{R}^3}} + \Vert s_\ed\Vert_{W^{m,\infty} \lr{\mathbb{R}^3}} } \dt \leq \omega(m,\ep,\delta) \lr{ \Vert \nabla \Phi_{0,\ep,\delta} \Vert_{L^{2} \lr{\mathbb{R}^3}} + \Vert s_{0,\ep,\delta}\Vert_{L^{2} \lr{\mathbb{R}^3}} }, 
\end{align}
where    
\begin{align*}
		\omega(m,\ep,\delta) \rightarrow 0 \quad \text{ as } \ep \rightarrow 0  \,\,\text{ for fixed values of }  m \geq 0,\; \delta>0.
\end{align*} 
Both estimates have been obtained by Feireisl et al. \cite{FJN14}, and by Donatelli and Feireisl \cite{DF-17}. The details of the proof are thus omitted.

 
\subsection{Relative entropy inequality for ill-prepared data}

Let $(\vr_{\ep},\vv_{\ep},\vw_{\ep})$ be a global $\kappa$-entropy weak solution in the sense of Definition {\ref{def-1a}} and $( r_{\ed},\vV_{\ed},\vW_{\ed})$ be smooth test functions with $r_{\ed}>0$. 
Because $ r_{\ed}$ will not (eventually) satisfy the continuity equation, we cannot benefit from reduction of the relative $\kappa$-entropy discussed in Section \ref{ss:improved}. Instead, we retrieve the general $\kappa$-relative energy inequality derived in Lemma \ref{rela-entr-ine}, we have:
{\small{
	\eq{\label{relEnt:6}
				& \mathscr{E}( \vr_{\ep},\vv_{\ep},\vw_{\ep}| r_{\ed},\vV_{\ed},\vW_{\ed})(\tau)  -\mathscr{E}( \vr_{\ep},\vv_{\ep},\vw_{\ep}| r_{\ed},\vV_{\ed},\vW_{\ed})(0) \\
				& \quad \quad +   2 \kappa \mu \int_0^{\tau}  \!\!\!\int_{\Om} \left|\Ov{ \sqrt{\vr_{\ep}}  \,  \ba(\vv_\ep-\vV_{\ed}) } \right|^2 \dx\ds + 2\mu \int_0^{\tau}  \!\!\!\int_{\Om} \left| \Ov{ \sqrt{\vr_\ep} \, \bd \left( \sqrt{1-\kappa}(\vv_{\ep}-\vV_{\ed})-\sqrt{\kappa}(\vw_{\ep}-\vW_{\ed})   \right)   }      \right|^2       \dx\ds                    \\
				& \quad \quad + 2\f{\kappa \mu}{\ep^2} \int_0^{\tau}  \!\!\!\int_{\Om} \vr_{\ep} \Big( p'(\vr_{\ep})\Grad \log \vr_{\ep}-p'(r_{\ed})\Grad \log r_{\ed}  \Big) \cdot (\Grad \log \vr_{\ep}-\Grad \log r_{\ed}   )           \dx\ds \\ 
				& \quad  \leq   
			 \int_0^{\tau}  \!\!\!\int_{\Om}\vr_{\ep}\left[  \left( \vv_{\ep}-\sqrt{ \f{\ka}{1-\ka} } \vw_{\ep}  \right)\cdot \Grad \vW_{\ed} \cdot (\vW_{\ed}-\vw_{\ep})
				+\left( \vv_{\ep}-\sqrt{ \f{\ka}{1-\ka} } \vw_{\ep}  \right)\cdot \Grad \vV_{\ed} \cdot (\vV_{\ed}-\vv_{\ep})             \right] \dx\ds\\
				& \quad \quad + \int_0^{\tau}  \!\!\!\int_{\Om}\vr_{\ep}\Big(  \p_t \vW_{\ed}\cdot (\vW_{\ed}-\vw_{\ep})+\p_t \vV_{\ed} \cdot (\vV_{\ed}-\vv_{\ep})     \Big) \dx\ds \\
				& \quad \quad +\f{1}{\ep^2} \int_0^{\tau}  \!\!\!\int_{\Om}\p_t P'( r_{\ed})(r_{\ed}-\vr_{\ep}) \dx\ds \\
				& \quad \quad 
				-
				\f{1}{\ep^2} 
				\int_0^{\tau}  \!\!\!\int_{\Om}\Grad P'( r_{\ed}) \cdot \left[ \vr_{\ep}\left( \vv_{\ep}-\sqrt{ \f{\ka}{1-\ka} } \vw_{\ep}  \right)-r_{\ed} \left( \vV_{\ed}-\sqrt{ \f{\ka}{1-\ka} } \vW_{\ed}   \right)         \right] \dx\ds  
				\\
				& \quad \quad +\f{1}{\ep^2} \int_0^{\tau}  \!\!\!\int_{\Om}( p(r_{\ed})-p(\vr_{\ep})  ) \Div \left( \vV_{\ed}-\sqrt{ \f{\ka}{1-\ka} } \vW_{\ed}   \right)  \dx\ds \\ 
				&\quad \quad -\f{\ka}{\ep^2} \int_0^{\tau}  \!\!\!\int_{\Om}p'(\vr_{\ep})\Grad \vr_{\ep}\cdot \left( 2\mu \f{\Grad r_{\ed}}{r_{\ed}} - \f{1}{ \sqrt{\ka(1-\ka)}  }\vW_{\ed}       \right) \dx\ds \\
				& \quad \quad + 2\mu \int_0^{\tau}  \!\!\!\int_{\Om} \sqrt{\vr_{\ep}} \Big(  \bd(\sqrt{1-\kappa} \vV_{{\ed}})-
                \bd( \sqrt{\ka} \vW_{\ed})     \Big): \Ov{\sqrt{\vr_{\ep}} \Big(  \bd(\sqrt{1-\kappa} (\vV_{\ed}-\vv_{\ep}))-
                \bd( \sqrt{\ka} (\vW_{\ed}-\vw_{\ep}))     \Big) }      \dx\ds \\
				&  \quad \quad 
                + 
                2\ka \mu \int_0^{\tau}  \!\!\!\int_{\Om} \sqrt{\vr_{\ep}}  \, \ba  \vV_{\ed}: \Ov{\sqrt{\vr_{\ep}}  \, \ba(\vV_{\ed}-\vv_{\ep})} \dx\ds 
                + 
                2\f{\kappa \mu}{\ep^2} \int_0^{\tau}  \!\!\!\int_{\Om}\f{\vr_{\ep}}{ r_{\ed}} p'( r_{\ed}) \Grad r_{\ed} \cdot \left( \f{\Grad r_{\ed}}{ r_{\ed}}-\f{\Grad \vr_{\ep}}{\vr_{\ep}}  \right)           \dx\ds   \\
&  \quad \quad 
+ 2 \sqrt{ \ka (1-\ka) } \mu \int_0^{\tau}  \!\!\!\int_{\Om}  \sqrt{\vr_{\ep}}   {\ba}\vW_{\ed}: \Ov{  \sqrt{\vr_{\ep}} \ba \vv_{\ep}}     \dx\ds 
+ \f{1}{\ep^2} \int_0^{\tau}  \!\!\!\int_{\Om}\vr_{\ep}\Grad G \cdot (\vv_{\ep}-\vV_{\ed}) \dx\ds\\
&\quad=\sum_{i=1}^{11} I_i.
	}
	}}
To proceed, we notice that $I_3+I_4+I_5$ gives
	\begin{align*} 
		I_3+I_4+I_5& 
		= - \f{1}{\ep^2} \int_0^{\tau}  \!\!\!\int_{\Om}
		\Big(
		p(\vr_{\ep}) -p( r_{\ed} )-p^\prime( r_{\ed}) (\vr_{\ep}-r_\ed )
		\Big)
		\Div \vU_{\ed}  \dx\ds \\ 
		& \quad \quad +\f{1}{\ep^2} \int_0^{\tau}  \!\!\!\int_{\Om} P{''}( r_{\ed})( r_{\ed}-\vr_{\ep}) \lr{\p_t r_{\ed} + \Div( r_{\ed} \vU_{\ed}) } \dx\ds \\
		& \quad \quad - \f{1}{\ep^2} 
		\int_0^{\tau}  \!\!\!\int_{\Om}\Grad P'( r_{\ed}) \cdot \vr_{\ep}(\vu_{\ep}-\vU_{\ed}) \dx\ds,
	\end{align*}
    where we used
    $ \vU_{\ed}:= \vV_{\ed}-\beta \vW_{\ed}$, $\beta=\sqrt{\frac{\kappa}{1-\kappa}}$.
	Observe next that 
	\begin{align*}
		& \f{1}{\ep^2} 
		\int_0^{\tau}  \!\!\!\int_{\Om}\Grad P'( r_{\ed}) 
		\cdot 
		\vr_{\ep}(\vu_{\ep}-\vU_{\ed}) 
		\dx\ds 
		\\
		& \quad =
		\f{1}{\ep^2} 
		\int_0^{\tau}  \!\!\!\int_{\Om}\Grad P'( r_{\ed}) \cdot \vr_{\ep}(\vv_{\ep}-\vV_{\ed}) \dx\ds - \f{\beta}{\ep^2} \int_0^{\tau}  \!\!\!\int_{\Om} \nabla P^{\prime  }(r_\ed)\cdot \vr_{\ep} 
		\left( \vw_{\ep}-\vW_{\ed}\right)       
		\dx\ds  .
	\end{align*}
    Focus on the second term on the right-hand side of the above equality for a while: we may sum it up with $I_6$ and $I_9$ to obtain
        \begin{align*}
            & \f{\beta}{\ep^2} \int_0^{\tau}  \!\!\!\int_{\Om} \nabla P^{\prime  }(r_\ed)\cdot \vr_{\ep} 
		\left( \vw_{\ep}-\vW_{\ed}\right)       
		\dx\ds -\f{\ka}{\ep^2} \int_0^{\tau}  \!\!\!\int_{\Om}p'(\vr_{\ep})\Grad \vr_{\ep}\cdot \left( 2\mu \f{\Grad r_{\ed}}{r_{\ed}} - \f{1}{ \sqrt{\ka(1-\ka)}  }\vW_{\ed}       \right) \dx\ds \\
        & \qquad 
        +2\f{\kappa \mu}{\ep^2} \int_0^{\tau}  \!\!\!\int_{\Om}\f{\vr_\ep}{ r_{\ed}} p'( r_{\ed}) \Grad r_{\ed} \cdot \left( \f{\Grad r_{\ed}}{ r_{\ed}}-\f{\Grad \vr_{\ep}}{\vr_{\ep}}  \right)           \dx\ds \\
        & \quad 
        =
        \f{\beta}{\ep^2} \int_0^{\tau}  \!\!\!\int_{\Om} \vr_\ep p^\prime(r_{\ed}) \nabla \log r_{\ed} \cdot 
		\left( \vw_{\ep}-\vW_{\ed}\right)       
		\dx\ds \\
        & \qquad 
        -\f{\ka}{\ep^2} \int_0^{\tau}  \!\!\!\int_{\Om} \vr_\ep p'(\vr_{\ep})  \Grad \log \vr_{\ep}\cdot \left( 2\mu \nabla \log r_{\ed} - \f{1}{ \sqrt{\ka(1-\ka)}  }\vW_{\ed}       \right) \dx\ds\\
        & \qquad 
        +2\f{\kappa \mu}{\ep^2} \int_0^{\tau}  \!\!\!\int_{\Om} \vr_{\ep} p'( r_{\ed}) \Grad \log r_{\ed} \cdot \left( \nabla \log r_{\ed} -\nabla \log \vr_\ep \right)           \dx\ds \\
         & \quad =
         -2\f{\ka\mu}{\ep^2} \int_0^{\tau}  \!\!\!\int_{\Om} \vr_\ep p'(\vr_{\ep})  \Grad \log \vr_{\ep}\cdot \left(  \nabla \log r_{\ed} - \f{1}{ 2\mu \sqrt{\ka(1-\ka)}  }\vW_{\ed}       \right) \dx\ds\\
          & \qquad 
          +2\f{\kappa \mu}{\ep^2} \int_0^{\tau}  \!\!\!\int_{\Om} \vr_{\ep} p'( r_{\ed}) \Grad \log r_{\ed} \cdot \left( \nabla \log r_{\ed} -\frac{1}{2\mu \sqrt{\kappa(1-\kappa)} } \vW_{\ed} \right)           \dx\ds .
        \end{align*}

Therefore, altogether from the discussion above, we have 
\begin{align} \label{ill-rei1}
    \begin{split}
	I_3+I_4+I_5+&I_6+I_9\\
    =& - \f{1}{\ep^2} \int_0^{\tau}  \!\!\!\int_{\Om}
		\Big(
		p(\vr_{\ep}) -p( r_{\ed} )-p^\prime( r_{\ed}) (\vr_{\ep}-r_\ed )
		\Big)
		\Div \vU_{\ed}  \dx\ds \\ 
		&  +\f{1}{\ep^2} \int_0^{\tau}  \!\!\!\int_{\Om} P{''}( r_{\ed})( r_{\ed}-\vr_{\ep}) \lr{\p_t r_{\ed} + \Div( r_{\ed} \vU_{\ed}) } \dx\ds \\
		&  - \f{1}{\ep^2} 
		\int_0^{\tau}  \!\!\!\int_{\Om}\Grad P'( r_{\ed}) \cdot \vr_{\ep}(\vv_{\ep}-\vV_{\ed}) \dx\ds \\
       &-2\f{\ka\mu}{\ep^2} \int_0^{\tau}  \!\!\!\int_{\Om} \vr_\ep p'(\vr_{\ep})  \Grad \log \vr_{\ep}\cdot \left(  \nabla \log r_{\ed} - \f{1}{ 2\mu \sqrt{\ka(1-\ka)}  }\vW_{\ed}       \right) \dx\ds\\
          &+2\f{\kappa \mu}{\ep^2} \int_0^{\tau}  \!\!\!\int_{\Om} \vr_{\ep} p'( r_{\ed}) \Grad \log r_{\ed} \cdot \left( \nabla \log r_{\ed} -\frac{1}{2\mu \sqrt{\kappa(1-\kappa)} } \vW_{\ed} \right)           \dx\ds .
    \end{split}
\end{align}
Finally, by invoking the algebraic relation \eqref{eq:log-rho}, we can transform \eqref{relEnt:6} into
{\small{
	\eq{\label{relEnt:7}
				& \mathscr{E}( \vr_{\ep},\vv_{\ep},\vw_{\ep}| r_{\ed},\vV_{\ed},\vW_{\ed})(\tau)  -\mathscr{E}( \vr_{\ep},\vv_{\ep},\vw_{\ep}| r_{\ed},\vV_{\ed},\vW_{\ed})(0) \\
				& \quad  +   2 \kappa \mu \int_0^{\tau}  \!\!\!\int_{\Om} \left|\Ov{ \sqrt{\vr_{\ep}}  \,  \ba(\vv_\ep-\vV_{\ed}) } \right|^2 \dx\ds + 2\mu \int_0^{\tau}  \!\!\!\int_{\Om} \left| \Ov{ \sqrt{\vr_\ep} \, \bd \left( \sqrt{1-\kappa}(\vv_{\ep}-\vV_{\ed})-\sqrt{\kappa}(\vw_{\ep}-\vW_{\ed})   \right)   }      \right|^2       \dx\ds                    \\
				& \quad 
                + 2\f{\kappa \mu}{\ep^2}  \int_0^{\tau}\!\!\! \int_{\Om}  \vr_{\ep}  p'(\vr_{\ep}) \vert \Grad \log \vr_{\ep}-\Grad \log  r_{\ed} \vert^2 \dx\ds   \\ 
				&  \leq   
			 \int_0^{\tau}\!\!\! \int_{\Om}\vr_{\ep}\left[  \left( \vv_{\ep}-\sqrt{ \f{\ka}{1-\ka} } \vw_{\ep}  \right)\cdot \Grad \vW_{\ed} \cdot (\vW_{\ed}-\vw_{\ep})
				+\left( \vv_{\ep}-\sqrt{ \f{\ka}{1-\ka} } \vw_{\ep}  \right)\cdot \Grad \vV_{\ed} \cdot (\vV_{\ed}-\vv_{\ep})             \right] \dx\ds\\
				& \quad  + \int_0^{\tau}  \!\!\!\int_{\Om}\vr_{\ep}\Big(  \p_t \vW_{\ed}\cdot (\vW_{\ed}-\vw_{\ep})+\p_t \vV_{\ed} \cdot (\vV_{\ed}-\vv_{\ep})     \Big) \dx\ds \\
  & \quad- \f{1}{\ep^2} \int_0^{\tau}  \!\!\!\int_{\Om}
		\Big(
		p(\vr_{\ep}) -p( r_{\ed} )-p^\prime( r_{\ed}) (\vr_{\ep}-r_\ed )
		\Big)
		\Div \vU_{\ed}  \dx\ds \\ 
		& \quad +\f{1}{\ep^2} \int_0^{\tau}  \!\!\!\int_{\Om} P{''}( r_{\ed})( r_{\ed}-\vr_{\ep}) \lr{\p_t r_{\ed} + \Div( r_{\ed} \vU_{\ed}) } \dx\ds \\
		& \quad - \f{1}{\ep^2} 
		\int_0^{\tau}  \!\!\!\int_{\Om}\Grad P'( r_{\ed}) \cdot \vr_{\ep}(\vv_{\ep}-\vV_{\ed}) \dx\ds \\
       &\quad-2\f{\ka\mu}{\ep^2} \int_0^{\tau}  \!\!\!\int_{\Om} \vr_\ep p'(\vr_{\ep})  \Grad \log \vr_{\ep}\cdot \left(  \nabla \log r_{\ed} - \f{1}{ 2\mu \sqrt{\ka(1-\ka)}  }\vW_{\ed}       \right) \dx\ds\\
          &\quad+2\f{\kappa \mu}{\ep^2} \int_0^{\tau}  \!\!\!\int_{\Om} \vr_{\ep} p'( r_{\ed}) \Grad \log r_{\ed} \cdot \left( \nabla \log r_{\ed} -\frac{1}{2\mu \sqrt{\kappa(1-\kappa)} } \vW_{\ed} \right)           \dx\ds \\
				& \quad + 2\mu \int_0^{\tau}  \!\!\!\int_{\Om} \sqrt{\vr_{\ep}} \Big(  \bd(\sqrt{1-\kappa} \vV_{ \ed })-\bd( \sqrt{\ka} \vW_{\ed})     \Big): \Ov{\sqrt{\vr_{\ep}} \Big(  \bd(\sqrt{1-\kappa} (\vV_{\ed}-\vv_{\ep}))-\bd( \sqrt{\ka} (\vW_{\ed}-\vw_{\ep}))     \Big) }      \dx\ds \\
				& \quad 
                + 
                2\ka \mu \int_0^{\tau}  \!\!\!\int_{\Om} \sqrt{\vr_{\ep}}  \, \ba  \vV_{\ed}: \Ov{\sqrt{\vr_{\ep}}  \, \ba(\vV_{\ed}-\vv_{\ep})} \dx\ds 
                \\
&  \quad  
+ 2 \sqrt{ \ka (1-\ka) } \mu \int_0^{\tau}  \!\!\!\int_{\Om}  \sqrt{\vr_{\ep}}   \ba \vW_{\ed}: \Ov{  \sqrt{\vr_{\ep}} \ba \vv_{\ep}}     \dx\ds 
+ \f{1}{\ep^2} \int_0^{\tau}  \!\!\!\int_{\Om}\vr_{\ep}\Grad G \cdot (\vv_{\ep}-\vV_{\ed}) \dx\ds\\
&\quad
- 2\f{\kappa \mu}{\ep^2}\int_0^{\tau}\!\!\!\int_{\Om}  \vr_{\ep} \Big( p'(\vr_{\ep})-p'( r_{\ed})\Big) \Grad \log  r_{\ed} \cdot (\Grad \log \vr_{\ep}-\Grad \log  r_{\ed}    )           \dx\ds.
	}
	}}

\subsection{Reduced relative \texorpdfstring{$\kappa$}{}-entropy inequality} \label{subsubs:reduce:ill}

 We now choose suitable test functions 
 \((r_{\ed}, \vV_{\ed}, \vW_{\ed})\) by incorporating the acoustic wave system \eqref{acc-wave}. 
 More precisely, we first consider $ (b, \vU) $ to be the regular solution to 
 the target system \eqref{ane-Euler} with initial value $\vU_0 := \f{1}{b} \mathbb{P}_b[b \vu_0]$, as given by Theorem \ref{thm:LWP-R3}.
 Then, we define $\vV_\ep$ and $\vW_\ep$ according to formula \eqref{def:vw}, namely 
 \[
 \vV_\ep := \vU+2\kappa \mu \Grad \log b \quad  \text{ and } \quad \vW_\ep:=2 \sqrt{\kappa (1-\kappa)} \, \mu \Grad \log b\,. 
 \]
 We notice that the dependence of $\vV_\ep$ and $\vW_\ep$ on $\ep>0$ only relies on the presence of the factor $\mu=\mu_\ep\to0$ when $\ep\to0$.
 We then consider \((r_{\ed}, \vV_{\ed}, \vW_{\ed})\) as
	\begin{align}\label{eq:test}
 		r_\ed =b+ \ep  s_{\ed} 
     \qquad	\vV_{\ed}=  \vV_{\ep}+ \nabla \Phi_{\ed} ,  
 	\qquad
    \vW_{\ed}= \vW_{\ep} =2 \mu_\ep \sqrt{\kappa(1-\kappa)} \nabla \log b ,
 	\end{align}
	where $\{ s_{\ed}, \Phi_{\ed}\}_{\ep>0}$ solves the acoustic wave system \eqref{acc-wave}. For later use, we set $\vU_\ed := \vV_\ed - \beta \vW_\ep$.

	We notice that the class specified for the test functions in Lemma \ref{rela-entr-ine} requires certain regularity of $(r_{\ed},\vV_{\ed},\vW_{\ed})$. This is precisely why we introduced the regularization of the data \eqref{acc-wave-data}.
    Although the sign of $r_{\ed}$  is not determined,  for any fixed $\delta$ and for sufficiently small $\ep$, we have
	\begin{align*}
		r_{\ed}= b+ \ep  s_{\ed} >0.
	\end{align*}
    Indeed, from the estimates \eqref{acc-wave-energy} and \eqref{acc-wave-reg}, it follows that
	\begin{align*}
		\sup_{(t,x)\in (0,T)\times \mathbb{R}^3 } |s_{\ep,\delta}(t,x)| \leq c(\delta)   
	\end{align*}
	uniformly with respect to $\ep$. 
	Hence, for fixed $\delta>0$ and for sufficiently small $\ep\in (0,\veps_\delta) $, it holds
	\begin{align}\label{r-ep-del}
    \frac{\underline{b}}{2}\leq   b- \ep c(\delta) \leq r_{\ed} = b + \ep s_{\ep,\delta}\leq b+ \ep c(\delta)\leq 2 \Ov{b}.     
	\end{align}

Now we are in the position of stating the following lemma, which expresses a reduced relative energy inequality.

\begin{Lemma}\label{lemm:REI-ill}
		Let $(\vr_{\ep},\vv_{\ep},\vw_{\ep})$ be a $\kappa$-entropy solution to the degenerate compressible Navier--Stokes equations \eqref{dege-ns},
in the sense of Definition \ref{def-1a}, emanating from the initial datum $(\vr_{0,\ep},\vv_{0,\ep},\vw_{0,\ep})$. Let $( r_{\ed},\vV_{\ed},\vW_{\ed})$ be defined by \eqref{eq:test}. Then, the following inequality holds:
		\small{
					\begin{align}\label{rel_entr13}
			\begin{split}
				& \mathscr{E}( \vr_{\ep},\vv_{\ep},\vw_{\ep}| r_{\ed},\vV_{\ed},\vW_{\ep} )(\tau)   -
				\mathscr{E}( \vr_{\ep},\vv_{\ep},\vw_{\ep}| r_{\ed},\vV_{\ed},\vW_\ep)(0)+\int_0^{\tau} \mathcal{D}_\ep \ds\\ 
				& \quad  \le \int_0^{\tau}  \!\!\!\int_{\Om}\vr_{\ep} \left(  \vu_{\ep} -\vU_{\ed}  \right)\cdot \Grad \vW_{\ep}  \cdot (\vW_{\ep} -\vw_{\ep})\dx\ds 
				+\int_0^{\tau}  \!\!\!\int_{\Om}\vr_{\ep}  \left(  \vu_{\ep} -\vU_{\ed}  \right)  \cdot \Grad \vV_{\ed} \cdot (\vV_{\ed}-\vv_{\ep})            \dx\ds\\
				& \quad \quad 
				-
				\int_0^{\tau}  \!\!\!\int_{\Om}\vr_{\ep} (\vV_{\ed}-\vv_{\ep})
				\cdot \nabla \Pi  \dx\ds \\
				& \quad \quad 
				-  
				\int_0^{\tau}  \!\!\!\int_{\Om} \vr_{\ep} (\vv_{\ep}-\vV_{\ed}) \lr{ \nabla \Phi_{\ed} \cdot \nabla \vV_{\ep}  + \vU \cdot \nabla  \nabla \Phi_{\ed}+ \nabla \Phi_{\ed}\cdot \nabla  \nabla \Phi_{\ed}}   \dx\ds   \\
				& \quad \quad + \int_0^{\tau}  \!\!\!\int_{\Om}\vr_{\ep} ( \vW_{\ep} -\vw_{\ep}) \lr{ \nabla \Phi_{\ed} \cdot \nabla \vW_{\ep} } \dx \ds \\
				& \quad \quad -\f{1}{\ep^2} \int_0^{\tau}  \!\!\!\int_{\Om}\lr{ p(\vr_{\ep})-p( r_{\ed}) -  p'( r_{\ed})(\vr_{\ep}- r_{\ed}) } \Div \vU_{\ed}\dx\ds \\
				& \quad \quad 
				- \f{1}{\ep} \int_0^{\tau}  \!\!\!\int_{\Om}P^{\prime \prime}( r_{\ed} )(\vr_{\ep}- r_{\ed}) \Div\lr{ s_{\ed} \vU_{\ed}} \dx\ds  \\
				&\quad \quad
				-\f{1}{\ep^2} \int_0^{\tau}  \!\!\!\int_{\Om}\vr_{\ep}  \Grad \lr{ P^\prime(r_{\ed})-P^\prime(b)-( r_{\ed}-b)P^{\prime \prime}(b)} \cdot (\vv_{\ep}-\vV_{\ed}) \dx\ds \\
				&\quad \quad 
                -
                2\mu\int_0^{\tau}\!\!\!\int_{\Om}  \lr{\nabla \log b-\nabla\log\vr_\ep} \otimes\vr_{\ep} (\vu_{\ep}-\vU_{\ed}): \Big[ 
				(1-\kappa)\bd \vV_{\ep}   	     
                -   
                \sqrt{\kappa (1-\kappa)}  \Grad \vW_{\ep}     
                \Big]    \dx\ds  \\
				& \quad \quad  +2\mu\kappa\int_0^{\tau} \!\!\!\int_{\Om} \lr{\nabla \log b-\nabla\log\vr_\ep} \otimes\vr_{\ep} (\vv_{\ep}-\vV_{\ed}) :     \ba \vV_{\ep} \dx\dt   \\
				& \quad \quad +  2 \mu \sqrt{\kappa (1-\kappa)}  \int_0^{\tau}  \!\!\!\int_{\Om}\lr{\nabla \log b-\nabla\log\vr_\ep}\otimes \vr_{\ep} (\vw_{\ep}-\vW_{\ep} ) :   \ba \vV_{\ep}    \dx \ds \\
              &\quad\quad 
              - 2\f{\kappa \mu}{\ep^2}\int_0^{\tau}\!\!\!\int_{\Om}  \vr_{\ep} \Big( p'(\vr_{\ep})-p'( r_{\ed})\Big) \Grad \log  r_{\ed} \cdot (\Grad \log \vr_{\ep}-\Grad \log  r_{\ed}    )           \dx\ds   \\
                      & \quad \quad 
               +\f{2\kappa \mu}{\ep^2} \int_0^{\tau}  \!\!\!\int_{\Om} \vr_{\ep} \lr{ p'( r_{\ed}) \Grad \log r_{\ed} - p'( \vr_{\ep}) \Grad \log \vr_{\ep}}\cdot \left( \nabla \log r_{\ed} -\nabla \log b \right)           \dx\ds \\
               &\quad \quad+
    2\mu_{\ep} \int_0^{\tau}  \!\!\!\int_{\Om}  \Ov{\sqrt{\vr_{\ep}}\Big(  \bd(\sqrt{1-\kappa} (\vV_{\ed}-\vv_{\ep}))-\bd( \sqrt{\ka} (\vW_{\ep} -\vw_{\ep}))     \Big)  } : \sqrt{ \vr_{\ep} }\Big(  \bd(\sqrt{1-\kappa} \Grad \Phi_{\ed} )   \Big)  \dx\ds
    \\
    &\qquad+
    2\mu\int_0^\tau\!\!\!\int_\Om
        \vr_\ep(\vv_\ep-\vV_{\ed})\cdot
        b^{-1}\Div(b\bd\vU)\dx\ds,
			\end{split}
		\end{align}
        where $\mathcal{D}_\ep$ is the dissipation term given by 
        \eqh{
        \mathcal{D}_\ep &
        =
        2 \kappa \mu  \int_{\Om}  \left| \Ov{ \sqrt{\vr_{\ep}}  \, \ba(\vv_{\ep}-\vV_{\ed}) } \right|^2 \dx
                + 2 \mu \int_{\Om}  \left| \Ov{ \sqrt{\vr_{\ep}} \,\bd \left( \sqrt{1-\kappa}(\vv_{\ep}-\vV_{\ed})-\sqrt{\kappa}(\vw_{\ep}- \vW_{\ep} )   \right)         } \right|^2       \dx     \\
				&  \quad 
                + 2\f{\kappa \mu  }{\ep^2} \int_{\Om}  \vr_{\ep}  p'(\vr_{\ep}) \vert \Grad \log \vr_{\ep}-\Grad \log  r_{\ed} \vert^2    \dx . 
        }
	}	
	\end{Lemma}
    For a later purpose, we denote the right-hand side of \eqref{rel_entr13} as $\sum_{i=1}^{15} \mathcal{I}_i$.

{\emph {Proof of Lemma \ref{lemm:REI-ill}.}} Making use of the acoustic wave equations \eqref{acc-wave}, we deduce the following identities: 
	\begin{align}\label{h3}
		\begin{split}
			&
			\partial_t r_{\ed} + \Div( r_{\ed} \vU_{\ed})= 
            \underbrace{
            {\ep} \partial_t  s_{\ed} + \Div\lr{ b \nabla \Phi_{\ed}}
            }_{=0} + \ep \Div \lr{  s_{\ed} \mathbf{U}_\ed } = \ep   \Div \lr{  s_{\ed} \mathbf{U}_\ed },   
			\\
			&\vU_{\ed}= \vV_{\ed}-\beta\vW_\ed= \vU +\nabla \Phi_{\ed} , \\
			&\vW_\ed= \vW_{\ep} =2\sqrt{\kappa(1-\kappa)} \, \mu \nabla \log b, 
		\end{split}
	\end{align} 
	where we also used the property $\Div (b \vU)=0$. Before going on, we notice that 
	\begin{align*}
		\ba \vW_{\ep}=\mathbf{0} \quad
        \text{ and } 
        \quad 
        \ba \vV_{\ed}=\ba \lr{  \vV_{\ep}+ \nabla \Phi_{\ed}}= \ba \vV_{\ep}.  
	\end{align*}
Plugging the relation \eqref{h3} into \eqref{relEnt:7}, we obtain
	 {\small{
			\begin{align}\label{rel-ent-ip-2}
				\begin{split}
					& \mathscr{E}( \vr_{\ep},\vv_{\ep},\vw_{\ep}| r_{\ed},\vV_{\ed},\vW_{\ep} )(\tau)  -\mathscr{E}( \vr_{\ep},\vv_{\ep},\vw_{\ep}| r_{\ed},\vV_{\ed},\vW_{\ep} )(0) +\int_0^{\tau} \mathcal{D}_\ep \ds\\
					& \quad  \leq    \int_0^{\tau}  \!\!\!\int_{\Om}\vr_{\ep} \left(  \vu_{\ep} -\vU_{\ed}  \right)\cdot \Grad \vW_{\ep}  \cdot (\vW_{\ep} -\vw_{\ep})\dx\ds 
					+\int_0^{\tau}  \!\!\!\int_{\Om}\vr_{\ep}  \left(  \vu_{\ep} -\vU_{\ed}  \right)  \cdot \Grad \vV_{\ed} \cdot (\vV_{\ed}-\vv_{\ep})            \dx\ds\\
					& \quad \quad + \int_0^{\tau}  \!\!\!\int_{\Om}\vr_{\ep} (\vW_{\ep} -\vw_{\ep}) \cdot \lr{ \p_t \vW_{\ep} + \vU_{\ed} \cdot \nabla \vW_{\ep} } \dx \ds -
					\int_0^{\tau}  \!\!\!\int_{\Om} \vr_{\ep} (\vv_{\ep}-\vV_{\ed}) \cdot \lr{ \p_t \vV_{\ed}+\vU_{\ed} \cdot \nabla \vV_{\ed}}   \dx\ds
					\\
					& \quad \quad -\f{1}{\ep^2} \int_0^{\tau}  \!\!\!\int_{\Om}
                    \Big(
                    p(\vr_{\ep})-p( r_{\ed}) -  p'( r_{\ed})(\vr_{\ep}- r_{\ed})
                    \Big)   
                    \Div \vU_{\ed}\dx\ds \\
					&\quad \quad 
					-{ \f{1}{\ep} \int_0^{\tau}  \!\!\!\int_{\Om}P^{\prime \prime}( r_{\ed} )(\vr_{\ep}- r_{\ed}) \Div\lr{ s_{\ed} \vU_{\ed}}\dx\ds }
					\\
					&\quad \quad + 
					\f{1}{\ep^2} \int_0^{\tau}  \!\!\!\int_{\Om}\vr_{\ep} \lr{\Grad G- \Grad P^\prime( r_{\ed})} \cdot (\vv_{\ep}-\vV_{\ed}) \dx\ds
					+ 2\ka   \mu \int_0^{\tau}  \!\!\!\int_{\Om} \sqrt{\vr_{\ep}} \ba  \vV_{\ep} : \Ov{\sqrt{\vr_{\ep}} \ba(\vV_{\ep} -\vv_{\ep})} \dx\ds  \\
					& \quad \quad +  2 \mu \int_0^{\tau}  \!\!\!\int_{\Om} \sqrt{ \vr_{\ep} }\Big(  \bd(\sqrt{1-\kappa} \vV_{\ed})-
                    \bd
                    ( \sqrt{\ka} \vW_{\ep} )     \Big): \Ov{\sqrt{\vr_{\ep}}\Big(  \bd(\sqrt{1-\kappa} (\vV_{\ed}-\vv_{\ep}))-\bd
                    ( \sqrt{\ka} (\vW_{\ep} -\vw_{\ep}))     \Big)  }      \dx\ds \\
                    &\quad\quad 
                   - 2\f{\kappa \mu  }{\ep^2}\int_0^{\tau}\!\!\!\int_{\Om}  \vr_{\ep} \Big( p'(\vr_{\ep})-p'( r_{\ed})\Big) \Grad \log  r_{\ed} \cdot (\Grad \log \vr_{\ep}-\Grad \log  r_{\ed}    )           \dx\ds  \\
                    & \quad \quad 
                   +
                    \f{2\kappa \mu}{\ep^2} \int_0^{\tau}  \!\!\!\int_{\Om} \vr_{\ep} \lr{ p'( r_{\ed}) \Grad \log r_{\ed} - p'( \vr_{\ep}) \Grad \log \vr_{\ep}}\cdot \left( \nabla \log r_{\ed} -\nabla \log b \right)           \dx\ds \\
					&\quad =:\sum_{i=1}^{11} \widetilde{I}_i.
				\end{split}
	\end{align}  }}
	For the term $ \widetilde{I}_7$ in the above inequality, we observe that 
	\begin{align*} 
		&   \f{1}{\ep^2} \int_0^{\tau}  \!\!\!\int_{\Om}\vr_{\ep} \lr{\Grad G- \Grad P^\prime( r_{\ed})} \cdot (\vv_{\ep}-\vV_{\ed}) \dx\ds  \\
		&\quad =-\f{1}{\ep^2} \int_0^{\tau}  \!\!\!\int_{\Om}\vr_{\ep}  \Grad \lr{ P^\prime( r_{\ed})-P^\prime(b)-( r_{\ed}-b)P^{\prime \prime}(b)} \cdot (\vv_{\ep}-\vV_{\ed}) \dx\ds \\
		&\qquad \quad -\f{1}{\ep^2} \int_0^{\tau}  \!\!\!\int_{\Om}\vr_{\ep}  \Grad \lr{ ( r_{\ed}-b)P^{\prime \prime}(b)} \cdot (\vv_{\ep}-\vV_{\ed}) \dx\ds \\ 
		&\quad =-\f{1}{\ep^2} \int_0^{\tau}  \!\!\!\int_{\Om}\vr_{\ep}  \Grad \lr{ P^\prime( r_{\ed})-P^\prime(b)-( r_{\ed}-b)P^{\prime \prime}(b)} \cdot (\vv_{\ep}-\vV_{\ed}) \dx\ds \\
		&\qquad \quad -\f{1}{\ep} \int_0^{\tau}  \!\!\!\int_{\Om}\vr_{\ep}  \Grad \lr{  s_{\ed} P^{\prime \prime}(b)} \cdot (\vv_{\ep}-\vV_{\ed}) \dx\ds .   
	\end{align*}
	Further, summing the above equality with $ \widetilde{I}_4$, with the help of \eqref{eq:test}$_2$ and \eqref{acc-wave}$_2$, we obtain
	\begin{align*}
		&  \f{1}{\ep^2} \int_0^{\tau}  \!\!\!\int_{\Om}\vr_{\ep} \lr{\Grad G- \Grad P^\prime( r_{\ed})} \cdot (\vv_{\ep}-\vV_{\ed}) \dx\ds 
		-   
		\int_0^{\tau}  \!\!\!\int_{\Om} \vr_{\ep} (\vv_{\ep}-\vV_{\ed}) \lr{ \p_t \vV_{\ed}+\vU_{\ed} \cdot \nabla \vV_{\ed}}   \dx\ds \\
		&\quad 
		=
		-\f{1}{\ep^2} \int_0^{\tau}  \!\!\!\int_{\Om}\vr_{\ep}  \Grad \lr{ P^\prime( r_{\ed})-P^\prime(b)-( r_{\ed}-b)P^{\prime \prime}(b)} \cdot (\vv_{\ep}-\vV_{\ed}) \dx\ds \\
		&\quad \qquad 
		-   
		\int_0^{\tau}  \!\!\!\int_{\Om} \vr_{\ep} (\vv_{\ep}-\vV_{\ed}) \lr{ \p_t \vV_{\ep} +\vU_{\ed} \cdot \nabla \vV_{\ed}}   \dx\ds	\\
		&\quad \qquad
		\underbrace{-\f{1}{\ep} \int_0^{\tau}  \!\!\!\int_{\Om}\vr_{\ep} \lr{ \ep \partial_t \nabla \Phi_{\ed} + \Grad \lr{  s_{\ed} P^{\prime \prime}(b)} }\cdot (\vv_{\ep}-\vV_{\ed}) \dx\ds }_{=0}.
	\end{align*}

	With this, we can rewrite \eqref{rel-ent-ip-2} as \small{
	\begin{align}\label{rel-ent-ip-3}
		\begin{split}
			 &\mathscr{E}( \vr_{\ep},\vv_{\ep},\vw_{\ep}| r_{\ed},\vV_{\ed},\vW_{\ep} )(\tau)   - \mathscr{E}( \vr_{\ep},\vv_{\ep},\vw_{\ep}| r_{\ed},\vV_{\ed},\vW_{\ep} )(0) +\int_0^{\tau} \mathcal{D}_\ep \ds\\
			& 
            \leq \int_0^{\tau}  \!\!\!\int_{\Om}\vr_{\ep} \left(  \vu_{\ep} -\vU_{\ed}  \right)\cdot \Grad \vW_{\ep}  \cdot (\vW_{\ep} -\vw_{\ep})\dx\ds 
			+\int_0^{\tau}  \!\!\!\int_{\Om}\vr_{\ep}  \left(  \vu_{\ep} -\vU_{\ed}  \right)  \cdot \Grad \vV_{\ed} \cdot (\vV_{\ed}-\vv_{\ep})            \dx\ds\\
			& \quad  + \int_0^{\tau}  \!\!\!\int_{\Om}\vr_{\ep} (\vW_{\ep} -\vw_{\ep}) \cdot \lr{ \p_t \vW_{\ep} + \vU_{\ed} \cdot \nabla \vW_{\ep}}  \dx \ds 
			-   
			\int_0^{\tau}  \!\!\!\int_{\Om} \vr_{\ep} (\vv_{\ep}-\vV_{\ed})
            \cdot 
            \lr{ \p_t  \vV_{\ep} +\vU_{\ed} \cdot \nabla \vV_{\ed}}   \dx\ds \\
			& \quad  -\f{1}{\ep^2} \int_0^{\tau}  \!\!\!\int_{\Om}
            \Big( 
            p(\vr_{\ep})-p(r_{\ed}) -  p'( r_{\ed})(\vr_{\ep}- r_{\ed}) 
            \Big) 
            \Div \vU_{\ed}\dx\ds \\
			& \quad 
			- \f{1}{\ep} \int_0^{\tau}  \!\!\!\int_{\Om}P^{\prime \prime}( r_{\ed} )(\vr_{\ep}- r_{\ed}) \Div\lr{ s_{\ed} \vU_{\ed}} \dx\ds  \\
			&\quad
			-\f{1}{\ep^2} \int_0^{\tau}  \!\!\!\int_{\Om}\vr_{\ep}  \Grad \lr{ P^\prime(r_{\ed})-P^\prime(b)-( r_{\ed}-b)P^{\prime \prime}(b)} \cdot (\vv_{\ep}-\vV_{\ed}) \dx\ds \\
			&\quad + 2\ka \mu \int_0^{\tau}  \!\!\!\int_{\Om} \sqrt{\vr_{\ep}} \ba  \vV_{\ep} : \Ov{ \sqrt{\vr_{\ep}} \ba(\vV_{\ep} -\vv_{\ep})} \dx\ds  \\
			& \quad +  2 \mu \int_0^{\tau}  \!\!\!\int_{\Om} \sqrt{ \vr_{\ep} }\Big(  \bd(\sqrt{1-\kappa} \vV_{\ed})-
                    \bd
                    ( \sqrt{\ka}    \vW_{\ep} )     \Big): \Ov{\sqrt{\vr_{\ep}}\Big(  \bd(\sqrt{1-\kappa} (\vV_{\ed}-\vv_{\ep}))-\bd
                    ( \sqrt{\ka} (  \vW_{\ep} -\vw_{\ep}))     \Big)  }      \dx\ds \\
        &\quad
        - 2\f{\kappa \mu }{\ep^2}\int_0^{\tau}\!\!\!\int_{\Om}  \vr_{\ep} \Big( p'(\vr_{\ep})-p'( r_{\ed})\Big) \Grad \log  r_{\ed} \cdot (\Grad \log \vr_{\ep}-\Grad \log  r_{\ed}    )           \dx\ds \\
    & \quad 
                  +
                    \f{2\kappa \mu  }{\ep^2} \int_0^{\tau}  \!\!\!\int_{\Om} \vr_{\ep} \lr{ p'( r_{\ed}) \Grad \log r_{\ed} - p'( \vr_{\ep}) \Grad \log \vr_{\ep}}\cdot \left( \nabla \log r_{\ed} -\nabla \log b \right)           \dx\ds \\
		&=:\sum_{i=1}^{11} \widetilde{J}_i.
		\end{split}
	\end{align}  }

To proceed, we consider the terms $ \widetilde{J}_3$ and $ \widetilde{J}_4$: using \eqref{h3}$_2$ and \eqref{eq:test}$_2$, we can write 
	\begin{align}\label{ip1}
		{ \p_t \vW_{\ep} + \vU_{\ed} \cdot \nabla \vW_{\ep} }=& \lr{\p_t \vW_{\ep} + \vU \cdot \nabla \vW_{\ep} }+ \nabla \Phi_{\ed}\cdot \nabla \vW_{\ep}  ,  \\
		\label{ip2}		
{ \p_t \vV_{\ep} + \vU_{\ed} \cdot \nabla \vV_{\ed}}=& \lr{\p_t \vV_{\ep} + \vU \cdot \nabla \vV_{\ep} }  + \lr{ \nabla \Phi_{\ed} \cdot \nabla \vV_{\ep}  +\vU \cdot \nabla  \nabla \Phi_{\ed}+ \nabla \Phi_{\ed}\cdot \nabla  \nabla \Phi_{\ed}} .
	\end{align}
    Recall that $(b,\vU)$ is the strong solution to the target system \eqref{ane-Euler}, thus $\vV_\ep$ and $\vW_\ep$ satisfy the relations
\[
\partial_t\vW_\ep + \vU\cdot\nabla\vW_\ep = 2\sqrt{\kappa (1-\kappa)} \mu \vU\cdot \nabla\nabla\log b \quad \mbox{ and }\quad
\partial_t\vV_\ep + \vU\cdot\nabla\vV_\ep = - \Grad\Pi + 2 \kappa \mu \vU\cdot \nabla\nabla\log b.
\]
Therefore, we have
\begin{align*}
    &\int_0^{\tau}\!\!\!\int_{\Om}
    \vr_{\ep}(\vv_{\ep}-\vV_{\ed})\cdot
    \lr{\p_t  \vV_{\ep}+\vU\cdot\nabla \vV_{\ep}  }
    \dx\ds
    +\int_0^{\tau}\!\!\!\int_{\Om}
    \vr_{\ep}(\vw_{\ep}- \vW_{\ep} )\cdot
    \lr{\p_t  \vW_{\ep}+\vU\cdot\nabla  \vW_{\ep} }
    \dx\ds
    \\
    &\quad=-\int_0^{\tau}\!\!\!\int_{\Om}
    \vr_{\ep}(\vv_{\ep}-\vV_{\ed})\cdot\Grad\Pi\dx\ds
    \\
    &\qquad+ 2 \k \mu \int_0^{\tau}\!\!\!\int_{\Om}
    \vr_{\ep}(\vv_{\ep}-\vV_{\ed})\cdot \vU\cdot\nabla \nabla\log b\dx\ds
    \\
    &\qquad+2 \sqrt{\k (1-\k)} \mu \int_0^{\tau}\!\!\!\int_{\Om}
    \vr_{\ep}(\vw_{\ep}-\vW_{\ep} )\cdot \vU\cdot\nabla \nabla\log b\dx\ds .
\end{align*}

By \eqref{ip1}--\eqref{ip2}, the pressure and acoustic transport terms give
$\mathcal I_3$--$\mathcal I_5$, while $\nabla^2\Phi_{\ed}$ in the symmetric
stress gives $\mathcal I_{14}$. For $(\vV_\ep,\vW_\ep)$ with $\mu=\mu_\ep$,
the Euler equation adds $-2\mu b^{-1}\Div(b\bd\vU)$ to the right-hand
side of the $\vV$-equation in \eqref{target_new}. Pairing this term with
$\vr_\ep(\vV_{\ed}-\vv_\ep)$ gives $\mathcal I_{15}$.
The weighted integrations by parts \eqref{z1}--\eqref{z3}, with the velocity
differences taken relative to $(\vU_{\ed},\vV_{\ed},\vW_\ep)$, give the
three terms of \eqref{z}, namely $\mathcal I_9$--$\mathcal I_{11}$.
This proves \eqref{rel_entr13}.\qed


	\subsection{The strategy of the proof of Theorem \ref{TH2}}\label{subs:ill-p}

	Now, we are ready to prove Theorem \ref{TH2}. To this end, we must estimate the right-hand side of the following relative energy inequality: 
					\begin{align}\label{rel_entr13short}
			\begin{split}
				& \mathscr{E}( \vr_{\ep},\vv_{\ep},\vw_{\ep}| r_{\ed},\vV_{\ed},\vW_{\ep})(\tau)   -\mathscr{E}( \vr_{\ep},\vv_{\ep},\vw_{\ep}| r_{\ed},\vV_{\ed},\vW_{\ep} )(0)+\int_0^{\tau} \mathcal{D}_\ep \ds
                 \leq \sum_{i=1}^{15} \mathcal{I}_i,  
			\end{split}
		\end{align}
     where the terms $\mathcal{I}_{1},\ldots, \mathcal{I}_{15}$ are specified in \eqref{rel_entr13}. 
 Our aim is to show that each of these terms is small (terms bounded by $C_\delta\mu$ tend to zero by \eqref{eq:van-vis}) or can be controlled by the left-hand side using a Gr\"{o}nwall-type argument.

We introduce the following notations for some generic constants, functions and reminder terms:
	\begin{enumerate}
	
	\item $0 < C(\mathcal{P})$ is a constant that does not depend on $\delta$, $\ep$, or $\tau$, but depends on the set of parameters $\mathcal{P}$.
	
	\item $0 < \chi(\t)$ is a function that does not depend on $\ep$ and $\delta$, but depends on $\tau \in (0, T_\ast)$, and satisfies $\int_0^\tau \chi(s)\, \mathrm{d}s < \infty$.
	
	\item $\chi_{\ep}( \delta,\t)$ is a nonnegative function that satisfies  
	\begin{align}\label{gen-const-1}
\forall\,\t\in(0,T_\ast),\qquad		\chi_{\ep}(\delta, \cdot) \to 0 \quad \text{in } L^1(0, \tau) \quad \text{ as } \quad  \ep \to 0.
	\end{align}

        \item \( R_\ep(t) \) is a remainder term, in the sense defined in \eqref{eq:remainder}; this means that
\begin{equation} \label{eq:remainder:ill}
	\forall\, T < T_\ast, \qquad \lim_{\ep \to 0} \sup_{t \in [0, T]} R_\ep(t) = 0,
\end{equation}
where \( T_\ast > 0 \) denotes the time of existence of the smooth solution to the target problem  {\eqref{ane-Euler}}.
\end{enumerate}

Including the resulting viscosity remainders in $\chi_\ep(\delta,\cdot)$, from the estimates performed in the following section it will follow that
	\begin{align}\label{ill:out:gron}
    \begin{split}
		&\mathscr{E}( \vr_\ep,\vv_\ep,\vw_\ep| r_{\ep,\delta},\vV_{\ep,\delta},\vW_{\ep,\delta})(\tau)\\
		& \quad 
        \leq \lr{\mathscr{E}( \vr_{0,\ep},\vv_{0,\ep},\vw_{0,\ep}| r_{0,\ep,\delta},\vV_{0,\ep,\delta},\vW_{0,\ep,\delta}) +\sup_{t \in [0, \t]} R_\ep(t) 
			+ \int_0^\tau 
			\chi_{\ep}  (\delta,s)  \ds  } \\
		& \qquad  
		+
		\int_0^\tau \Big(\chi(s) +\chi_{\ep}(\delta,s)  \Big)  \mathscr{E}( \vr_{\ep},\vv_{\ep},\vw_{\ep}| r_{\ed},\vV_{\ed},\vW_{\ed})(s) \ds, 
    \end{split}
	\end{align}
for a suitable pair $(\ep,\delta)$, from which we will want to deduce: 
    \begin{align}\label{ill:out:main}
        \lim_{\delta \rightarrow 0} \limsup_{\ep\rightarrow 0} \mathscr{E}( \vr_\ep,\vv_\ep,\vw_\ep| r_{\ep,\delta},\vV_{\ep,\delta},\vW_{\ep,\delta})(\tau)=0.
    \end{align}

The choice of suitable dependency between $\delta$ and $\ep$ that allows us to pass to the limit follows the general idea discussed already before the derivation of \eqref{r-ep-del}. That is, for a fixed but sufficiently small $\delta$ we will choose correspondingly small $\ep$. More precisely:
\begin{itemize}
    \item We first fix  $\eta>0$, and choose $\delta_0>0$ such that for $0<\delta\leq \delta_0$, we have 
		\begin{align}\label{ill-ep-del1}
			\mathscr{E}( \vr_{0,\ep},\vv_{0,\ep},\vw_{0,\ep}| r_{0,\ep,\delta},\vV_{0,\ep,\delta},\vW_{0,\ep,\delta}) \leq \eta. 
		\end{align}
        \item Then, for any fixed $\delta\in(0,\delta_0]$, from \eqref{r-ep-del}, \eqref{gen-const-1} and \eqref{eq:remainder:ill}, we find $\ep_\delta > 0$ 
        such that, for all $0 < \ep < \ep_\delta$, we have
		\begin{align}\label{ill-ep-del2}
        &  \frac{\underline{b}}{2} \leq r_{\ed}   \leq 2\Ov{b} , \quad 
        \int_0^\tau \chi_{\ep}({\delta,s}) \ds \leq \eta , \quad 
        \sup_{t \in [0, \t]} R_\ep(t) \leq \eta.
		\end{align}
\end{itemize}
After this suitable choice of $\ep$ and $\delta$,  application of Gr\"{o}nwall’s inequality  to \eqref{ill:out:gron} implies
	\begin{align*}
		&\mathscr{E}( \vr_\ep,\vv_\ep,\vw_\ep| r_{\ep,\delta},\vV_{\ep,\delta},\vW_{\ep,\delta})(\tau) \\
		&\quad 
        \leq 
        \lr{
			\mathscr{E}( \vr_{0,\ep},\vv_{0,\ep},\vw_{0,\ep}| r_{0,\ep,\delta},\vV_{0,\ep,\delta},\vW_{0,\ep,\delta})+\sup_{t \in [0, \t]} R_\ep(t) 
			+ \int_0^\tau 
			\chi_{\ep}  (\delta,s)  \ds  } \times 
		\exp \lr{ \int_0^\tau \Big(\chi(s)+ \chi_{\ep}({\delta,s})\Big) \ds } .
	\end{align*}
Hence, from \eqref{ill-ep-del1} and \eqref{ill-ep-del2}, for $0< \ep<{\ep_\delta}$, we obtain 
		\begin{align*}
			\mathscr{E}( \vr_\ep,\vv_\ep,\vw_\ep| r_{\ep,\delta},\vV_{\ep,\delta},\vW_{\ep,\delta})(\tau) \leq 3\eta \exp\big(C T \lr{\eta+1}\big) \quad  \text{ for a.e. }\tau \in (0,T).
		\end{align*} 
Therefore, performing first the limit $\ep \to 0$ with $\delta$ fixed, and then letting $\delta \to 0$, allows us to get \eqref{ill:out:main}. 
In what follows, we stick to this suitable choice of $(\ep,\delta)$, without specifying it each time when it is used.

\subsection{The relative entropy estimate} \label{ss:rel-entr-ill}

    In this section we estimate the right-hand side of \eqref{rel_entr13}.
    As in Subsection \ref{subs:well}, we will write dependence of the arising multiplicative constants in terms of suitable norms of the test functions $\vV$ and $\vW$;
    however, they can be easily estimated in terms of similar norms of the target profiles $\vU$ and $b$. In addition, 
    we will use the following notation:
	\begin{align}\label{vrep1}
		\vr_\ep^{(1)} = \frac{\vr_\ep-b}{\ep} .
	\end{align}
	\paragraph{Estimate of terms $\mathcal{I}_1$ and $\mathcal{I}_2$:} Let us first recall the relation 
	\eq{\label{100UVW}
	\vU_{\ed} = \vV_{\ed} - \beta \vW_{\ep},
	}
	and a similar one  for the weak solution 
	\eq{\label{100uvw}
	\sqrt{\vr_\ep}\, \vu_\ep = \sqrt{\vr_\ep}\, \vv_\ep - \beta \sqrt{\vr_\ep}\, \vw_\ep.
	}
	Therefore, we have
	\begin{align*}
		|\mathcal{I}_1| &= \left\vert	\int_0^{\tau}  \!\!\!\int_{\Om}\vr_{\ep} \left(  \vu_{\ep} -\vU_{\ed}  \right)\cdot \Grad \vW_{\ep} \cdot (\vW_{\ep}-\vw_{\ep})\dx\ds \right\vert \\
		&= \left\vert	\int_0^{\tau}  \!\!\!\int_{\Om} \Big(   \sqrt{\vr_\ep}\lr{\vv_{\ep} -\vV_{\ed}}-\beta  \sqrt{\vr_\ep}\lr{\vw_{\ep}- \vW_{\ep}}  \Big)\cdot \Grad \vW_{\ep} \cdot  \sqrt{\vr_\ep}(\vW_{\ep}-\vw_{\ep})\dx\ds \right\vert \\
		&\leq  C \int_0^{\tau} \| \nabla \vW_{\ep} \|_{L^\infty_x} \lr{\int_{\Om}\vr_{\ep}  | \vv_{\ep}-\vV_{\ed}|^2 \dx } \ds +C \int_0^{\tau} \|\nabla \vW_{\ep} \|_{L^\infty_x} \lr{\int_{\Om}\vr_{\ep}  | \vw_{\ep}-\vW_{\ep}|^2 \dx }\ds \\
		&\leq  \int_0^\tau \chi(s) \mathscr{E}( \vr_{\ep},\vv_{\ep},\vw_{\ep}| r_{\ed},\vV_{\ed},\vW_{\ep})(s) \ds ,
	\end{align*} 
	where $\chi \in L^1(0,\tau)$ depends only on $\|\nabla \vW_{\ep}\|_{L^1_t L^\infty_x} = t \ C({\mu},\k) \ \left\|\nabla^2\log b\right\|_{L^\infty}$.

	By a similar argument, we estimate $\mathcal{I}_2 $ as 
	\begin{align*}
		|\mathcal{I}_2| & =
        \left\vert \int_0^{\tau}  \!\!\!\int_{\Om}\vr_{\ep}  \left(  \vu_{\ep} -\vU_{\ed}  \right)  \cdot \Grad \vV_{\ed} \cdot (\vV_{\ed}-\vv_{\ep})            \dx\ds\right\vert\\
		& 
        \leq \left\vert \int_0^{\tau}  \!\!\!\int_{\Om}\vr_{\ep}  \left(  \vu_{\ep} -\vU_{\ed}  \right)  \cdot \Grad \vV_{\ep} \cdot (\vV_{\ed}-\vv_{\ep})            \dx\ds\right\vert+ \left\vert \int_0^{\tau}  \!\!\!\int_{\Om}\vr_{\ep}  \left(  \vu_{\ep} -\vU_{\ed}  \right)  \cdot \Grad^2 \Phi_{\ed} \cdot (\vV_{\ed}-\vv_{\ep})            \dx\ds\right\vert \\
		&
        \leq  C \int_0^{\tau} \| \nabla \vV_{\ep} \|_{L^\infty_x} \lr{\int_{\Om}\vr_{\ep}  | \vv_{\ep}-\vV_{\ed}|^2 \dx} \ds +C \int_0^{\tau} \|\nabla \vV_{\ep} \|_{L^\infty_x} \lr{\int_{\Om}\vr_{\ep}  | \vw_{\ep}-\vW_{\ep}|^2 \dx }\ds \\ 
		&\quad 
        + C \int_0^{\tau} \| \nabla^2 \Phi_{\ed} \|_{L^\infty_x} \lr{\int_{\Om}\vr_{\ep}  | \vv_{\ep}-\vV_{\ed}|^2 \dx} \ds
		+
		C \int_0^{\tau} \|\nabla^2 \Phi_{\ed} \|_{L^\infty_x}\lr{ \int_{\Om}\vr_{\ep}  | \vw_{\ep}-\vW_{\ep}|^2 \dx} \ds \\
		& 
        \leq  \int_0^\tau \chi(s)   \mathscr{E}( \vr_{\ep},\vv_{\ep},\vw_{\ep}| r_{\ed},\vV_{\ed},\vW_{\ep})(s) \ds + C \int_0^\tau  \|\nabla^2 \Phi_{\ed} (s)\|_{L^\infty_x} \mathscr{E}( \vr_{\ep},\vv_{\ep},\vw_{\ep}| r_{\ed},\vV_{\ed},\vW_{\ep})(s) \ds.
	\end{align*}
	It follows from the dispersive estimate \eqref{acc-wave-dis} that, for any $\delta>0$ fixed, one has
	\begin{align*}
    \|\nabla^2 \Phi_{\ed} (s)\|_{L^1_tL^\infty_x} \rightarrow 0  \quad 
    \quad  \text{ as } \ep \rightarrow 0 .
    \end{align*}
	Hence, by recalling the notations introduced in \eqref{gen-const-1}, we have 
	\begin{align*}
		|\mathcal{I}_2|  \leq  \int_0^\tau \chi(s)  \mathscr{E}( \vr_{\ep},\vv_{\ep},\vw_{\ep}| r_{\ed},\vV_{\ed},\vW_{\ep})(s) \ds +  \int_0^\tau \chi_{\ep}(s,\delta) \mathscr{E}( \vr_{\ep},\vv_{\ep},\vw_{\ep}| r_{\ed},\vV_{\ed},\vW_{\ep})(s) \ds .
	\end{align*}
	Summing the two estimates up, it holds
	\begin{align}\label{i1-i2}
		|\mathcal{I}_1| + |\mathcal{I}_2|  \leq   \int_0^\tau \chi(s)  \mathscr{E}( \vr_{\ep},\vv_{\ep},\vw_{\ep}| r_{\ed},\vV_{\ed},\vW_{\ep})(s) \ds+  \int_0^\tau \chi_{\ep}(s,\delta) \mathscr{E}( \vr_{\ep},\vv_{\ep},\vw_{\ep}| r_{\ed},\vV_{\ed},\vW_{\ep})(s) \ds .
	\end{align}

	\paragraph{Estimate of term $\mathcal{I}_3$:} Using relation \eqref{h3}$_2$, we get 
	\begin{align*}
		\vV_{\ed}= \vU+ \beta \vW_{\ep}+ \nabla \Phi_{\ed} .
	\end{align*}
	This, together with \eqref{100uvw}, enables us to rewrite $\mathcal{I}_3$ in the following form: 
	\begin{align*}
		{-\mathcal{I}_3}&=\int_0^{\tau}  \!\!\!\int_{\Om}\vr_{\ep} (\vV_{\ed}-\vv_{\ep}) \cdot \nabla \Pi  \dx\ds \\
		&=\int_0^{\tau}  \!\!\!\int_{\Om}\vr_{\ep} (\vU-\vu_{\ep}) \cdot\nabla \Pi  \dx\ds+ \beta \int_0^{\tau}  \!\!\!\int_{\Om}\vr_{\ep} (\vW_{\ep}-\vw_{\ep}) \cdot \nabla \Pi  \dx\ds + \int_0^{\tau}  \!\!\!\int_{\Om}\vr_{\ep}  \nabla \Phi_{\ed}  \cdot\nabla \Pi  \dx\ds \\
		&=\int_0^{\tau}  \!\!\!\int_{\Om} (\vr_{\ep} \vU-\vr_{\ep} \vu_{\ep} )\cdot  \nabla \Pi  \dx\ds + \beta \int_0^{\tau}  \!\!\!\int_{\Om} \Big( (\vr_{\ep}-b)\vW_{\ep}+(b\vW_{\ep}-\vr_{\ep}\vw_{\ep}) \Big) \cdot \nabla \Pi  \dx\ds \\
		&\qquad + \int_0^{\tau}  \!\!\!\int_{\Om}\vr_{\ep}  \nabla \Phi_{\ed}  \cdot\nabla \Pi  \dx\ds \\
        & 
        =:\sum_{i=1}^3 \mathcal{I}_{3,i} .
	\end{align*}
	Because the terms $\mathcal{I}_{3,1}$ and $\mathcal{I}_{3,2}$ do not depend on the acoustic part at all, they can be estimated exactly as the term $\mathcal{A}_1$ in the case for well-prepared data, see Section \ref{Sec:5.2}. In view of the  uniform bounds \eqref{bound:vw}, \eqref{bound:rhou} and the weak convergences \eqref{eq:conv_prop:ill-2}, the adaptation of this argument is straightforward. Notice that here $\vr_\ep\vw_\ep=2\mu\sqrt{\kappa(1-\kappa)}\Grad\vr_\ep\to0$ in the sense of distributions. Therefore,
    \begin{align*}
	\left|\mathcal{I}_{3,1}\right| + \left|\mathcal{I}_{3,2}\right|
    \leq\sup_{t\in[0,\tau]} R_\ep(t) . 
	\end{align*}
    
    The term $\mathcal{I}_{3,3} $ is the most delicate one, so we need to proceed carefully.
    To begin with, we make use of \eqref{vrep1} to write
	\begin{align*}
	\mc I_{3,3}
&= \int_0^{\tau}  \!\!\!\int_{\Om}b  \nabla \Phi_{\ed}  \cdot\nabla \Pi  \dx\ds +
\veps \int_0^{\tau}  \!\!\!\int_{\Om} \vr_{\ep}^{(1)}  \nabla \Phi_{\ed}  \cdot\nabla \Pi  \dx\ds .
	\end{align*}
	By cutting into essential and residual set, and using estimates \eqref{bound:rhou}, \eqref{acc-wave-reg} and \eqref{acc-wave-dis}, we easily see that
	\begin{align*}
\veps   \left|\int_0^{\tau}  \!\!\!\int_{\Om} \vr_{\ep}^{(1)}  \nabla \Phi_{\ed}  \cdot\nabla \Pi  \dx\ds\right| &\leq 
\ep \left(\| [\vr_{\ep}^{(1)}]_{\ess}\|_{L^\infty_t L^2_x} \|\nabla \Pi \|_{L^\infty_t L^2_x}  
        \int_0^{\tau} \| \nabla \Phi_{\ed}  \|_{L^\infty_x} \ds \right. \\
  &\qquad\qquad  \left.   + 
    \| [\vr_{\ep}^{(1)}]_{\res}\|_{L^\infty_t L^\gamma_x} \|\nabla \Pi \|_{L^\infty_t L^{\gamma^\prime}_x}  \int_0^{\tau} \| \nabla \Phi_{\ed}  \|_{L^\infty_x} \ds\right) \\
    &\leq \int^\tau_0\chi_\veps(\delta,s) \ds.
	\end{align*}

	For the first term appearing in the expression of $\mc I_{3,3}$, instead, we can integrate by parts and take advantage of the target equation {\eqref{ane-app2-Euler}}.
In this way, we obtain
\begin{align*}
 \int_0^{\tau}  \!\!\!\int_{\Om}b  \nabla \Phi_{\ed}  \cdot\nabla \Pi  \dx\ds &= 
 -\int_0^{\tau}  \!\!\!\int_{\Om} \Phi_{\ed}  \div\Big(b\ \nabla \Pi\Big)  \dx\ds \\
 &=\int_0^{\tau}  \!\!\!\int_{\Om} \Phi_{\ed}  \div\Big(b \ \vU\cdot\nabla\vU\Big)  \dx\ds. 
\end{align*}
Since the  term on the right is quadratic, we can simply bound it as follows:
\begin{align*}
 \left|\int_0^{\tau}  \!\!\!\int_{\Om} \Phi_{\ed}  \div\Big(b \ \vU\cdot\nabla\vU\Big)  \dx\ds\right| &\leq \left\|b\right\|_{W^{1,\infty}}
 \left\|\vU\right\|^2_{L^\infty_tH_x^2} \int_0^{\tau} \| \nabla \Phi_{\ed}  \|_{L^\infty_x} \ds \\
 &\leq \int^\tau_0\chi_\veps(\delta,s) \ds,
\end{align*}
where we used again the dispersive estimates from \eqref{acc-wave-dis} in the last step.

All in all, we have proved that
    \begin{align*}
        | \mathcal{I}_{3,3}| \leq \int_0^\tau \chi_{\ep}  (\delta,s) \ds ;
    \end{align*}
whence we finally obtain
	\begin{align}\label{I3}
		 \left|\mathcal{I}_3\right|\leq \left|\mathcal{I}_{3,1}\right|+ \left|\mathcal{I}_{3,2}\right| + |\mathcal{I}_{3,3}|
		\leq 
		\sup_{t \in [0, \t]} R_\ep(t)+ \int_0^\tau \chi_{\ep}  (\delta,s) \ds . 
	\end{align}

	\paragraph{Estimate of terms $\mathcal{I}_4$ and $\mathcal{I}_5$:} 
We notice that all the parts of terms $\mathcal{I}_4$ and $\mathcal{I}_5$ can be handled in the same way, and so we estimate only one of them. 
\begin{align*}
		&\int_0^{\tau}  \!\!\!\int_{\Om} \vr_{\ep} (\vv_{\ep}-\vV_{\ed}) \cdot \vU \cdot \nabla  \nabla \Phi_{\ed}    \dx\ds  \\
		&
		\quad 
		= \int_0^{\tau}  \!\!\!\int_{\Om} \sqrt{\vr_{\ep}} (\vv_{\ep}-\vV_{\ed}) \vU \cdot \lr{\sqrt{\vr_{\ep}} } \nabla^2 \Phi_{\ed}    \dx\ds \\
		&
		\quad
		\leq 
			\int_0^{\tau} 
        \|  \nabla^2 \Phi_{\ed}  (s) \|_{L^\infty_x } \lr{\int_{\Om}  \vr_{\ep} | \vU|^2 | \nabla^2 \Phi_{\ed}  | \dx}\ds + 	\int_0^{\tau}  \!\!\!\int_{\Om} \vr_{\ep} |\vv_{\ep}-\vV_{\ed}|^2   \dx\ds \\
           &
		\quad \leq \int_0^{\tau} \|  \nabla^2 \Phi_{\ed}  (s) \|_{L^\infty_x } \lr{\int_{\Om} \left(b+\ep \vr_{\ep}^{(1)}\right) | \vU|^2 | \nabla^2 \Phi_{\ed}  | \dx}\ds + 	\int_0^{\tau}  \!\!\!\int_{\Om} \vr_{\ep} |\vv_{\ep}-\vV_{\ed}|^2   \dx\ds \\ 
        &\quad\leq \int_0^{\tau} \|  \nabla^2 \Phi_{\ed}  (s) \|_{L^\infty_x } \lr{\int_{\Om} b | \vU|^2 | \nabla^2 \Phi_{\ed}  | \dx}\ds +\ep \int_0^{\tau} 
        \|  \nabla^2 \Phi_{\ed}  (s) \|_{L^\infty_x } \lr{\int_{\Om} |[\vr_{\ep}^{(1)}]_{\text{ess}}| | \vU|^2 | \nabla^2 \Phi_{\ed}  | \dx}\ds\\
        &\quad\quad +\ep \int_0^{\tau}
        \|  \nabla^2 \Phi_{\ed}  (s) \|_{L^\infty_x } \lr{\int_{\Om} | [\vr_{\ep}^{(1)}]_{\text{res}}| | \vU|^2 | \nabla^2 \Phi_{\ed}  | \dx}\ds + 	\int_0^{\tau}  \!\!\!\int_{\Om} \vr_{\ep} |\vv_{\ep}-\vV_{\ed}|^2   \dx\ds .
	\end{align*} 
    Using \eqref{bound:rhou}, \eqref{acc-wave-reg}, Sobolev embeddings, and the dispersive estimate of $\Phi_{\ed}$ from \eqref{acc-wave-dis}, we obtain
    \begin{align*}
        &\left|\int_0^{\tau}  \!\!\!\int_{\Om} \vr_{\ep} (\vv_{\ep}-\vV_{\ed}) \cdot \vU \cdot \nabla  \nabla \Phi_{\ed}    \dx\ds \right| \\
        & \quad 
        \leq C(b) \| \vU \|^2_{L^\infty_t L^4_x}
        \|\nabla^2 \Phi_{\ed} \|_{L^\infty_t L^2_x}  
        \int_0^{\tau} \| \nabla^{{2}} \Phi_{\ed} \|_{L^\infty_x} \ds \\
        &\qquad
        +\ep \| [\vr_{\ep}^{(1)}]_{\text{ess}}\|_{L^\infty_t L^2_x} \|\vU \|_{L^\infty_t L^\infty_x}^2 \|\nabla^2 \Phi_{\ed} \|_{L^\infty_t L^2_x} \int_0^{\tau} \| \nabla^2 \Phi_{\ed}  \|_{L^\infty_x} \ds \\
        &\qquad
        +\ep \| [\vr_{\ep}^{(1)}]_{\text{res}}\|_{L^\infty_t L^{{1}}_x}\|\vU \|_{L^\infty_t L^\infty_x}^2 \|\nabla^2 \Phi_{\ed} \|_{L^\infty_t L^{{\infty}}_x}  \int_0^{\tau} \| \nabla^2 \Phi_{\ed}  \|_{L^\infty_x} \ds \\
        &\qquad+ 	\int_0^{\tau}  \!\!\!\int_{\Om} \vr_{\ep} |\vv_{\ep}-\vV_{\ed}|^2   \dx\ds \\
        & \quad 
        \leq \int_0^\tau \chi_{\ep}  (\delta,s) \ds
        + \sup_{t \in [0, \t]} R_\ep(t)+  \int_0^\tau  \chi(s) \mathscr{E}( \vr_{\ep},\vv_{\ep},\vw_{\ep}| r_{\ed},\vV_{\ed},\vW_{\ep})(s) \ds ,
    \end{align*}
which gives rise to 
	\begin{align}\label{I4-I5}
		|\mathcal{I}_4|+|\mathcal{I}_5| \leq \int_0^\tau \chi_{\ep}  (\delta,s) \ds
        + \sup_{t \in [0, \t]} R_\ep(t)+  \int_0^\tau  \chi(s) \mathscr{E}( \vr_{\ep},\vv_{\ep},\vw_{\ep}| r_{\ed},\vV_{\ed},\vW_{\ep})(s) \ds .
	\end{align}

	\paragraph{Estimate of  term $\mathcal{I}_6$:} From the fact that $p(\vr)=(\gamma-1) P(\vr)$, we have 
	\begin{align*}
		\Big(
        p(\vr_{\ep})-p( r_{\ed}) -  p'( r_{\ed})(\vr_{\ep}- r_{\ed}) 
        \Big)
        = (\gamma-1)
        \Big(
        P(\vr_{\ep})-P( r_{\ed}) -  P'( r_{\ed})(\vr_{\ep}- r_{\ed}) 
         \Big).
	\end{align*}
	Hence, it holds 
	\begin{align*}
		\mathcal{I}_6&=	-
        \f{1}{\ep^2} \int_0^{\tau}  \!\!\!\int_{\Om}
        \Big( 
        p(\vr_{\ep})-p( r_{\ed}) -  p'( r_{\ed})(\vr_{\ep}- r_{\ed})
        \Big)
        \Div \,\vU_{\ed}\dx\ds \\ 
		& 
		=
		-
        \f{\gamma-1}{\ep^2} \int_0^{\tau}  \!\!\!\int_{\Om}
        \Big(
        P(\vr_{\ep})-P( r_{\ed}) -  P'( r_{\ed})(\vr_{\ep}- r_{\ed}) 
        \Big) \Div \, \vU\dx\ds \\
		&
		\quad 
		-
		\f{\gamma-1}{\ep^2} \int_0^{\tau}  \!\!\!\int_{\Om}
        \Big( 
        P(\vr_{\ep})-P( r_{\ed}) -  P'( r_{\ed})(\vr_{\ep}- r_{\ed})
        \Big)
        \Delta \Phi_{\ed}\dx\ds \\
		&
		=: 
		\mathcal{I}_{6,1} + \mathcal{I}_{6,2}.
	\end{align*} 
	It is easily seen that  
	\begin{align*}
		|\mathcal{I}_{6,1} | 
        \leq  
        \int_0^\tau  \chi(s) \, 
        \mathscr{E}( \vr_{\ep},\vv_{\ep},\vw_{\ep}| r_{\ed},\vV_{\ed},\vW_{\ep})(s) \ds .
	\end{align*}
	Here, we notice that the $L^1$-norm of $\chi$ depends on $ \|  \Div \,\vU  \|_{L^1_t L^\infty_x} $. 
    By applying again the dispersive estimate \eqref{acc-wave-dis} and invoking the notation \eqref{gen-const-1}, it follows that
	\begin{align*}
		|\mathcal{I}_{6,2}| &\leq  
        (\gamma-1)
        \int_0^\tau 
        \|  \Delta \Phi_{\ed} \|_{L^\infty_x} \, \mathscr{E}( \vr_{\ep},\vv_{\ep},\vw_{\ep}| r_{\ed},\vV_{\ed},\vW_{\ep})(s) \ds \nonumber \\
		&\leq  
        \int_0^\tau \chi_{\ep}(\delta,s) \, \mathscr{E}( \vr_{\ep},\vv_{\ep},\vw_{\ep}| r_{\ed},\vV_{\ed},\vW_{\ep})(s) \ds .
	\end{align*}
Combining the two estimates above, we obtain
	\begin{align}\label{I6}
		|\mathcal{I}_6| \leq  \int_0^\tau 
       \Big( \chi(s) +\chi_{\ep}(\delta,s) \Big) \,
        \mathscr{E}( \vr_{\ep},\vv_{\ep},\vw_{\ep}| r_{\ed},\vV_{\ed},\vW_{\ep})(s) \ds .
	\end{align}

	\paragraph{Estimate of term $\mathcal{I}_7$:} 

    For fixed $\delta $ and sufficiently small $\ep$, we infer from  \eqref{r-ep-del} that 
	\begin{align*}
		\|  P^{\prime \prime}( r_{\ed})  \|_{ L^\infty_{t,x} }  
        \leq
        C(\delta, \Ov{b}, \underline{b} ) .
	\end{align*}
	This, together with definitions \eqref{eq:test}, gives rise to  
	\begin{align*}
		\mathcal{I}_7&= -\f{1}{\ep} \int_0^{\tau}  \!\!\!\int_{\Om}\lr{ P^{\prime \prime}( r_{\ed} )(\vr_{\ep}- r_{\ed}) \Div\lr{ s_{\ed} \nabla \Phi_{\ed}} }\dx\ds
		\\
		& 
		\leq C(\delta, \Ov{b}, \underline{b} )
		\int_0^{\tau}  \!\!\!\int_{\Om} \left( \left| \f{\vr_{\ep}-b}{\ep}   \right| + |s_{\ed}|  \right) \left(	|s_{\ed}|+ |\Grad s_{\ed}|	\right)  \left(  |\Grad \Phi_{\ed}| +|\Delta \Phi_{\ed}|   \right) \dx\ds \\
		&\leq C(\delta, \Ov{b}, \underline{b} ) 
		\int_0^{\tau}  \!\!\!\int_{\Om}	\left| \f{\vr_{\ep}-b}{\ep}   \right| 	\left(	|s_{\ed}|+ |\Grad s_{\ed}|	\right)  \left( |\Grad \Phi_{\ed}| +|\Delta \Phi_{\ed}|   \right)	\dx\ds \\
		&\quad +  C(\delta, \Ov{b}, \underline{b} )
		\int_0^{\tau}  \!\!\!\int_{\Om} 	 |s_{\ed}|  	\left(	|s_{\ed}|+ |\Grad s_{\ed}| \right)  \left(  |\Grad \Phi_{\ed}| +|\Delta \Phi_{\ed}|   \right)
		\dx\ds \\
		&\leq C(\delta, \Ov{b}, \underline{b} )  
		\int_0^{\tau} 
        \left\| 
        \left[   \f{\vr_{\ep}-b}{\ep}  \right]_{\text{ess}} \right\|_{L^2_x}  
        \| \left(	|s_{\ed}|+ |\Grad s_{\ed}|	\right)  \|_{L^2_x } 
        \| \left( |\Grad \Phi_{\ed}| +|\Delta \Phi_{\ed}|   \right) \|_{L^\infty_x } 
        \ds \\
		&\quad +C(\delta, \Ov{b}, \underline{b} )  
		\int_0^{\tau} \ep 
        \left\| 
        \left[ \f{1+\vr_{\ep}^\gamma}{\ep^2} \right]_{\text{res}} \right\|_{L^1_x}  \| \left(	|s_{\ed}|+ |\Grad s_{\ed}|	\right)  \|_{L^\infty_x } 
        \| 
        \left( |\Grad \Phi_{\ed}| +|\Delta \Phi_{\ed}|   \right)
        \|_{L^\infty_x }  \ds \\
		&\quad + C(\delta, \Ov{b}, \underline{b} ) \int_0^{\tau} \|  s_{\ed} \|_{L^2_x}  \| \left(	|s_{\ed}|+ |\Grad s_{\ed}|	\right)  \|_{L^2_x }
        \| \left( |\Grad \Phi_{\ed}| +|\Delta \Phi_{\ed}|   \right) \|_{L^\infty_x }  \ds. 
\end{align*}
Using the uniform bounds of $\displaystyle \left[\frac{\vr_{\ep} - b}{\ep} \right]_{\text{ess}} $ and $ \displaystyle \left[ \f{1+\vr_{\ep}^\gamma}{\ep^2} \right]_{\text{res}} $ in \eqref{bound:rhou}, together with the energy and dispersive estimates \eqref{acc-wave-reg} and \eqref{acc-wave-dis}, we obtain
	\begin{align}\label{I7}
		|\mathcal{I}_7| 
        \leq 
		\int_0^\tau \chi_{\ep } (\delta,s)  \ds  .
	\end{align}

	\paragraph{Estimate of term $\mathcal{I}_8$:} 
	A direct computation shows that
	\begin{align}\label{ill8}
		\begin{split}
		&\Grad 
        \Big(
        P^\prime( r_{\ed})-P^\prime(b)-( r_{\ed}-b)P^{\prime \prime}(b)
        \Big)  \\
		&
		\quad 
		= \ep 
        \Big(
        P^{\prime\prime}( r_{\ed})-P^{\prime\prime}(b) 
        \Big) \nabla  s_{\ed} + 
        \Big(
        P^{\prime\prime}( r_{\ed})-P^{\prime\prime}(b)-P^{\prime\prime\prime}(b)( r_{\ed}-b)
        \Big)
        \nabla b.
	\end{split}
	\end{align}
	From the choice of $\delta$ and $\ep$ leading to \eqref{r-ep-del}, we see 
	\begin{align*}
		\|  P^{(3)}( \xi )  \|_{L^\infty_{t,x}} + 	\|  P^{(4)}( \xi )  \|_{ L^\infty_{t,x}}  \leq C(\delta, \Ov{b}, \underline{b} ) ,
	\end{align*}
for any $\xi \in \big(\min\{r_{\ed} (t,x), b(x)\},\max\{r_{\ed} (t,x), b(x)\}\big)$.

	Based on the observation above and Taylor's formula, from \eqref{ill8} we deduce that 
	\begin{align*}
	|	\mathcal{I}_8| 
    &\leq 
    \frac{1}{\ep^2} \int_0^{\tau}  \!\!\!\int_{\Om} \vr_{\ep}  \left\vert \Grad \lr{ P^\prime( r_{\ed})-P^\prime(b)-( r_{\ed}-b)P^{\prime \prime}(b)} \right\vert |(\vv_{\ep}-\vV_{\ed}) |\dx \ds \\
	&\leq 
    \frac{1}{\ep^2}
   ({1+}\|\Grad b\|_{L^\infty_x})
    \int_0^{\tau}  \!\!\!\int_{\Om}\vr_{\ep}  \lr{\ep^2 C(\delta, \Ov{b}, \underline{b} )  }  \lr{ s_{\ed}|\nabla  s_{\ed}|+  s_{\ed}^2} |\vv_{\ep}-\vV_{\ed}| \dx\ds \\
	& \leq C(\delta, \Ov{b}, \underline{b} ,1+ \|\Grad b\|_{L^\infty_x}) \int_0^{\tau}  \!\!\!\int_{\Om}\vr_{\ep}  \lr{ s_{\ed}|\nabla  s_{\ed}|+  s_{\ed}^2}  |\vv_{\ep}-\vV_{\ed}| \dx\ds. 
    \end{align*}
    Seeing further that 
    \begin{align*}
	|	\mathcal{I}_8| 
	& \leq   {\tfrac12} \int_0^{\tau}  \!\!\!\int_{\Om}\vr_{\ep}    |\vv_{\ep}-\vV_{\ed}| ^2  \dx\dt  +  C(\delta, \Ov{b}, \underline{b},\|\Grad b\|_{L^\infty_x}) \int_0^{\tau}  \!\!\!\int_{\Om}  \lr{|s_\ed |^4+|s_\ed |^2|\nabla s_{\ed}|^2} \mathbf{1}_{\text{ess}}\dx \dt  \\
    &\quad  +  C(\delta, \Ov{b}, \underline{b},\|\Grad b\|_{L^\infty_x}) \int_0^{\tau}  \!\!\!\int_{\Om} \vr_\ep \lr{|s_\ed |^4+|s_\ed |^2|\nabla s_{\ed}|^2} \mathbf{1}_{\text{res}}\dx \dt\\
	&\leq   {\tfrac12} \int_0^{\tau}  \int_{\Om} \vr_{\ep}    |\vv_{\ep}-\vV_{\ed}| ^2 \dx\dt  +    C(\delta, \Ov{b}, \underline{b}) \int_0^{\tau}  \!\!\!\int_{\Om}  \lr{|s_\ed |^4+|s_\ed |^2|\nabla s_{\ed}|^2} \mathbf{1}_{\text{ess}}\dx \dt \\
    &\quad + C(\delta, \Ov{b}, \underline{b},\|\Grad b\|_{L^\infty_x}) \int_0^{\tau}  \!\!\!\int_{\Om} [1+\vr_{\ep}^\gamma]_{\text{res}}\lr{|s_\ed |^4+|s_\ed |^2|\nabla s_{\ed}|^2} \dx \dt .
	\end{align*} 
	Recalling again that $\displaystyle \int_{\Om}  \left[ \frac{1+\vr_\ep^\gamma}{\ep^2}  \right]_{\text{res}} \dx $ is uniformly bounded for a.e $t\in (0,\tau)$ and using similar argument of $\mathcal{I}_7$, we arrive at 
	\begin{align}\label{I8}
		|\mathcal{I}_8| \leq   \int_0^\tau
        \Big( 
        \chi (s) + \chi_{\ep} 
		(\delta,s)   \Big)  \, \mathscr{E}( \vr_{\ep},\vv_{\ep},\vw_{\ep}| r_{\ed},\vV_{\ed},\vW_{\ep})(s) \ds +  \int_0^\tau 
		\chi_{\ep} 
		(\delta,s)  \ds .
	\end{align}

	\paragraph{Estimate of terms $\mathcal{I}_{9}$, $\mathcal{I}_{10}$, $\mathcal{I}_{11}$:} 
	All these terms  can be bounded by the left-hand side of \eqref{rel_entr13}. Indeed, recalling from Subsection \ref{Sec:5.2} that
	\begin{align*}
	\sqrt{\vr_\ep} \Big(\nabla\log\vr_\ep - \nabla\log b\Big) = \frac{1}{2{\mu}\sqrt{\k (1-\k)}} \sqrt{\vr_\ep}\Big(\vw_\ep- \vW_{\ep} \Big),
	\end{align*}
	 we may estimate
	\begin{align}\label{I9-11}
    \begin{split}
		&\left|\mathcal{I}_{9} \right| + \left|\mathcal{I}_{10} \right| + \left|\mathcal{I}_{11}\right|\\ 
		& \quad 
        \leq  \left|\frac{1}{\sqrt{\kappa(1-\kappa)}}\int_0^{\tau}\!\!\!\int_{\Om}  \sqrt{\vr_\ep}\Big(\vw_\ep- \vW_{\ep} \Big) \otimes \sqrt{\vr_{\ep}} (\vu_{\ep}-\vU_{\ed}): \Big[ 
		(1-\kappa)\bd \vV_{\ep}  	     -    \sqrt{\kappa (1-\kappa)}  \Grad \vW_{\ep}         \Big]    \dx\ds  \right|\\
		& \qquad   +\left|\frac{\sqrt{\kappa}}{\sqrt{1-\kappa}}\int_0^{\tau} \!\!\!\int_{\Om} \sqrt{\vr_\ep}\Big(\vw_\ep- \vW_{\ep} \Big)\otimes \sqrt{\vr_{\ep}} (\vv_{\ep}-\vV_{\ed}) :     \ba \vV_{\ep} \dx\dt \right|  \\
		& \qquad  + \left|  \int_0^{\tau}  \!\!\!\int_{\Om}\sqrt{\vr_\ep}\Big(\vw_\ep- \vW_{\ep} \Big)\otimes \sqrt{\vr_{\ep}}  (\vw_{\ep}-\vW_{\ep}) :   \ba \vV_{\ep}   \dx \ds \right|\\
		& \quad 
        \leq	   \int_0^{\tau}  \chi(s)   \, \mathscr{E}( \vr_{\ep},\vv_{\ep},\vw_{\ep}| r_{\ed},\vV_{\ed},\vW_{\ep})(s) \ds,
        \end{split}
	\end{align}
	where  $\chi>0$ and $\displaystyle \| \chi \|_{L^1(0,\tau)}$ depends only on  $\kappa$  and on the $L^1(0,T;L^\infty(\Omega))$-norms of $\Grad \vV_{\ep}$ and $\Grad\vW_{\ep}$.

\paragraph{Estimate of term $\mathcal{I}_{12}$:} 
	To control $ \mathcal{I}_{12} $, we first invoke the identity 
	\begin{align*}
		&\vr_{\ep}  \Big( p'(\vr_{\ep})-p'( r_{\ed})\Big) \Grad \log r_{\ed}  \cdot (\Grad \log \vr_{\ep}-\Grad \log r_{\ed}    )      \\
		& \quad 
		= \nabla\lr{ p(\vr_{\ep})-p( r_{\ed}) -  p'( r_{\ed})(\vr_{\ep}- r_{\ed}) } \cdot \nabla \log r_{\ed} \\
		& 
		\qquad 
		- \Big(
        \vr_{\ep} (p^\prime(\vr_{\ep}) - p^\prime( r_{\ed}))- p^{\prime \prime}( r_{\ed}) (\vr_{\ep}- r_{\ed}) r_{\ed} 
        \Big)  \vert \nabla \log r_{\ed} \vert^2.
	\end{align*}
	Based on the following identities,
	\begin{align*}
		\nabla \log  r_{\ed}= \frac{\nabla  r_{\ed}}{ r_{\ed}}
        \quad \text{ and }	
        \quad 
        \Delta \log  r_{\ed} = \frac{\Delta  r_{\ed}}{ r_{\ed}}- \frac{|\nabla  r_{\ed}|^2}{ r_{\ed}^2},
	\end{align*}
	and
	\begin{align*}
		&\vr_{\ep} (p^\prime(\vr_{\ep}) - p^\prime( r_{\ed}))- p^{\prime \prime}( r_{\ed}) (\vr_{\ep}- r_{\ed}) r_{\ed} \\
		& \quad 
		=\vr_\ep p^\prime(\vr_\ep)-  r_{\ed} p^\prime ( r_{\ed}) -(\vr_\ep- r_{\ed})\lr{p^{\prime}( r_{\ed} )+p^{\prime\prime}( r_{\ed})  r_{\ed} } \\
		& 
		\quad 
		=\gamma(\gamma-1)\lr{  P(\vr_{\ep})-P( r_{\ed}) -  P'( r_{\ed})(\vr_{\ep}- r_{\ed}) }, 
	\end{align*}
	we deduce
	\begin{align*}
		|\mathcal{I}_{12}| &
        \leq 
         {2\kappa\mu}\f{\gamma-1}{\ep^2} \int_0^{\tau}  \!\!\!\int_{\Om}\lr{ P(\vr_{\ep})-P( r_{\ed}) -  P'( r_{\ed})(\vr_{\ep}- r_{\ed}) } \lr{\frac{|\Delta  r_{\ed}|}{ r_{\ed}}+ \frac{|\nabla  r_{\ed}|^2}{ r_{\ed}^2}}\dx\ds \\
		& \quad 
        + {2\kappa\mu}\f{\gamma(\gamma-1)}{\ep^2} \int_0^{\tau}  \!\!\!\int_{\Om}\lr{ P(\vr_{\ep})-P( r_{\ed}) -  P'( r_{\ed})(\vr_{\ep}- r_{\ed}) } \f{|\nabla  r_{\ed}|^2}{ r_{\ed}^2}\dx\ds \\
		& \leq 
        \f{C}{\ep^2} \int_0^{\tau}  \!\!\!\int_{\Om}\lr{ P(\vr_{\ep})-P( r_{\ed}) -  P'( r_{\ed})(\vr_{\ep}- r_{\ed}) } \frac{|\Delta b|+ \ep |\Delta   s_{\ed}|}{ r_{\ed}}\dx\ds \\
		& \quad 
        +\f{C}{\ep^2} \int_0^{\tau}  \!\!\!\int_{\Om}\lr{ P(\vr_{\ep})-P( r_{\ed}) -  P'( r_{\ed})(\vr_{\ep}- r_{\ed}) } \f{|\nabla b|^2+\ep^2  |\nabla  s_{\ed}|^2}{ r_{\ed}^2}\dx\ds  \\
		& =:
		 \,
		\mathcal{I}_{12,1}+ \mathcal{I}_{12,2}.
	\end{align*}
	From \eqref{r-ep-del} and the dispersive estimate \eqref{acc-wave-dis}, we obtain that
	\begin{align*}
		|\mathcal{I}_{12,1}| &\leq 
        \f{C}{\ep^2} \left|\int_0^{\tau}  \!\!\!\int_{\Om}  \lr{ P(\vr_{\ep})-P( r_{\ed}) -  P'( r_{\ed})(\vr_{\ep}- r_{\ed}) } \frac{|\Delta b| }{ r_{\ed}}   \dx\ds \right|\\
		& \quad 
        + \ep \f{C}{\ep^2} \left|\int_0^{\tau}  \!\!\!\int_{\Om}  \lr{ P(\vr_{\ep})-P( r_{\ed}) -  P'( r_{\ed})(\vr_{\ep}- r_{\ed}) } \frac{  |\Delta   s_{\ed}|}{ r_{\ed} } \dx\ds
         \right| \\ 
		& \leq 
        \f{C}{\ep^2} \int_0^{\tau}  \!\!\!\int_{\Om}\left|  \lr{ P(\vr_{\ep})-P( r_{\ed}) -  P'( r_{\ed})(\vr_{\ep}- r_{\ed}) } \right|  \frac{|\Delta b| }{\underline{b}} \dx\ds \\
		& \quad 
        + \ep \f{C}{\ep^2} \int_0^{\tau}  \!\!\!\int_{\Om} \left|  \lr{ P(\vr_{\ep})-P( r_{\ed}) -  P'( r_{\ed})(\vr_{\ep}- r_{\ed}) } \right|  \frac{  |\Delta   s_{\ed}|}{ \underline{b}}  \dx\ds \\
		& \leq 
        {C(\underline{b})} \int_0^{\tau} 
        \| \Delta b \|_{L^\infty_x } \mathscr{E}( \vr_{\ep},\vv_{\ep},\vw_{\ep}| r_{\ed},\vV_{\ed},\vW_{\ep})(s) \ds  \\
		& \quad 
        +  C(\underline{b}) \int_0^{\tau}   \ep \|\Delta   s_{\ed} \|_{L^\infty_x} \mathscr{E}( \vr_{\ep},\vv_{\ep},\vw_{\ep}| r_{\ed},\vV_{\ed},\vW_{\ep})(s) \ds \\
		& \leq 
        \int_0^{\tau} 
       \Big( \chi(s) + \chi_{\ep}(\delta, s)  \Big) \,
        \mathscr{E}( \vr_{\ep},\vv_{\ep},\vw_{\ep}| r_{\ed},\vV_{\ed},\vW_{\ep})(s) \ds.
	\end{align*}
	The estimate of $\mathcal{I}_{12,2}$ is similar. In summary, 
	\begin{align}\label{I12}
			|\mathcal{I}_{12}| 
            \leq  
            \int_0^{\tau} 
       \Big( \chi(s) + \chi_{\ep}(\delta, s)  \Big) \,
        \mathscr{E}( \vr_{\ep},\vv_{\ep},\vw_{\ep}| r_{\ed},\vV_{\ed},\vW_{\ep})(s) \ds.
	\end{align}

\paragraph{Estimate of term $\mathcal{I}_{13}$:} It remains to treat 
    \begin{align*}
        \mathcal{I}_{13}
        =
        \f{2\kappa {\mu}}{\ep^2} \int_0^{\tau}  \!\!\!\int_{\Om} \vr_{\ep} \lr{ p'( r_{\ed}) \Grad \log r_{\ed} - p'( \vr_{\ep}) \Grad \log \vr_{\ep}}\cdot \left( \nabla \log r_{\ed} -\nabla \log b \right)           \dx\ds  .
    \end{align*}
    At first, we notice that 
    \begin{align}\label{i13-1}
         \nabla \log r_{\ed} -\nabla \log b = \frac{b\nabla r_{\ed}- r_\ed \nabla b}{\; r_{\ed} \; b}= \ep \frac{b \nabla s_{\ed} - s_{\ed} \nabla b }{ \; r_{\ed}\; b}.
    \end{align}
    Then, we observe the following identity 
    \begin{align*}
         &\vr_{\ep} \lr{ p'( r_{\ed}) \Grad \log r_{\ed} - p'( \vr_{\ep}) \Grad \log \vr_{\ep}}\cdot \left( \nabla \log r_{\ed} -\nabla \log b \right)     \\
         & \quad =\vr_{\ep} p'( \vr_{\ep}) \lr{ \Grad \log r_{\ed}- \Grad \log \vr_{\ep}}  \cdot \left( \nabla \log r_{\ed} -\nabla \log b \right) \\
         & \quad \quad + \vr_{\ep} \lr{ p'( r_{\ed}) - p'( \vr_{\ep}) } \Grad \log r_{\ed} \cdot \left( \nabla \log r_{\ed} -\nabla \log b \right) .
    \end{align*}
    The second term can be estimated similarly to $\mathcal{I}_{12}$. More precisely,  
    \begin{align*}
       & \vr_{\ep} \lr{ p'( r_{\ed}) - p'( \vr_{\ep}) } \Grad \log r_{\ed} \cdot \left( \nabla \log r_{\ed} -\nabla \log b \right) \\
        & \quad = - \left[ \vr_\ep p'( \vr_{\ep})-  r_\ed p'( r_{\ed})-
        \p_{r_{\ed} }\lr{ r_\ed p'( r_{\ed})} (\vr_\ep-r_\ed )\right] \Grad \log r_{\ed} \cdot \left( \nabla \log r_{\ed} -\nabla \log b \right)\\
        &\quad \quad - \left[ (\vr_\ep-r_\ed) r_{\ed} p^{\prime\prime}(r_\ed)  \right ] \Grad \log r_{\ed} \cdot \left( \nabla \log r_{\ed} -\nabla \log b \right) \\
        & \quad =- \gamma (\gamma-1) \left[ P( \vr_{\ep})-  P( r_{\ed})-P^\prime (r_\ed) (\vr_\ep-r_\ed )\right] \Grad \log r_{\ed} \cdot \left( \nabla \log r_{\ed} -\nabla \log b \right)\\
        & \quad \quad- \left[ (\vr_\ep-r_\ed)  \right ] \Grad p^\prime( r_{\ed}) \cdot \left( \nabla \log r_{\ed}  -\nabla \log b \right) .
    \end{align*}
    Hence, we rewrite 
    \begin{align*}
        \mathcal{I}_{13} &=
        \f{2\kappa {\mu}}{\ep^2} \int_0^{\tau}  \!\!\!\int_{\Om} \vr_{\ep} p'( \vr_{\ep}) \lr{ \Grad \log r_{\ed}- \Grad \log \vr_{\ep}}  \cdot \left( \nabla \log r_{\ed} -\nabla \log b \right)         \dx\ds \\
        & \quad 
        - \gamma (\gamma-1) \f{2\kappa {\mu}}{\ep^2} \int_0^{\tau}  \!\!\!\int_{\Om} 
        \Big(
        P( \vr_{\ep})-  P( r_{\ed})-P^\prime (r_\ed) (\vr_\ep-r_\ed )
        \Big)  
        \Grad \log r_{\ed} \cdot \left( \nabla \log r_{\ed} -\nabla \log b \right) \dx\ds \\
        & \quad 
        - \f{2\kappa {\mu}}{\ep^2} \int_0^{\tau} 
        \!\!\!\int_{\Om} (\vr_\ep-r_\ed)   
        \Grad p^\prime( r_{\ed}) \cdot \left( \nabla \log r_{\ed}  -\nabla \log b \right) \dx \ds \\
       &  =: 
       \mathcal{I}_{13,1}+\mathcal{I}_{13,2}+\mathcal{I}_{13,3}.
    \end{align*}
    For the terms $\mathcal{I}_{13,2}$ and $\mathcal{I}_{13,3}$, we use the dispersive estimate similarly to what done for getting  \eqref{I12}, and obtain 
\begin{align}\label{I13-2,3}
			|\mathcal{I}_{13,2}|+ |\mathcal{I}_{13,3}|
            \leq   {\mu}
            \int_0^{\tau} 
       \Big( \chi(s) + \chi_{\ep}(\delta, s)  \Big) \,
        \mathscr{E}( \vr_{\ep},\vv_{\ep},\vw_{\ep}| r_{\ed},\vV_{\ed},\vW_{\ep})(s) \ds .
	\end{align}
Regarding $\mathcal{I}_{13,1}$, for any $\lambda>0$ we have  
\begin{align*}
        &{|\mathcal{I}_{13,1}|}\leq 
        \lambda
        \f{2\kappa {\mu}}{\ep^2} \int_0^{\tau}  \!\!\!\int_{\Om} \vr_{\ep} p'( \vr_{\ep}) |\lr{ \Grad \log r_{\ed}- \Grad \log \vr_{\ep}} |^2      \dx\ds
        + \f{1}{4\lambda}
        \f{2\kappa {\mu}}{\ep^2} \int_0^{\tau}  \!\!\!\int_{\Om} \vr_{\ep} p'( \vr_{\ep}) | \left( \nabla \log r_{\ed} -\nabla \log b \right)  |^2       \dx\ds.
    \end{align*}
    In particular, for $\lambda= \kappa<1$  the first term can be absorbed by the left-hand side of the relative entropy inequality \eqref{rel_entr13}. 
    We next see from \eqref{i13-1} that 
    \begin{align*}
        &\f{1}{4\lambda}
        \f{2\kappa {\mu}}{\ep^2} \int_0^{\tau}  \!\!\!\int_{\Om} \vr_{\ep} p'( \vr_{\ep}) | \left( \nabla \log r_{\ed} -\nabla \log b \right)  |^2       \dx\ds \\
       & \quad 
       \leq 
       \f{2}{4\lambda}
       2\kappa {\mu}
       \int_0^{\tau}  \!\!\!\int_{\Om} \vr_{\ep} p'( \vr_{\ep}) \lr{b^2 |\nabla s_{\ed}|^2+|s_\ed|^2|\nabla  b|^2} \frac{1}{r_{\ed}^2\; b^2}      \dx\ds.
    \end{align*}
    The difficulty is {the control of the term depending on $|\nabla s_{\ed}|^2$ on the essential set, for which the dispersive estimate \eqref{acc-wave-dis}
    is out of use. However, thanks to \eqref{acc-wave-reg}, we have
    \begin{align*}
        \f{1}{4\lambda}
        2\kappa \mu
       \int_0^{\tau}  \!\!\!\int_{\Om} [\vr_{\ep} p'( \vr_{\ep})]_{\text{ess}} |\nabla s_{\ed}|^2        \dx\ds \leq 2 \f{C_\delta \k}{4\lambda} \mu \t,
    \end{align*}
which tends to zero when $\mu=\mu_\ep\to0$, and contributes to the final 
estimates as a remainder of the form $\int^\t_0\chi_\ep(\de,s) \ds$.
}
   For the residual part, we have instead
   \begin{align*}
    \int_0^{\tau}  \!\!\! \int_{\Om} [\vr_{\ep} p'( \vr_{\ep})]_{\text{res}} |\nabla s_{\ed}|^2        \dx\ds 
    \leq
       \int_0^{\tau}  \!\!\!\int_{\Om} (1+ \vr_{{\ep}}^\gamma)_{\text{res}} |\nabla s_{\ed}|^2        \dx\ds\leq \int_0^\tau 
		\chi_{\ep} 
		(\delta,s)  \ds, 
   \end{align*}
   where we used \eqref{bound:rhou} and \eqref{acc-wave-dis} to conclude.
   Recalling the fact that $\nabla b$ has compact support, we obtain 
   \begin{align*}
       &\int_0^{\tau}  \!\!\!\int_{\Om} \vr_{\ep} p'( \vr_{\ep})  \frac{|s_\ed|^2|\nabla  b|^2}{r_{\ed}^2\; b^2}      \dx\ds\\
       & \quad 
       \leq \int_0^{\tau}  \!\!\!\int_{\Om} [\vr_{\ep} p'( \vr_{\ep})]_{\text{ess}}  \frac{|s_\ed|^2|\nabla  b|^2}{r_{\ed}^2\; b^2}      \dx\ds+\int_0^{\tau}  \!\!\!\int_{\Om} [\vr_{\ep} p'( \vr_{\ep})]_{\text{res}}  \frac{|s_\ed|^2|\nabla  b|^2}{r_{\ed}^2\; b^2}      \dx\ds \\
       & \quad 
       \leq C(\delta, \Ov{b}, \underline{b},\|\Grad b\|_{L^\infty_x}) \int_0^\tau 
       \|s_{\ed} \|_{L^\infty_x}
        \|s_{\ed} \|_{L^2_x}  
        \| \nabla b \|_{L^2_x} \ds 
        +
        \ep^2
        C ( \Ov{b}, \underline{b},\|\Grad b\|_{L^\infty_x} ) 
        \int_0^{\tau}  \!\!\!\int_{\Om}
        \f{ (1+ \vr_{{\ep}}^\gamma)_{\text{res}} }
        {   \ep^2   }
        | s_{\ed}|^2        \dx\ds. 
   \end{align*}
   Hence, in view of $\lambda=\kappa$, we obtain 
   \begin{align*} 
			{|\mathcal{I}_{13,1}|}&
            \leq  \f{2\kappa^2 {\mu}}{\ep^2} \int_0^{\tau}  \!\!\!\int_{\Om} \vr_{\ep} p'( \vr_{\ep}) |\lr{ \Grad \log r_{\ed}- \Grad \log \vr_{\ep}} |^2      \dx\ds \\
            &\quad+  
            \int_0^\tau 		\chi_{\ep} (\delta,s)  \ds 
        + 
        \int_0^{\tau}  
        \left(\chi(s) +\chi_{\ep}(\delta, s) \right) \,
        \mathscr{E}( \vr_{\ep},\vv_{\ep},\vw_{\ep}| r_{\ed},\vV_{\ed},\vW_{\ep})(s) \ds
        .
	\end{align*}
   In summary,  
   \begin{align*} 
			|\mathcal{I}_{13}| &\leq  \f{2\kappa^2 {\mu}}{\ep^2} \int_0^{\tau}  \!\!\!\int_{\Om} \vr_{\ep} p'( \vr_{\ep}) |\lr{ \Grad \log r_{\ed}- \Grad \log \vr_{\ep}} |^2      \dx\ds \\
            &\quad+   \int_0^\tau 
		\chi_{\ep} 
		(\delta,s)  \ds 
        + 
         \int_0^{\tau}  
        \left(\chi(s) +\chi_{\ep}(\delta, s) \right) \,
        \mathscr{E}( \vr_{\ep},\vv_{\ep},\vw_{\ep}| r_{\ed},\vV_{\ed},\vW_{\ep})(s) \ds .
	\end{align*}

\paragraph{Estimate of term $\mathcal I_{14}$:}
{By use of the Cauchy-Schwarz and Young inequalities, we can simply bound}
\begin{align*}
    |\mathcal I_{14}|
    &\lesssim {\mu}
    \int_0^\tau\!\!\!\int_\Om
    \left|\Ov{\sqrt{\vr_\ep}\bd\Big(
    \sqrt{1-\kappa}(\vv_\ep-\vV_{\ed})
    -\sqrt{\kappa}(\vw_\ep-\vW_{\ep})\Big)}\right|^2\dx\ds
+{\mu}
    \int_0^\tau\!\!\!\int_\Om
    \vr_\ep|\Grad^2\Phi_{\ed}|^2\dx\ds.
\end{align*}
By \eqref{bound:rhou} and \eqref{acc-wave-reg},
\begin{align*}
    \int_\Om\vr_\ep|\Grad^2\Phi_{\ed}|^2\dx
    &\le C\|\Grad^2\Phi_{\ed}\|_{L^2_x}^2
    +C\|[1+\vr_\ep^\gamma]_{\text{res}}\|_{L^1_x}
    \|\Grad^2\Phi_{\ed}\|_{L^\infty_x}^2
    \le C_\delta.
\end{align*}
{ Consequently, treating this term  as a remainder of the form  $\int^\t_0\chi_\ep(\de,s) \ds$ as done in the estimate of $\mc I_{13}$, we find}
\begin{align}\label{I14}
    |\mathcal I_{14}|
    &\le{\mu}\int_0^\tau\!\!\!\int_\Om
    \left|\Ov{\sqrt{\vr_\ep}\bd\Big(
    \sqrt{1-\kappa}(\vv_\ep-\vV_{\ed})
    -\sqrt{\kappa}(\vw_\ep-\vW_{\ep})\Big)}\right|^2\dx\ds
    +\int^\t_0\chi_\ep(\de,s) \ds. 
\end{align}
{The first term is half of the symmetric dissipation $\int^\t_0\mc D_\ep \ds$, so it can be absorbed into the left-hand side of \eqref{rel_entr13}.}

\paragraph{Estimate of term $\mathcal I_{15}$:}
Applying Young's inequality, we conclude that
\begin{align}\label{I15}
    |\mathcal I_{15}|
    &\le\int_0^\tau\!\!\!\int_\Om
    \vr_\ep|\vv_\ep-\vV_{\ed}|^2\dx\ds
    +\mu^{2}\int_0^\tau\!\!\!\int_\Om
    \vr_\ep\left|b^{-1}\Div(b\bd\vU)\right|^2\dx\ds
    \nonumber\\
    &\le 2\int_0^\tau
    \mathscr E(\vr_\ep,\vv_\ep,\vw_\ep|
    r_{\ed},\vV_{\ed},\vW_{\ep})(s)\ds
    +C{\mu}^{2}
    \nonumber\\
    &\le\int_0^\tau\chi(s)
    \mathscr E(\vr_\ep,\vv_\ep,\vw_\ep|
    r_{\ed},\vV_{\ed},\vW_{\ep})(s)\ds
    +\int_0^\tau\chi_\ep(\delta,s)\ds,
\end{align}
where we used $b^{-1}\Div(b\bd\vU)\in L^\infty_t(L^2_x\cap L^6_x)$ and \eqref{bound:rhou}. 
All the remainder functions above tend to zero in $L^1(0,T)$ for
fixed $\delta$, since $\mu=\mu_\ep\to0$ and $\ep\to0$.
    

	\paragraph{The Gr\"{o}nwall argument.}  
	Collecting the estimates \eqref{i1-i2}, \eqref{I3}, \eqref{I4-I5}, \eqref{I6}, \eqref{I7}, \eqref{I8}, \eqref{I9-11} and \eqref{I12}, we finally arrive at inequality \eqref{ill:out:gron}. From there, following the strategy described in Section \ref{subs:ill-p}, we can deduce
\begin{align}\label{ill-rel-0}
		\lim\limits_{\delta \rightarrow 0 }
		\limsup\limits_{\ep \rightarrow 0 } 
		\mathop{\rm ess\,sup}_{0<\tau<T}\,\mathscr{E}( \vr_\ep,\vv_\ep,\vw_\ep| r_{\ep,\delta},\vV_{\ep,\delta},\vW_{\ep})(\tau)=0 ,
	\end{align}
for $0<{T}<T_{\ast}$, provided that $\mathscr{E}( \vr_{0,\ep},\vv_{0,\ep},\vw_{0,\ep}| r_{0,\ep,\delta},\vV_{0,\ep,\delta},\vW_{\ep})$ is sufficiently small.

To this end, we observe that 
	\begin{align*}
		& \mathscr{E}( \vr_{0,\ep},\vv_{0,\ep},\vw_{0,\ep}| r_{0,\ep,\delta},\vV_{0,\ep,\delta},\vW_\ep)  \\
		& \quad 
		= 
		\f{1}{2} \int_{\Om}  \vr_{0,\ep} \Big(|\vv_{0,\ep}-\vV_{0,\ep,\delta}|^2+|\vw_{0,\ep}-\vW_{\ep}|^2  \Big)     \dx  \nonumber \\
		& \qquad +\f{1}{\ep^2} \int_{\Om}   \Big( P(\vr_{0,\ep})-P(r_{0,\ep,\delta})-P'(r_{0,\ep,\delta})(\vr_{0,\ep}-r_{0,\ep,\delta})  \Big)     \dx \\
		& \quad 
		=\f{1}{2} \int_{\Om}  \vr_{0,\ep}  
		\left| 
		\vu_{0,\ep}-\f{1}{b} \mathbb{P}_b[b \vu_0] -\Grad \Phi_{0,\ed}
		+ 2\kappa {\mu} \Grad \log \vr_{0,\ep}
		- 2\kappa {\mu} \Grad \log b
		\right|^2
		\dx \\
		& \qquad +
		\f{1}{2} \int_{\Om}  \vr_{0,\ep}  
		\left|
		2 \sqrt{ \kappa (1-\kappa)}
		{\mu} \Grad \log \vr_{0,\ep}
		- 
		2 \sqrt{ \kappa (1-\kappa)}
		{\mu} \Grad \log b
		\right|^2 \dx \\
		& \qquad +\f{1}{\ep^2} \int_{\Om}  
		\Big( 
		P(\vr_{0,\ep})-P(b+\ep s_{0,\ep,\delta})-P'(b+\ep s_{0,\ep,\delta})(\vr_{0,\ep}-b-\ep s_{0,\ep,\delta})   
		\Big)     \dx,
	\end{align*}
which implies
	\begin{align*}
		\lim\limits_{\delta \rightarrow 0 }
		\limsup\limits_{\ep \rightarrow 0 } \mathscr{E}( \vr_{0,\ep},\vv_{0,\ep},\vw_{0,\ep}| r_{0,\ep,\delta},\vV_{0,\ep,\delta},\vW_{\ep}) 
		=0.
	\end{align*}

\subsection{Conclusion of the proof of Theorem \ref{TH2}}
Our ultimate goal is to establish the following strong convergence as $\ep \rightarrow 0$: 
	\begin{align*}
		\sqrt{\vr_\ep} \vv_\ep \rightarrow \sqrt{b} \vV  
        \quad 
        \text{ strongly in } 
        \quad 
        L^2(0,T;L^2_{\text{loc}} (\Omega;\R^3)) .
	\end{align*}
	Let $K$ be a compact subset of $\R^3$. It holds  
	\begin{align*}
		&\int_0^T \int_K |\sqrt{\vr}_\ep \vv_\ep -\sqrt{b} \vV|^2 \dx \dt \\
		& \quad \leq 2 \int_0^T \int_K  \vr_\ep  |\vv_\ep - \vV|^2 \dx \dt + 2 \int_0^T \int_K |\sqrt{\vr}_\ep  -\sqrt{b}|^2 |\vV|^2 \dx \dt \\ 
		&\quad  \leq 4 \int_0^T \int_K  \vr_\ep  |\vv_\ep - \vV_{\ep,\delta}|^2 \dx \dt + 4 \int_0^T \int_K  \vr_\ep  | \nabla \Phi_{\ep,\delta}|^2 \dx \dt + 2\int_0^T \int_K |\sqrt{\vr}_\ep  -\sqrt{b}|^2 |\vV|^2 \dx \dt \\ 
		&\quad =:R_1+R_2+R_3 . 
	\end{align*}
    The right-hand side is estimated as follows. 
	\begin{itemize}
		\item Returning to \eqref{ill-rel-0}, we see that for any $\tau \in(0,T)$, we have 
		\begin{align*}
			\mathscr{E}( \vr_\ep,\vv_{\ep},\vw_{\ep}| r_{\ed},\vV_{\ed},\vW_{\ep})(\tau) =R_\delta(\tau) \quad 
			\text{ as } \quad  \ep \rightarrow 0,
		\end{align*}
        where  $ R_\delta(\tau)$ satisfies
        \begin{align*}
           R_\delta(\tau) \rightarrow 0  \quad \text{ as }  \delta \rightarrow 0   \text{ for }  \tau \in (0,T). 
        \end{align*}
		which implies  
		\begin{align*}
			\lim\limits_{\ep \rightarrow 0 } R_1 \leq \int_0^T R_\delta(s) \ds.
		\end{align*}
		\item For $R_2$, we observe that $\displaystyle \vr_\ep = b+ \ep \vr^{(1)}_\ep$ with $b\in L^\infty(\R^3)$ (this eventually gives $b\in L^p_{\text{loc}} (\R^3)$ for all $p\geq 1$) and $\vr^{(1)}_\ep $ is uniformly bounded in $L^\infty(0,T; L^2 +L^\gamma (\R^3))$ (this yields $\vr^{(1)}_\ep $ is uniformly bounded in $L^\infty(0,T; L^{\tilde{\gamma}} (K))$ with $\tilde{\gamma}=\min\{2,\gamma\}$). Thus, we have 
		\begin{align*}
			|R_2|
			&
			\leq \int_0^T || b ||_{L^2(K)} ||\nabla \Phi_{\ep,\delta} ||_{L^4(K)}^2 \dt + \ep \int_0^T || \vr^{(1)}_\ep ||_{L^{\tilde\gamma}(K)} ||\nabla \Phi_{\ep,\delta} ||_{L^{2\tilde{\gamma}^\prime}(K)}^2 \dt \\
			&
			\leq
			|| b ||_{L^2(K)} \int_0^T  ||\Phi_{\ep,\delta} ||_{W^{1,4}(K)}^2 \dt +  \ep || \vr^{(1)}_\ep ||_{L^\infty(0,T;L^{\tilde\gamma}(K))} \int_0^T || \Phi_{\ep,\delta} ||_{W^{1,2\tilde{\gamma}^\prime}(K)}^2 \dt .
		\end{align*}
		Next, we employ the dispersive estimate \eqref{acc-wave-dis} and \eqref{acc-wave-reg} to get
		\begin{align*}
			\begin{split}
				\lim\limits_{\ep \rightarrow 0 } R_2 =0.
			\end{split}
		\end{align*}
		\item For $R_3$, we use the inequality 
		\begin{align*}
			(\sqrt{a}-\sqrt{b})^2 \leq |a-b|  \quad \text{ for all }a,b\geq 0
		\end{align*}
		to get 
		\begin{align*}
			|R_3| \leq \ep \int_0^T \int_K |\vr^{(1)}_\ep| |\vV|^2 \dx \dt.
		\end{align*}
		Since $\vV\in L^\infty(0,T;L^2 \cap L^\infty(\Omega;\R^3))$, using the uniform bounds on $\vr^{(1)}_\ep$, we have 
		\begin{align*}
			\begin{split}
				\lim\limits_{\ep \rightarrow 0 } R_3 =0.
			\end{split}
		\end{align*}
	\end{itemize}
	Thus combining the estimates for $R_1$, $R_2 $ and $R_3$, 
	\begin{align*}
		\begin{split}
			\lim\limits_{\ep \rightarrow 0 } \int_0^T \int_K |\sqrt{\vr}_\ep \vv_\ep -\sqrt{b} \vV|^2 \dx \dt  
            \leq 
            \int_0^T R_\delta(s) \ds.
		\end{split}
	\end{align*}
It follows that 
\begin{align*}
    &\lim\limits_{\ep \rightarrow 0 } \int_0^T \int_K |\sqrt{\vr}_\ep \vv_\ep -\sqrt{b} \vV|^2 \dx \dt =\lim\limits_{\delta \rightarrow0}\lim\limits_{\ep \rightarrow 0 } \int_0^T \int_K |\sqrt{\vr}_\ep \vv_\ep -\sqrt{b} \vV|^2 \dx \dt
    \leq 
    \lim\limits_{\delta \rightarrow0}\int_0^T R_\delta(s) \ds=0.
\end{align*}

Exactly in the same manner, one verifies as $\ep \rightarrow 0$ 
\begin{align*}
		\sqrt{\vr_\ep} \vw_\ep \rightarrow {\bf{0}} 
        \quad 
        \text{ strongly in } 
        \quad 
        L^2(0,T;L^2_{\text{loc}} (\Omega;\R^3)) .
	\end{align*}
This completes the proof of Theorem \ref{TH2}.      

\section{Concluding remarks}\label{sec:6}
In this paper, we rigorously justified the convergence from the global $\kappa$-entropy weak solution of the degenerate compressible Navier--Stokes system \eqref{dege-ns} towards the local strong solution of the generalized anelastic approximation \eqref{ane-app}. 
Such an asymptotic limit has been investigated previously by imposing various regularizing terms. 
The main achievement of this paper is to remove these additional regularizing terms in the system. 

To this end, we appeal to the $\kappa$-entropy solution, adapted to the two-velocity formulation, introduced by Bresch et al. \cite{BDZ15}, and make full use of the relative entropy method proposed in Bresch et al. \cite{BNV15}. 
In this way, the low-Mach number and low-Froude number limits were performed for well-prepared initial data in the periodic domain $\mathbb{T}^3$ and for ill-prepared initial data in the whole space $\mathbb{R}^3$. 

In the case of a bounded smooth domain, a natural further direction would be  to extend our results to other physically relevant boundary conditions, such as homogeneous Dirichlet or Navier boundary conditions for the velocity field. Some existence results for compressible Navier--Stokes systems with density-dependent viscosities in bounded domains are available, for instance in \cite{BDGV}. However, they do not provide the global-in-time weak existence theory for the primitive system \eqref{dege-ns} needed here. Consequently, the corresponding singular limit remains beyond the scope of the present paper.

We emphasize that the crucial step in the proof of Theorem \ref{TH2} is an application of dispersive estimate for acoustic waves. To simplify the presentation, we assumed that the external force $G$ has compact support. 
In view of the work of Donatelli and Feireisl \cite{DF-17}, it is indeed possible to relax $G$ by keeping the smoothness and suitable spatial decay rates. Indeed, to keep the dispersive estimates of acoustic waves, it suffices to make the following assumptions of $G$:
\begin{align*}
& G\in C^{\infty}(\R^3), \quad G(x)>0 \quad \text{ for all }x \in \R^3, \\
&    \underline{G}  \f{1}{|x|} \leq  G(x) \leq \overline{G} \f{1}{|x|} \quad \text{ for all } |x|>C>0, \\
& |x|^2 |\Grad G(x)|+  |x|^3 |\Grad^2 G(x)| \leq C \quad \text{ for all }x \in \R^3. 
\end{align*}
The above constraints show that $G$ shares the asymptotic properties with the standard gravitational potential, see \cite{DF-17} for more physical discussions.

\appendix

\section*{Appendix A: details of the construction of $\kappa$-entropy solutions} \label{app:k-entropy}

In this appendix, we want to give details about the construction of $\kappa$-entropy solutions, as defined in Subsection \ref{ss:weak}, 
thus justifying some claims appearing in the series of remarks following Definitions \ref{def-1} and \ref{def-1a}.
Although nowadays the discussed issues are well understood, our analysis may be very useful
for readers not familiar with the degenerate Navier--Stokes system \eqref{dege-ns}.

More precisely, we are going to show:
\begin{itemize}
 \item the validity of the $\k$-entropy inequality \eqref{energy-inequa};
 \item the validity of the conditions \eqref{def2:prop};
 \item the validity of the relation $\sqrt{\vr}\vw = 
 2 \mu \sqrt{\k (1-\k)} 
 \sqrt{\vr}\Grad\log\vr$.
\end{itemize}

\medbreak
We start by pointing out that $\k$-entropy solutions $\big(\vr,\vu\big)$ to system \eqref{dege-ns} arise as limits of suitable smoother
approximate solutions. As a matter of fact, the process of construction (see \cite{VaYu16, li-xin, BVY21}) involves several approximation parameters and one
has to pass to the limit with respect to them, in a precise order.

For the sake of the present discussion, let us consider only two levels of approximation and assume that we have a family of regular, approximate
solutions $\big\{\big(\vr_{n,\nu} , \vu_{n,\nu}\big)\big\}_{n,\nu}$ depending on the two parameters $n\in\N$ and $\nu\in(0,1)$.
For any fixed value of these parameters, we assume to have $\vr_{n,\nu}>0$, so that the quantity $\vu_{n,\nu}$ is well-defined together with its gradient.
We assume that the $\k$-entropy inequality \eqref{energy-inequa} holds true, yielding corresponding uniform bounds for such a family
$\big\{\big(\vr_{n,\nu} , \vu_{n,\nu}\big)\big\}_{n,\nu}$.
Finally, we assume that, in the limit $\nu\to0$, the functions $\vr_{n,\nu}$ converge strongly, in a suitable topology,
to some $\vr_n$ such that $\vr_n\geq0$ only. 

Thanks to the approximate version of \eqref{energy-inequa} and the relations linking $\vu_{n,\nu}$, $\vv_{n,\nu}$ and $\vw_{n,\nu}$,
we deduce that the sequence (with respect to $\nu$) of the $\sqrt{\vr_{n,\nu}} \bd\vu_{n\nu}$'s is uniformly bounded in $L^2_tL^2_x$.
However, as $\vr_n\geq0$, in the limit passage $\nu\to0$, we can only say that
$\vr_{n,\veps}\bd\vu_{n,\veps}\,\rightharpoonup\,\oline{\vr_n \bd\vu_n}$ in the weak topology of  $L^2_tL^2_x$.
Likewise, as $\vr_{n,\nu}>0$, we have, up to a suitable subsequence, the following weak convergences in $L^2_tL^2_x$:
	\begin{align*}
			\sqrt{\vr_{n,\nu}} \, \ba \vv_{n,\nu} = \sqrt{\vr_{n,\nu}} \, \ba \vu_{n,\nu}   \rightharpoonup
			\Ov{\sqrt{\vr_n} \,\ba  \vv_n } = \Ov{\sqrt{\vr_n} \, \ba  \vu_n } , \quad 
            \sqrt{\vr_{n,\nu} } \, \ba \vw_{n,\nu} =\vc{0} \rightharpoonup \oline{\sqrt{\vr_n} \, \ba\vw_n} = \vc{0},
		\end{align*}
together with
\begin{align*}
& \sqrt{\vr_{n,\nu}} \, [\bd (\sqrt{1-\kappa}\vv_{n,\nu} ) -\Grad(\sqrt{\kappa}\vw_{n,\nu})] =\sqrt{1-\kappa} \sqrt{\vr_{n,\nu}}  \, \bd \vu_{n,\nu} \\
&\qquad\qquad\qquad\qquad\qquad\qquad\qquad\qquad
\rightharpoonup \Ov{\sqrt{\vr_n} \, [\bd (\sqrt{1-\kappa}\vv_n ) -\Grad(\sqrt{\kappa}\vw_n)]}=\sqrt{1-\kappa} \,\Ov{\sqrt{\vr_n} \, \bd \vu_n}.
\end{align*}
Owing to those weak convergences and weak lower semicontinuity of norms, one can pass to the limit $\nu\to0$ in the
approximate $\kappa$-entropy inequality and obtain \eqref{energy-inequa} at level $n$, for any $n\in\N$.

Observe that, in the limits above, the bar symbol is interpreted by the same way as in Remark \ref{Rem:2.1}.
Using equalities \eqref{weak-id1}, \eqref{weak-id2} of that Remark and the strong convergence $\vr_n\,\longrightarrow\,\vr$ in the limit $n\to+\infty$,
it is rather straightforward to see that the sequence $\big(\oline{\vr_n \bd\vu_n}\big)_n$ is weakly convergent in $L^2_tL^2_x$
to an element $\mbb U$, which precisely coincides with $\oline{\vr \bd\vu}$: up to a subsequence,
\[
 \mbox{when }\ n\to+\infty,\qquad\qquad \oline{\vr_n \bd\vu_n}\,\rightharpoonup\,\oline{\vr \bd \vu} \qquad \mbox{ in } \quad L^2_tL^2_x.
\]
The same observation obviously applies also to the other terms $\vr_n \bd\vv_n$, $\vr_n \ba\vv_n$, $\bd\vw_n$ and $\ba\vw_n$.
Precisely as before, thanks to 
those weak convergences and weak lower semicontinuity of the norms, one can pass to the limit $n\to\infty$ in the
approximate $\kappa$-entropy inequality \eqref{energy-inequa} at level $n$, thus obtaining \eqref{energy-inequa} for the target solution $\big(\vr,\vu\big)$,
as stated in Definition \ref{def-1}.

\medbreak
We point out that, in particular, the argument above shows that, as $\ba\vw_{n,\nu} = \vc{0}$ for any $\nu\in(0,1)$, then one also has that
\[
 \vr_{n,\nu} \ba\vw_{n,\nu} = \vc{0}\,\rightharpoonup\, \oline{\vr_n \ba\vw_n } = \vc{0}\qquad\qquad \mbox{ in }\quad L^2_tL^2_x\,,
\]
for any $n\in\N$. Thus, passing to the limit $n\to+\infty$ yields $\oline{\vr \ba\vw} = \vc{0}$. By the same token, one gets 
$\oline{\vr \bd\vw} = \oline{\vr \nabla \vw}$, as well as $\oline{\vr \ba\vv} = \oline{\vr \ba \vu}$.

In the end, we have shown that conditions \eqref{def2:prop} hold.

\medbreak
Observe that, at this stage, we cannot deduce that $\vw$ is a gradient. In fact, this quantity alone is not well-defined
and makes sense only when multiplied by $\sqrt{\vr}$.
However, let us resort again to the previous approximation procedure, based on taking the limit of  smooth approximate solutions
$\big\{(\vr_{n,\nu},\vu_{n,\nu})\big\}_{n,\nu}$ first for $\nu\to0$ and then for $n\to+\infty$.
In the first limit $\nu\to0$, the strong convergence of the density functions $\vr_{n,\nu}\longrightarrow\vr_n$
allows us to derive the following property, where the limit has to be understood in the sense of distributions:
\[
 \sqrt{\vr_{n,\nu}}\vw_{n.\nu} = 4 \mu \sqrt{\k (1-\k)}\Grad\sqrt{\vr_{n,\nu}} \ \stackrel{\mathcal{D}'}{\rightharpoonup} \ 
 4\mu \sqrt{\k (1-\k)}\Grad\sqrt{\vr_n} = 2 \mu \sqrt{\k (1-\k)} \sqrt{\vr_n} \Grad\log\vr_n\,.
\]
However, by definition of $\sqrt{\vr_n}\vw_n$, we know that 
\[
 \sqrt{\vr_{n,\nu}}\vw_{n.\nu}\,\stackrel{*}{\rightharpoonup}\,\sqrt{\vr_n}\vw_n\qquad\qquad \mbox{ in }\quad L^\infty_tL^2_x\,,
\]
from which we deduce that $\sqrt{\vr_n}\vw_n = 2 \mu \sqrt{\k (1-\k)} \sqrt{\vr_n} \Grad\log\vr_n$. Arguing precisely in the same way when taking the limit
$n\to+\infty$, in turn we are able to guarantee the fundamental equality
\[
 \sqrt{\vr}\vw = 2 \mu \sqrt{\k (1-\k)} \sqrt{\vr} \Grad\log\vr\,,
\]
which we have exploited several times in our study.

\vspace{5mm}

\centerline{\bf Acknowledgements}
\vspace{2mm}

The research of Y.L. is supported by National Natural Science Foundation of China (12571228), Natural Science Foundation of Anhui Province (2408085MA018). He sincerely thanks Professor Yongzhong Sun for continuous encouragement. 

The research of N.C. and of E.Z. was supported by the EPSRC Early Career Fellowship EP/V000586/1. N.C. also acknowledges the support of the “Excellence Initiative Research University (IDUB)” program at the University of Warsaw, and the NAWA ULAM grant BPN/SEL/2025/1/00005/U/00001.

F.F. has been partially supported by the project CRISIS (ANR-20-CE40-0020-01), operated by the French National Research Agency (ANR),
by the Basque Government through the BERC 2022-2025 program and by the Spanish State Research Agency through the BCAM Severo Ochoa excellence accreditation
CEX2021-001142. The author also acknowledges the support of the European Union through the COFUND program [HORIZON-MSCA-2022-COFUND-101126600-SmartBRAIN3].

The authors are grateful to the anonymous referees for providing many insightful comments and suggestions, which
significantly improved the presentation.

For the purpose of open access, the author(s) has applied a Creative Commons Attribution (CC BY) license to any Author Accepted Manuscript version arising  from this submission.

\vspace{5mm}

\centerline{\bf Conflict of interest}
\vspace{2mm}
On behalf of all authors, the corresponding author states that there is no conflict of interest.

\vspace{5mm}

\centerline{\bf Data Availability Statement}
\vspace{2mm}
Data sharing not applicable to this article as no datasets were generated or
analysed during the current study.


\end{document}